\documentclass{article}
\usepackage[margin=1.2in]{geometry}
\usepackage{setspace}
\usepackage{array}
\usepackage{multirow}
\usepackage{graphicx} % Required for inserting images
\usepackage{hyperref}
\usepackage{todonotes}
\usepackage{amsfonts}
\usepackage{amsmath, amsthm, amssymb}
\usepackage{geometry}
\usepackage{xr}
\usepackage{zref-xr}

\usepackage[normalem]{ulem}

\usepackage[authoryear]{natbib}

\usepackage[ruled]{algorithm2e}
\RestyleAlgo{ruled} 

\usepackage{nicefrac}       % compact symbols for 1/2, etc.
\usepackage{microtype}      % microtypography
\usepackage{textcomp}

\usepackage{comment}

\allowdisplaybreaks[1]

\definecolor{mygray}{gray}{0.8}

\newcommand{\numberthis}{\addtocounter{equation}{1}\tag{\theequation}}

\newcommand{\shubhada}[1]{{\color{cyan} #1}}

\newcommand{\lrset}[1]{\left\{ #1 \right\}}
\newcommand{\lrp}[1]{\left( #1 \right)}
\newcommand{\lrs}[1]{\left[ #1 \right]}

\newcommand{\abs}[1]{\left|{#1}\right|}

\newcommand{\E}[1]{\mathbb{E}\lrs{#1}}
\newcommand{\Exp}[2]{\mathbb{E}_{#1}\lrs{#2}}
\newcommand{\Prob}[1]{\mathbb{P}\lrs{#1}}

\newcommand{\R}{\mathbb{R}}
\newcommand{\N}{\mathbb{N}}

\newcommand{\calS}{{\mathcal S}}

\newcommand{\lcm}{\ell_{cM}}
\newcommand{\ucm}{u_{cM}}
\newcommand{\lcs}{\ell_{cs}}
\newcommand{\ucs}{u_{cs}}

\newcommand{\xkp}{x_{k+1}}
\newcommand{\xk}{x_k}

\newtheorem{theorem}{Theorem}[section]
\newtheorem{corollary}{Corollary}[theorem]
\newtheorem{lemma}[theorem]{Lemma}
\newtheorem{proposition}[theorem]{Proposition}

\newtheorem{example}{Example}[section]
\newtheorem{remark}{Remark}
\newtheorem{assmp}{Assumption}

\title{Concentration of General Stochastic Approximation Under Heavy-Tailed Markovian Noise}
\usepackage{authblk}

\usepackage{authblk}

\author[1]{Shubhada Agrawal}
\author[2]{Siva Theja Maguluri}
\author[3]{Martin Zubeldia}
\affil[1]{Indian Institute of Science, \texttt{shubhada@iisc.ac.in}}
\affil[2]{Georgia Institute of Technology, \texttt{siva.theja@gatech.edu}}
\affil[3]{University of Minnesota, \texttt{zubeldia@umn.edu}}

\date{}

\begin{document}
\maketitle

\begin{abstract} 
We establish maximal concentration bounds for the iterates generated by stochastic approximation algorithms with general step sizes, where the noise has a finite-state Markovian component plus a Martingale-difference component. When the Martingale-difference noise is bounded, we show that the tail of the error can be sub-Gaussian, sub-Weibull, or something lighter than any Pareto but heavier than any Weibull, depending on the step size sequence and on whether the random operator is almost surely contractive, almost surely non-expansive, or expansive with positive probability. Our analysis relies on a novel Lyapunov function involving the moment-generating function of the solution to a Poisson equation, together with an auxiliary projected algorithm. We complement the upper bounds with worst-case examples showing that qualitatively sharper bounds are impossible. We further study the case of unbounded Martingale-difference noise when the average operator is contractive, and the step sizes are of order $1/k$. In this setting, we show that if the random operator is almost surely non-expansive, then the error tail is at most three times heavier than the noise tail, whereas if the random operator is expansive with positive probability, then the error may have substantially heavier tails. These results are obtained through a novel black-box truncation argument that reduces the unbounded-noise setting to the bounded-noise case.
\end{abstract}

\setcounter{tocdepth}{2}
\tableofcontents

%\todo[inline]{Revisit abstract. Recall how Telgarsky's ball argument compares with ours (possibly for a later version).}

\section{Introduction}\label{sec:intro}
The stochastic approximation (SA) method, introduced by \citet{robbins1951stochastic}, is a fundamental framework for stochastic iterative algorithms and underlies many methods in modern optimization and machine learning \citep{kober2013reinforcement,silver2017mastering,jumper2021highly,ouyang2022training}. Formally, SA provides an approach to solving fixed-point equations of the form
\[
\bar{F}(x) = x,
\]
where $\bar{F}:\mathbb{R}^d \to \mathbb{R}^d$ is (possibly) nonlinear. When $\bar{F}$ can be evaluated exactly, the equation can be solved via the classical fixed-point iteration $x_{k+1} = \bar{F}(x_k)$, which converges to the unique fixed point under standard contraction assumptions \citep{banach1922operations}. %, namely when there exist a constant~$\gamma_c \in (0,1)$ and a norm~$\|\cdot\|_c$ such that
%\[
%\|\bar{F}(x_1) - \bar{F}(x_2)\|_c \le \gamma_c \|x_1 - x_2\|_c, \qquad \forall \, x_1, x_2 \in \mathbb{R}^d.
%\]
{In many practical settings, such as reinforcement learning or large-scale stochastic optimization, the operator $\bar{F}$ cannot be computed exactly due to limited information or computational constraints. In such cases, one may only have access to noisy approximations of $\bar{F}(x)$ given by $F(x,S) + Z$, where $S$ and $Z$ are random elements defined on probability spaces $\mathcal{S}$ and $\mathbb{R}^d$, respectively.  Therefore, a stochastic approximation recursion is used, which is of the form
\begin{align}\label{eq:stochastic-approximation}
x_{k+1} = x_k + \alpha_k \big(F(x_k, S_k) + Z_k - x_k\big),
\end{align}
where $\{S_k\}_{k\geq 0}$ and $\{Z_k\}_{k\ge 0}$ are random sequences on $\mathcal{S}$ and $\mathbb{R}^d$, respectively, $F:\mathbb{R}^d\times \mathcal{S}\to \mathbb{R}^d$, and $\{\alpha_k\}_{k\geq 0}$ is a sequence of positive stepsizes, typically $\alpha_k = \alpha/(k+h)^z$ with $\alpha,h>0$ and $z\in(0,1]$. In this work, we will assume that %$\mathcal S$ to be the state space of a finite-state Markov chain, with 
$\{S_k\}_{k\geq 0}$ is a finite-state Markov chain, and that $\{Z_k\}_{k\ge 0}$ is a Martingale-difference sequence. Note that this subsumes the special case of i.i.d.\ noise.}

The recursion given in Equation \eqref{eq:stochastic-approximation} subsumes several classical algorithms. When $\bar{F}(x) = -c\nabla J(x) + x$ for a smooth and strongly convex function $J:\mathbb{R}^d\to\mathbb{R}$ and some $c>0$, it coincides with stochastic gradient descent (SGD) \citep{bottou2018optimization,lan2020first}, which can be viewed as a contractive SA scheme \citep{ryu2016primer}. In reinforcement learning, algorithms such as $Q$-learning \citep{watkins1992q} and temporal-difference learning \citep{sutton1988learning} also take the form of Equation \eqref{eq:stochastic-approximation}, with $\bar{F}$ corresponding to the Bellman operator \citep{bertsekas1996neuro,bellman1957dynamic}. Understanding the convergence properties of the iterates $\{x_k\}_{k\geq 0}$ generated by Equation \eqref{eq:stochastic-approximation} is thus of fundamental importance.

Classical results on stochastic approximation focus on asymptotic convergence and almost sure stability \citep{robbins1951stochastic,tsitsiklis1994asynchronous,mokkadem2006convergence,borkar2009stochastic,kushner2012stochastic}. More recently, attention has shifted toward non-asymptotic or finite-sample analysis, which provides explicit bounds on $\|x_k - x^*\|_c$, where $x^*$ satisfies $\bar{F}(x^*) = x^*$ \citep{bhandari2018finite,srikant2019finite,chen2020finite}. Such results yield quantitative rates of convergence and are directly applicable in high-dimensional learning settings.

Depending on the mode of convergence considered, two classes of bounds are commonly studied. The first concerns mean-square error bounds, i.e., estimates on $\mathbb{E}[\|x_k - x^*\|_c^2]$, which have been extensively developed in the literature \citep{srikant2019finite,bhandari2018finite,chen2019finitesample,chen2020finite,wainwrightHDS}. The second concerns high-probability or tail bounds of the form
\[
\mathbb{P}(\|x_k - x^*\|_c \le \epsilon),
\]
which quantify both accuracy and confidence. High-probability guarantees are, however, significantly more difficult to establish. For instance, in the special case of the law of large numbers, viewed as an SA recursion with $\alpha_k = 1/(k+1)$, obtaining an $\mathcal{O}(1/k)$ mean-square rate is straightforward, whereas deriving exponential concentration inequalities such as Hoeffding’s \citep{hoeffding1994probability}, Chernoff’s \citep{chernoff1952measure}, or Bernstein’s \citep{bernstein1924modification} requires stronger probabilistic control. More recently, \cite{chen2025concentration} obtained such high-probability bounds under restrictive assumptions on the noise and step sizes, which we generalize in key directions that are essential for applications such as Reinforcement Learning algorithms with function approximation, and SGD with heavy-tailed noise. A detailed review of related results is provided in Section~\ref{subsec:literature}.

\subsection{Our contributions}
{Throughout, we assume that the dynamics are driven by finite-state Markovian noise, a setting that arises naturally in many applications such as Reinforcement Learning. In addition, the updates may contain an additive Martingale-difference noise term. In this paper, we develop maximal concentration bounds for: (1) exponentially stable  SA algorithms in the presence of bounded additive Martingale-difference noise, and (2) contractive SA algorithms with $\mathcal O(1/k)$ step sizes in the presence of arbitrary unbounded, identically-distributed, additive Martingale-difference noise. We analyze three regimes for the random operator $F$: almost surely contractive, almost surely non-expansive, and expansive with positive probability, under various step-size schedules. The main contributions of this work are summarized below. }

\begin{enumerate}
    \item \textbf{General SA with finite-state Markovian and additive bounded Martingale-difference noise.} We establish high-probability bounds for SA algorithms under diminishing step sizes of the form $\alpha_k=\alpha/(k+h)^z$, where $\alpha,h>0$ and $z\in(0,1]$. In particular, we show the following.
    \begin{itemize}
        \item When $z=1$, our bound is sub-Weibull in general, and becomes sub-Gaussian in the special case where the random operator $F$ is almost surely non-expansive. These bounds are (order-wise) the best possible, as they match the lower bound established in \cite{chen2025concentration} in the case of a special counter-example.
        \item When $z\in(0,1)$, our bound is either sub-Gaussian when the random operator $F$ is almost surely contractive, sub-Weibull when it is almost surely non-expansive, or something heavier than any Weibull but lighter than any Pareto when $F$ can be expansive with positive probability. Moreover, we show that it is inevitable to obtain such distributions in all cases. 
    \end{itemize}
    Our results provide bounds on the entire tail of the iterates as our step sizes do not depend on either the desired accuracy level or the probability tolerance level. To the best of our knowledge, these are the first concentration bounds for SA with Markovian noise and general step sizes. Prior results were limited to Martingale-difference noise and $\mathcal{O}(1/k)$ step sizes. See Table \ref{tab:summary} for a detailed summary of our results.
\begin{table}[ht!]
\setlength{\arrayrulewidth}{0.3mm}
\renewcommand{\arraystretch}{1.9}
\begin{center}
\begin{tabular}{|
    >{\centering\arraybackslash}m{1.5cm}||
    >{\centering\arraybackslash}m{2cm}|
    >{\centering\arraybackslash}m{2.5cm}|
    >{\centering\arraybackslash}m{3.5cm}|} 
 \hline
 & $D<0$ 
 & $D = 0$ 
 & $D>0$ \\  
 \hline
 $z=1$
 & $\tilde{\mathcal{O}}\left(\frac{\log(1/\delta)}{k}\right)$
 & $\tilde{\mathcal{O}}\left(\frac{\log(1/\delta)}{k}\right)$
 & $\tilde{\mathcal{O}}\left(\frac{\log(1/\delta)^{2\alpha D+1}}{k}\right)$ \\
 \hline
 $z\in (0,1)$
 & $\tilde{\mathcal{O}}\left(\frac{\log(1/\delta)}{k^z}\right)$
 & $\tilde{\mathcal{O}}\left(\frac{\log(1/\delta)^{1/z}}{k^z}\right)$
 & $\tilde{\mathcal{O}}\lrp{\frac{\exp\left(\log(1/\delta)^{(1-z)/z}\right)}{k^z}}$ \\ 
 \hline
\end{tabular}
\end{center}
\caption{Bounds on $\|x_k-x^*\|_c^2$ that hold with probability at least $1-\delta$, where $\tilde{\mathcal{O}}(\cdot)$ hides polylog($k$) factors. The parameter $z$ is the exponent in the step size $\alpha_k=\alpha/(k+h)^z$. The parameter $D$ is the sum of the Lipchitz constant of the operator $\bar{F}(\cdot)$ and a bound on the Markovian noise. The case where the random operator $F$ is a.s.\ contractive falls under $D<0$, a.s.\ non-expansive falls under $D=0$, and expansive with positive probability under $D>0$. %In $\tilde{\mathcal{O}}(\cdot)$ and $\tilde{\Omega}(\cdot)$, logarithmic factors are ignored. 
}
\label{tab:summary}
\end{table}

    \item \textbf{Contractive SA with finite-state Markovian and additive unbounded, identically distributed, Martingale-difference noise.} For contractive operators, we also consider additive Martingale-difference noise with an arbitrary unbounded distribution, assumed to be identically distributed across time. In this setting, we establish maximal concentration bounds for the SA algorithms with $\mathcal{O}(1/k)$ step sizes. The resulting concentration bound has tails that are between two and three times as heavy as that of the underlying noise distribution when the random operator $F(\cdot, \cdot)$ is almost surely non-expansive, and arbitrarily heavier when the random operator is expansive with positive probability. To the best of our knowledge, such concentration results for general SA with unbounded noise are unknown in the literature, except in certain special cases such as SGD. See Table \ref{tab:summary2} for a detailed summary of our results.
\begin{table}[ht!]
\setlength{\arrayrulewidth}{0.3mm}
\renewcommand{\arraystretch}{1.9}
\begin{center}
\begin{tabular}{|
    >{\centering\arraybackslash}m{1.5cm}||
    >{\centering\arraybackslash}m{3cm}|
    >{\centering\arraybackslash}m{3cm}|
    >{\centering\arraybackslash}m{4.5cm}|} 
 \hline
 & $D<0$ 
 & $D = 0$ 
 & $D>0$ \\  
 \hline
 Bound
 & $\tilde{\mathcal{O}}\lrp{\frac{B^2(\delta)\log(1/\delta)}{k}}$
 & $\tilde{\mathcal{O}}\lrp{\frac{B^2(\delta)\log(1/\delta)}{k}}$
 & $\tilde{\mathcal{O}}\lrp{\frac{B^2(\delta)\big[B(\delta)\log(1/\delta)\big]^{2\alpha D + 1}}{k}}$ \\
 \hline
\end{tabular}
\end{center}
\caption{Bounds on $\|x_k-x^*\|_c^2$ that hold with probability at least $1-\delta$, where $\tilde{\mathcal{O}}(\cdot)$ hides polylog($k$) factors. The function $B(\cdot)$ describes the tail of the noise, and it is such that $\mathbb{P}\lrp{W^2 \le B(\delta)} \ge 1-\delta$, where $W$ is the unbounded additive part of the noise. The parameter $D$ is the sum of the Lipchitz constant of the operator $\bar{F}(\cdot)$ and a bound  on the Markovian noise. The case where the random operator $F$ is a.s. contractive falls under $D<0$, a.s. non-expansive falls under $D=0$, and expansive with positive probability under $D>0$. %In $\tilde{\mathcal{O}}(\cdot)$ and $\tilde{\Omega}(\cdot)$, logarithmic factors are ignored.
}
\label{tab:summary2}
\end{table}
    \item \textbf{Methodological Contributions.} To overcome the challenge of having Markovian noise, general step sizes, and potentially unbounded added noise, we develop new proof ideas to tackle each of these challenges. These new ideas involve bounding the moment-generating function (MGF) of the convergence error plus a correction term, % that depends on the solution to the Poisson equation of the underlying Markov chain of the noise, 
    which serves as a novel Lyapunov function in our analysis, and the construction of two key auxiliary coupled iterates, that are easier to analyze, in order to deal with general step sizes and with added unbounded noise. 
\end{enumerate}

\noindent We now provide more details about the challenges and our technical contributions.

\subsection{Challenges \& Our Techniques}\label{subsec:challenge}

Obtaining high-probability bounds in such generality presents several challenges that cannot be tackled by the proof techniques developed in the previous literature and require new ideas.\\

\textbf{Challenge $1$: General step sizes.} When the step sizes are of the order of $1/k^z$, for $z\in(0,1)$, and we have ``multiplicative'' noise, it is known that the error cannot be uniformly bounded by any sub-Weibull distribution (see Theorem 2.3 in \citet{chen2025concentration}). In this case, the error can grow exponentially fast with positive probability, and thus the ``bootstrapping'' proof technique from \citet{chen2025concentration} does not work. To overcome this, we define a coupled SA iteration such that, after each step of the SA recursion, the result is projected onto a ball of size $B$ centered at the fixed point $x^*$. Since these new iterates are uniformly bounded, the variance of the noise is also uniformly bounded, which prevents the error from growing exponentially fast in time. This allows us to obtain tight high-probability bounds for the projected iterates. We then show that, if the ball radius $B$ is sufficiently large, the projection step is inactive with high probability, and thus the projected iterates coincide with the original ones with high probability. Therefore, the high probability bounds for the former also apply to the latter.\\

\textbf{Challenge $2$: Added noise with general distribution.} When the added noise has a general distribution, the MGF of the error might not exist, and thus current proof techniques (such as the one in \citet{chen2025concentration} for sub-Gaussian noise) fail. To overcome this, we define new coupled iterates by truncating the original noise at an arbitrarily large value. This results in a sequence of iterates generated by an SA recursion with uniformly bounded noise, for which we already have high-probability bounds. In order to link the new iterates to the original ones, we take two steps:
\begin{enumerate}
    \item Since the truncation may introduce a bias in the noise, the expected truncated operator may have a different fixed point, to which the new iterates converge to. Thus, the first step is to bound this bias as a function of the truncation level, crucially exploiting the fact that the expected operator is a contraction (even with the added bias).
    \item The second step is to show that, for sufficiently large truncation levels, the noise is below these levels with high probability. Therefore, the original and the new iterates coincide with high probability, and the high-probability bounds for the former also apply to the latter.
\end{enumerate}

\textbf{Challenge $3$: Markovian noise.} When the noise is Markovian, the expected drift of the error $\mathbb{E}[\| x_{k+1}-x^*\|^2 - \| x_{k}-x^*\|^2 \mid \mathcal{F}_k]$ can be positive due to the bias introduced by the Markov chain. This means that the convergence is not monotone in expectation, as it is with Martingale-difference noise. This prevents us from using the same Lyapunov function as in \citet{chen2025concentration}, where the noise is assumed to be a Martingale-difference, and thus unbiased. Also, averaging techniques as used in \citet{srikant2019finite} to obtain mean-square bounds critically rely on the linearity of the expectation that is being bounded. {Our approach is more similar to the one in \citet{haque2024stochastic} and \cite{liu2026almost}, where they use the solution to a Poisson equation to prove tight mean-square bounds and almost sure convergence in the presence of Markovian noise, respectively, again crucially using the linearity of the expectation.}

Obtaining bounds on the distribution of an average of correlated iterates is a significant step over obtaining bounds on their variance. In order to overcome this, we introduce a new Lyapunov function that uses the solution of the Poisson equation for the underlying Markov chain of the noise to compensate for bias in the drift of the error. This allows us to obtain a one step recursion on the MGF of the error along with an additive corrective term, which leads to the desired high-probability bounds on the error itself.

\subsection{Related Literature}\label{subsec:literature}

Before presenting our problem setting and main results, we briefly summarize related work on establishing concentration bounds for stochastic approximation algorithms in the settings of stochastic gradient descent (SGD), linear SA, and reinforcement learning (RL).

\textit{Stochastic Gradient Descent.}
A large body of work establishes exponential high-probability bounds for SGD and its variants. In \citet{rakhlin2012,Hazan14}, exponential high-probability bounds are obtained for SGD with non-smooth but strongly convex objectives under conditionally unbiased noise and compact iterates. The analysis was extended in \citet{Harvey19} to the case of sub-Gaussian noise and unbounded iterates, providing one of the few results with exponential bounds under unbounded noise. Exponential high-probability bounds were also established for ergodic mirror descent with Markovian, conditionally biased noise and uniformly bounded variance in \citet{duchi2012ergodic}, again assuming compactness of the iterates. Polynomial high-probability bounds for SGD with heavy-tailed noise in linear models were derived in \citet{heavySGD}. In \citet{telgarsky22}, mirror descent with constant step size and i.i.d.\ noise of uniformly bounded variance (either almost surely bounded or sub-Gaussian) was analyzed, and exponential high-probability bounds were obtained by selecting an appropriate constant stepsize.

\textit{Reinforcement Learning.}
Temporal-difference (TD) learning was introduced as an SA method for policy evaluation in RL \citep{sutton1988learning,sutton2018reinforcement}. Mean-square bounds for TD learning were established in \citet{bhandari2018finite,srikant2019finite}, and high-probability bounds in \citet{dalal2018finite,LSTD17,chandak2023concentration}. In \citet{LSTD17}, the least-squares temporal difference (LSTD) algorithm, which includes projection onto a compact set, was analyzed, yielding exponential high-probability bounds for the $\ell_2$-error with stepsizes of order $\mathcal{O}(k^{-1})$. In \citet{dalal2018finite}, TD learning with linear function approximation and i.i.d.\ sampling was studied, yielding maximal exponential bounds for the last iterate once the number of iterations exceeds a threshold of order $\log(1/\delta)$. In the off-policy setting, finite-sample mean-square bounds were obtained in \citet{chen2021GB,chen2022sample}, although no results are available on high-probability bounds with faster-than-polynomial decay for off-policy TD learning, with or without function approximation. 

One of the earliest works establishing exponential high-probability bounds in RL is \citet{even2003learning}, which analyzes synchronous $Q$-learning with stepsizes of order $\mathcal{O}(k^{-z})$ for $z\in(1/2,1)$ and obtains exponential bounds for sufficiently large iterates. More recent works \citep{li2020sample,li2021q,li2024q} analyze $Q$-learning with constant step size and uniformly bounded, Markovian, possibly conditionally biased noise, establishing exponential high-probability bounds at a prescribed runtime given sufficiently long execution. 

\textit{Linear Stochastic Approximation.}
For linear SA, first-moment bounds on the $\ell_2$-norm of the error with constant step size were established in \citet{lakshminarayanan2018linear,srikant2019finite}. These results imply high-probability bounds with polynomial rather than exponential tails. The strongest known exponential high-probability result for linear SA appears in \citet{dalal2018general}, which analyzes a two-timescale linear SA with decreasing stepsizes and multiplicative, almost surely bounded, Martingale-difference noise. In that setting, maximal exponential bounds were obtained for all sufficiently large iterates by selecting stepsizes depending on both the runtime and the confidence level. A related line of work studies products of random matrices and applies these results to linear SA. In \citet{durmus2021tight}, linear SA with constant step size and Markovian, almost surely bounded noise was considered. High-probability bounds on products of random matrices yield sub-exponential tails when the random matrices are almost surely Hurwitz and polynomial tails when they are Hurwitz only in expectation. These results were extended to Polyak–Ruppert averaged iterates in \citet{durmusAveraged,Mou20}.

\textit{General Stochastic Approximation.}
For general nonlinear SA under arbitrary norms and decreasing stepsizes, \citet{chen2020finite,chen2019finitesample} provide bounds on the second moment of the error. These moment bounds yield high-probability estimates, though without exponential tails. In \citet{thoppe2019concentration}, SA with decreasing stepsizes and Martingale-difference sub-exponential noise is analyzed, yielding maximal exponential high-probability bounds conditioned on the iterates being sufficiently close to the fixed point after a finite time. The follow-up work \citep{chandak2022concentration} considers multiplicative, almost surely bounded, Markovian noise and obtains maximal exponential bounds without conditioning, although the results hold only after a transient phase, and both the bound and probability depend on the random norm of the iterate after that phase. 

In a related line of research, \citet{qu2020finite} studies a general SA algorithm under the infinity norm, where the noise has a uniformly bounded Martingale-difference component and a Markovian component determining the coordinate updates. The resulting random operator is a conditionally biased estimator of the underlying operator, and exponential high-probability bounds are derived under compactness of the iterates. Moreover, \citet{infiniteDim22} considers a variance-reduced version of general SA in Banach spaces with constant step size and i.i.d.\ multiplicative, almost surely bounded noise. By appropriately selecting the step size and the averaging scheme, exponential high-probability bounds are obtained for the error.

Finally, there are three recent papers that are close to our work. In \cite{chen2025concentration} the authors obtain high probability bounds for the case of contractive operators with bounded but ``multiplicative'' Martingale-difference noise, with $\mathcal{O}(1/k)$ step sizes. The Martingale-difference and step size choice are critical constraints for many applications, such as Reinforcement Learning, where the noise is Markovian and the step sizes are larger. Furthermore, in \cite{qian2024almost} the authors obtain high-probability bounds for Markovian noise with $\mathcal{O}(1/k)$ step sizes, albeit with a convergence rate of order $\mathcal{O}\Big(k^{-\ln^{-\nu}(k)/(1-\nu)}\Big)$ for $\nu\in(0,1)$, which is slower than any polynomial (which is far from the true $1/k$ rate). Lastly, in \cite{pham2025time}, the authors obtain anytime high-probability bounds for the case of Martingale-difference multiplicative noise, but only with step-sizes of order $1/k$ with constants that depend on the probability tolerance level $\delta$.

%\todo[inline]{Add the new paper shared by Zaiwei (possibly as a parallel work).}

\subsection{Notation}

We use $\N$, $\R$, and $\R_+$ to denote the collection of natural numbers (excluding $0$), real numbers, and positive reals. For any two numbers $a, b \in \R$, we use $a\vee b$ and $a \wedge b$ to denote $\max\{a,b\}$ and $\min\{a,b\}$, respectively. We use the notation ``a.s.'' to denote that a statement holds almost surely. Finally, we use the Bachmann-Landau notation to describe the asymptotic behavior of certain functions. For functions $f(\cdot)$ and $g(\cdot)$, we write $f(x) = \mathcal{O}(g(x))$ as $x\rightarrow \infty$ if there exists a constant $c > 0$ and $x_0$ such that $|f(x)| \le c g(x)$ for all $x \ge x_0$, or equivalently, $\limsup\nolimits_{x\rightarrow\infty}~{|f(x)|}/{g(x)} < \infty$. We write $f(x) = \Omega(g(x))$ if $g(x) =\mathcal{O}(f(x))$ and use $f(x) = o(g(x))$ to denote $\lim\nolimits_{x\rightarrow\infty} f(x)/g(x) = 0$.

\section{Concentration for general stochastic approximation}\label{sec:boundednoise}
In this section, we present the mathematical model and assumptions for our general stochastic approximation iterates, our main high probability bound for the squared error, and illustrations and applications of these results.

\subsection{Model and Notation}\label{sec:model}
Given a time-homogeneous and ergodic Markov chain $\{S_k\}_{k\geq 0}$ defined on a finite state space $\mathcal S$ with a transition probability matrix $P$, a sequence of $d$-dimensional random vectors $\{Z_k\}_{k\geq 0}$ independent from the Markov chain, an operator $F:\R^d\times \calS \rightarrow \R^d$, and a sequence of step sizes $\{\alpha_k\}_{k\geq 0}$ such that $\alpha_k = \alpha/(k+h)^z$ for some $\alpha, h>0$ and $z\in(0,1]$, we consider the general stochastic approximation algorithm defined by the recursive equation
\begin{equation}\label{eq:SAalgo}
    x_{k+1} = x_k + \alpha_k \lrp{ F(x_k, S_k) + Z_k - x_k },
\end{equation}
with a deterministic initial condition $x_0\in \mathbb{R}^d$.
%\todo[inline]{Shubhada: I think it would look better to use $S_{k+1}$ and $Z_{k+1}$ here instead of $S_{k}$ and $Z_{k}$. That way, if we define  the filtration $\mathcal{F}_k$ as the one defined by the random variables $S_0,Z_0,S_1,Z_1,\dots,S_k,Z_k$, we have that $x_k$ is $\mathcal{F}_k$-measurable. With the current definition, we have that $x_k$ is $\mathcal{F}_{k-1}$-measurable, which looks confusing. -- This would require a lot of changes in the appendix.}
Let $\pi$ be the unique stationary distribution of the Markov chain $\{S_k\}_{k\geq 0}$. We define the averaged operator $\bar{F}: \R^d \rightarrow \R^d$ as
\[ \bar{F}(x) = \E{ F(x,S)}, \]
where $S\sim \pi$. Moreover, it is known \citep[Section 21.2]{douc2018markov} that, for every $x\in\mathbb{R}^d$, there exists a function $V_x:\mathcal{S}\to\R^d$ that satisfies the following Poisson's Equation
\begin{equation}\label{eq:PE}
F(x, s)  - \E{F(x,S)} + (PV_x)(s) = V_x(s), \qquad \forall\, s\in \mathcal{S}, 
\end{equation}
where $(PV_x)(s)= \sum_{s'\in\mathcal S} P(s,s') V_x(s')$. Finally, let $\|\cdot\|_c$ be a norm in $\R^d$, and let $\{\mathcal{F}_k\}_{k\geq 0}$ be the filtration such that $\mathcal F_k := \sigma(S_{k-1}, Z_{k-1}, \dots, S_0, Z_0)$. % that  natural filtration defined by the random sequences $\{S_k\}_{k\geq 0}$ and~$\{Z_k\}_{k\geq 0}$.
We make the following assumptions.

\begin{assmp}[Operator $\bar{F}$]\label{assmp:barf} The following hold for the average operator $\bar{F}$.
\begin{enumerate} 
\item There exists a unique $x^* \in \R^d$ such that $\bar{F}(x^*) = x^*$. 
\item There exist constants $A_3,B_3 \geq 0$ such that, for all $x\in \R^d$,
\[ \|\bar{F}(x) - x^*\|_c \le A_3 \|x- x^*\|_c + B_3. \numberthis\label{eq:barfassump} \]   
\item %For all $s\in \cal S$ and $x\in\R^d$, there exists a function $V_x(s)\in \R^d$ that satisfies the following Poisson Equation: 
    %\[ F(x, s)  - \bar{F}(x) + (PV_x)(s) = V_x(s). \numberthis\label{eq:PE}\]
There exist constants $L_1,L_2\geq 0$ such that, for all $x_1, x_2 \in \R^d$ and $s\in\cal S$,
    \begin{align*}
        &\|V_{x_2}(s) - V_{x_1}(s)\|_c \le L_1 \|x_2 - x_1\|_c, \quad \text{and} \quad \|V_{x^*}(s)\|_c \le L_2. \numberthis\label{eq:Lipschitzness}
    \end{align*}
\end{enumerate}
\end{assmp}

\begin{assmp}[Bounded multiplicative noise]\label{assmp:boundedmultnoise}
    There exist constants $A_1, B_1 \geq 0$ such that, for all $x\in\mathbb{R}^d$ and $s\in\mathcal{S}$, we have
    \[ \|F(x, s) - \bar{F}(x) \|_c \le A_1 \|x \|_c + B_1. \]
\end{assmp}

\begin{assmp}[Conditionally unbiased and bounded additive noise $Z$]\label{assmp:boundedaddnoise} We have $\mathbb{E}[Z_k\mid \mathcal{F}_{k-1}] = 0$ for all $k\geq 0$. Furthermore, there exists a constant $B_2 \geq 0$ such that, for every $k\geq 0$, we have $\|Z_k \|_c \le B_2$, almost surely.
\end{assmp}

\noindent Note that Assumption~\ref{assmp:boundedaddnoise} implies that $\{Z_k\}_{k\geq 0}$ is a Martingale-difference sequence.

\begin{remark}\label{rem:a13significance}
Assumptions~\ref{assmp:barf},~\ref{assmp:boundedmultnoise}, and~\ref{assmp:boundedaddnoise} imply that, for every $x\in\mathbb{R}^d$ and $k\geq 0$, we have
\[ \|F(x,S_k) + Z_k - x^*\|_c \le (A_1 + A_3) \|x-x^*\|_c + B_{123} + A_1\|x^*\|_c, \qquad \text{a.s.}, \]
where $B_{123} = B_1+B_2+B_3$. Note that, if $A_{1} + A_3 < 1$, then the noisy operator $F(\cdot, S_k) + Z_k$ is almost surely an approximate pseudo-contraction (approximate due to the presence of additive constants on the right-hand side). On the other hand, if $A_{1} + A_3 > 1$, then the operator can be expansive with a positive probability. Later, we see that in the latter case, the tail of the error is significantly heavier. In particular, the factor
\begin{equation}\label{eq:def_D}
    D := A_{1} + A_3 - 1
\end{equation}
plays a significant role in determining the tail behavior of the error. 
\end{remark}

\begin{assmp}[Lyapunov Function]\label{assmp:lyapunov} There exists a Lyapunov function $M(\cdot)$, a norm $\|\cdot\|_s$, and constants $\eta, L_s, l, u > 0$ such that, for all $x,y \in \mathbb{R}^d$,
    \begin{align}
        &\langle \nabla M(x-x^*), \bar{F}(x) - x \rangle \le -\eta M(x-x^*), \label{eq:negdrift}\\
        & M(y) \le M(x) + \langle \nabla M(x), y-x \rangle + \frac{L_s}{2} \|x-y\|^2_s, \label{eq:Smoothness}\\
        & l M(x) \le \|x\|^2_c \le u M(x), \label{eq:normMcM} \\
        & \nabla M(0) = 0. \label{eq:0gradatmin} 
        %& M(\Pi_{\cal B}(x) - x^*) \le M(x-x^*) \label{eq:projectioncontracts},
    \end{align}
\end{assmp}

While this assumption seems hard to check in practice, it is quite straightforward in many cases. For example, when the average operator $\bar{F}$ is a contraction with respect to a smooth norm $\|\cdot\|_c$, then it can be checked that this assumption is satisfied with $M(x)=\|x-x^*\|_c^2$ and $\|\cdot\|_s=\|\cdot\|_c$. This and other examples are explored in detail in Section \ref{sec:illustrations}.

\medskip

Finally, we have the following assumption for the step sizes $\alpha_k=\alpha/(k+h)^z$.

\begin{assmp}[Conditions on tunable parameters $\alpha$ and $h$]\label{assmp:stepsize} When $z=1$, we have $\alpha> {2}/{\eta}$, and for $z\in (0,1)$, we have $\alpha > 0$ and $h\ge \lrp{{4z}/({\alpha \eta})}^{1/(1-z)}$. Further, in both the cases, we have 
\[ h \geq \max\lrset{ \frac{4 C_1}{\eta\alpha}, \frac{\eta}{2\alpha}, \frac{u L_s (L_1 + L_2)}{\lcs^2\alpha}, 8 }, \]
where $C_1$ is defined in Appendix~\ref{app:const_ball}, and $\lcs$ and $\ucs$ are constants such that, for all $x\in\R^d$,
\[ \lcs \|x\|_s \le \|x\|_c \le \ucs \|x\|_s. \numberthis\label{eq:normequiv} \] 
\end{assmp}

%To prove the concentration of $x_k-x^*$, we will work with the exponential of the Lyapunov function $M$ from Assumption~\ref{assmp:lyapunov}. 

\subsection{Main results}
We are now ready to present our concentration results for the error $\|x_k-x^*\|_c$ for the iterates defined in Equation~\eqref{eq:SAalgo}. We first consider the case where the step sizes are of order $\mathcal{O}(1/k)$.

\begin{theorem}\label{th:ballargumentbound}
Suppose that we have step sizes $\alpha_k=\alpha/(k+h)$, where $\alpha$ and $h$ satisfy Assumption \ref{assmp:stepsize}. Under assumptions~\ref{assmp:barf}, \ref{assmp:boundedmultnoise},~\ref{assmp:boundedaddnoise}, and~\ref{assmp:lyapunov}, for $x_0 \ne x^*$, the following hold with the constants $\bar{a}_1, \bar{b}_1, \bar{b}_2, \bar{b}_4$ as in equations~\eqref{eq:abar1},~\eqref{eq:bar1},~\eqref{eq:bar2}, and~\eqref{eq:bar4}.
\begin{enumerate}
    \item\label{ball_part1_cont} Suppose $D \le 0$. For any $\delta\in(0,1)$, the following holds with probability at least $1-\delta$: 
    \begin{align*}
        \forall k \ge 0, \qquad   \| x_k - x^*\|^2_c 
        &\leq %\min\bigg\{\bar{a}_1\lrp{1 \vee \log^2\lrp{\frac{k-1+h}{h-1}}},  \\ &
        \frac{\bar{a}_1 \lrp{ 1 \vee \log\left(\frac{k-1+h}{h-1}\right)^2}}{k+h-1} \bigg[\bar{b}_1\log\left(\frac{1}{\delta}\right) + \bar{b}_2 + \bar{b}_4 \log\left(\frac{k+h-1}{h-1}\right)\bigg]. %\bigg\}.
    \end{align*}
    \item\label{ball_part1_exp} Suppose $D > 0$. For any $\delta \in (0,1)$, the following holds with probability at least $1-\delta$: 
    \begin{align*}
        \forall k\ge 0, \qquad \| x_k - x^*\|^2_c 
        &\leq %\min\left\{ \bar{a}_1\lrs{\bar{b}_1\log\left(\frac{1}{\delta}\right) + \bar{b}_2 + \bar{b}_3}^{2\alpha (A_{13}-1)}, \right. \\ & \left. 
        \frac{\bar{a}_1 \lrs{\bar{b}_1\log\left(\frac{1}{\delta}\right) + \bar{b}_2 }^{2\alpha D}}{k+h-1} \bigg[\bar{b}_1\log\left(\frac{1}{\delta}\right) + \bar{b}_2 + \bar{b}_4 \log\left(\frac{k+h-1}{h-1}\right)\bigg]. %\right\}.
    \end{align*}
\end{enumerate}
    The constants $\bar{a}_1$, $\bar{b}_1$, $\bar{b}_2$, and $\bar{b}_4$ are given by 
    \begin{align*}
        &\bar{a}_1 = \begin{cases} \lrp{\frac{1}{h-1}}^{2\alpha D} \lrp{ \|x_0 - x^*\|_c + \frac{A_1 \|x^*\|_c + B_{123}}{D} }^2, & \qquad \text{if } D > 0, \\
        2\|x_0 - x^*\|^2_c + 2(A_1 \|x^*\|_c + B_{123})^2 \alpha^2,&\qquad\text{if } D = 0,\\
        2\|x_0 - x^*\|^2_c + \frac{2(A_1 \|x^*\|_c + B_{123})^2}{D^2},&\qquad\text{if } D < 0,
        \end{cases}\numberthis\label{eq:abar1} \\
        &\bar{b}_1 = \frac{u\alpha }{\theta}, \numberthis\label{eq:bar1} \\
        &\bar{b}_2 = \frac{u\alpha}{\theta}\lrp{ \frac{8\alpha e \theta D_2}{(\frac{\alpha \eta}{2} - 1)\|x_0 - x^*\|^2_c} + \frac{ 2u L_s (L_1 + L_2) \theta }{ l\lcs^2}} \\
        &\qquad\qquad\qquad + \frac{\eta\lcs^2 \alpha}{64 l L^2_s \alpha_0 L^2_1 L_2 \theta }\lrs{2\lcs^2 L_2 +  2 u \alpha_{-1} (L_1 + L_2)  L_2 L_s + \lcs^2 \alpha_{-1} l L_s L^2_1 }, \numberthis\label{eq:bar2} \\
        &\bar{b}_4 = \frac{\alpha^2 D_3 u}{\theta},\numberthis\label{eq:bar4}
    \end{align*}
    where the constants $D_2, D_3$ and $\theta$ are defined in Appendix~\ref{app:const_ball}.
\end{theorem}
We present the proof sketch for Theorem~\ref{th:ballargumentbound} in Section~\ref{sec:proof_sketch_12}.\\

Note that the parameter $D$ plays a crucial role in determining the tail behavior of squared error, as the noise is significantly heavier when $D > 0$. This is because, in this case, the noisy operator is an expansion with positive probability.

\begin{remark}\label{rem:reduction}
    The bound in Part~\ref{ball_part1_cont} of Theorem~\ref{th:ballargumentbound} can be recovered from the one in Part~\ref{ball_part1_exp} with $D$ set to $0$, modulo a multiplicative factor of $\log(k)^2$ (which is present only when $D = 0$). 
\end{remark}

The tail behavior of the additive noise sequence $\{Z_k\}_{k\geq 0}$ is captured by the constant $B_2$, which corresponds to the bound in Assumption~\ref{assmp:boundedaddnoise}. %, and is proportional to the sub-Gaussian constant (recall, the sub-Gaussian variance parameter for $Z_k$ is proportional to $B^2_2$). 
Below, we highlight the dependence of the time-uniform bounds from Theorem~\ref{th:ballargumentbound} on $B_2$, $k$, and $\delta$ in the three cases, treating every other parameter as a constant.
\[
    \|x_k - x^*\|^2_c \le \begin{cases}
        \mathcal{O}\lrp{\frac{B^4_2 \log\left(\frac{k}{\delta}\right)}{k}}, & \text{ if } D < 0, \\
        \mathcal{O}\lrp{\frac{B^4_2 \log(k)^2 \log\left(\frac{k}{\delta}\right) }{k}}, & \text{ if } D = 0, \\
        \mathcal{O}\lrp{ \frac{\lrp{B^2_2 \log\left(\frac{1}{\delta}\right)}^{2\alpha D} B^4_2 \log\left(\frac{k}{\delta}\right)}{k} }, & \text{ if } D > 0.
    \end{cases}\numberthis\label{eq:B_2dep}
\]
We use this later in Section~\ref{sec:generalcontractive}, where we relax Assumption~\ref{assmp:boundedaddnoise}. %, and the constant $B_2$ will become a function of the probability $\delta$. 

\begin{remark} Recall from Law of Iterated Logarithm that even in the simplest setting of empirical averages of i.i.d.\ observations (which, we show in Section~\ref{sec:contractiveSA}, is a special case of the setting considered in this section), the time-uniform upper bound resulting from the maximal concentration inequalities cannot be tighter than $\mathcal{O}(\log\log(k)/k)$. Our upper bound in Theorem~\ref{th:ballargumentbound}, Part~\ref{ball_part1_cont} is $\mathcal{O}(\log k/k)$, hinting at a possible looseness of our bound in the dependence on $k$. In fact, the impossibility result of \citet[Section 2.2]{chen2025concentration} demonstrates the tightness of bounds in Theorem~\ref{th:ballargumentbound} in their dependence on $\delta$ and $k$ (up to $\log$ factors). We leave the improvement in the logarithmic factors for future research.
\end{remark}

In the following, we extend the above results step sizes of order $\mathcal{O}(1/k^z)$ for some $z\in(0,1)$. %, still treating separately the cases where $A_{13} - 1$ is $< 0$, $= 0$, and $> 0$. 
We remark that we present the first tight bounds for step sizes of the aforementioned order, even for the special case of i.i.d.\ noise, instead of the Markov noise that we consider in this work. %Recall constants $D_2, D_3, $ and $\theta$ from Theorem~\ref{th:ballargumentbound} (or Appendix~\ref{app:const_ball}).

\begin{theorem}\label{th:ballargumentboundz}
Suppose that we have step sizes $\alpha_k=\alpha/(k+h)^z$ for some $z\in(0,1)$, where $\alpha$ and $h$ satisfy Assumption \ref{assmp:stepsize}. Under assumptions~\ref{assmp:barf}, \ref{assmp:boundedmultnoise},~\ref{assmp:boundedaddnoise}, and~\ref{assmp:lyapunov}, for $x_0 \ne x^*$, the following hold.
\begin{enumerate}
    \item\label{ball_partz_cont} Suppose $D < 0$. For any $\delta \in (0,1)$, the following holds with probability at least $1-\delta$: 
    \begin{align*}
        \forall k\ge 0,~\|x_k - x^*\|^2_c &\le  \frac{\bar{a}_1}{(k+h-1)^z} \left[ \bar{c}_1\log\tfrac{1}{\delta} + 2\bar{c}_1\log\lrp{{k+1}} + \bar{c}_2 + \bar{c}_3 e^{\frac{-\eta \alpha}{2(1-z)}\lrp{(k+h)^{1-z} - h^{1-z}}} \right],
    \end{align*}
    where
    \begin{align*}
        &\bar{a}_1 = 2\|x_0 - x^*\|^2_c + \frac{2(A_1 \|x^*\|_c + B_{123})^2}{D^2}, \\
        &\bar{c}_1 = \frac{\alpha u}{\theta},\\
        &\bar{c}_2 = \frac{8 \alpha u D_2 }{\eta\|x_0 - x^*\|^2_c } + \frac{u^2 L_s (L_1 + L_2) \alpha}{\lcs^2 l} + \frac{\alpha u}{\theta}\log\left(\frac{\pi^2}{6}\right), \\
        &\bar{c}_3 = \frac{\eta\lcs^2 \alpha }{64 \theta l L^2_s \alpha_0 L^2_1 L_2 }\lrs{2\lcs^2 L_2 +  2 u \alpha_{-1} (L_1 + L_2)  L_2 L_s + \lcs^2 \alpha_{-1} l L_s L^2_1 }.
    \end{align*}
    \item \label{ball_partz_equal} Suppose $D = 0$. For any $\delta\in(0,1)$, the following holds with probability at least $1-\delta$: 
    \begin{align*}
     \forall k\ge 0, \qquad \|{x}_k - x^*\|^2_c &\leq \mathcal{O}\lrp{\frac{\log\left(\frac{1}{\delta}\right)^\frac{1}{z}}{k^z}}. 
    \end{align*}
    \item\label{ball_partz_exp} Suppose $D > 0$. For any $\delta\in(0,1)$, the following holds with probability at least $1-\delta$: 
    \begin{align*}
     \forall k\ge 0, \qquad \|{x}_k - x^*\|^2_c &\leq \mathcal{O}\lrp{\frac{e^{\log\left(\frac{1}{\delta}\right)^\frac{1-z}{ \beta z}}}{k^z}}, \quad \text{ for any }  \beta < 1. 
    \end{align*}
\end{enumerate}
\end{theorem}
We present the proof sketch for Theorem~\ref{th:ballargumentboundz} in Section~\ref{sec:proof_sketch_12}.\\ % A few remarks are in order.

\begin{remark}\label{rem:phasetransition}
    When the operator is expansive with positive probability, i.e., when $D > 0$,  Part~\ref{ball_part1_exp} of Theorem~\ref{th:ballargumentbound} states that, for $z=1$, we always have a sub-Weibull tail for the squared error. Moreover, Part~\ref{ball_partz_exp} of Theorem~\ref{th:ballargumentboundz} states that, if $1/2 < z < 1$, we can pick $\beta$ large enough so that the exponent $(1-z)/(\beta z)$ is less than $1$, and thus we get a bound of order $\mathcal{O}\big(e^{\log\left(1/\delta\right)^\alpha}/k^z\big)$, for some $\alpha \in (0,1)$. This means that the resulting tail for the squared error is lighter than any Pareto distribution, but heavier than any Weibull distribution. 
    
    On the other hand, for $z \le 1/2$, the tail is heavier than any Pareto distribution. In this case, our bound in Theorem~\ref{th:ballargumentboundz} is necessarily loose. Indeed, a tighter tail bound can be obtained by using the mean-squared error bound from \citet{chen2024lyapunov} and applying Markov's inequality, which yields a Pareto tail. This suggests that the true tail of the squared error is Pareto or even lighter in this case. Finally, in the extreme case where the step sizes are a constant (i.e., when $z=0$) \citet{srikant2019finite} show that some moments of the squared-error do not exist, and thus the tail is indeed Pareto in that case. However, whether the transition from lighter-than-Pareto to Pareto occurs at $z=1/2$ or at some $z < 1/2$ remains open. 
\end{remark}

\begin{remark}\label{rem:initialerror}
    The presence of condition $x_0 \ne x^*$ in theorems~\ref{th:ballargumentbound} and~\ref{th:ballargumentboundz} is only for simplicity of presentation. This condition can be easily removed, and a similar tail bound can be obtained by appropriately modifying the constants involving $\|x_0 - x^*\|_c$ and adjusting the corresponding terms in the proof.
\end{remark}

%\subsubsection{An impossibility result on the tail decay rate}

Theorem \ref{th:ballargumentboundz} shows that, when $z<1$ and $D \le 0$ (i.e., when the operator is almost surely a contraction), the squared error has sub-Weibull tail, and when $z<1$ and $D > 0$ (i.e., when the operator is expansive with positive probability), the squared error has a tail that is heavier than any Weibull but lighter than any Pareto. The following impossibility result states that, in general, qualitatively better tail bounds are not possible.

\begin{theorem}\label{thm:impossibility}
    $ $
    \begin{itemize}
        \item[(1)] For any $\beta'<2-z$, there do not exist constants $K_1',K_2'>0$ s.t. 
        \begin{align*}
            \mathbb{P}\left(\|x_k-x^*\|_c^2 \leq \frac{K_1'\log\left(\frac{K_2'}{\delta}\right)^{\beta'}}{(k+h)^z} \right) \geq 1- \delta,\qquad \forall\,\delta\in(0,1), \quad k\geq 0.
        \end{align*}
        \item[(2)] For any $\beta'<1-z$, there do not exist constants $K_1',K_2'>0$ s.t. 
        \begin{align*}
            \mathbb{P}\left(\|x_k-x^*\|_c^2 \leq \frac{K_1' \exp\Big(K_2'\log\left(\frac{1}{\delta}\right)^{\beta'}\Big)}{(k+h)^z} \right) \geq 1- \delta,\qquad \forall\,\delta\in(0,1), \quad k\geq 0.
        \end{align*}
    \end{itemize}
\end{theorem}
Theorem \ref{thm:impossibility}(1) implies that, in general, the tail of the error is at least sub-Weibull with shape parameter $2-z$, which is always heavier than sub-Gaussian. This is slightly lighter than our high probability bound of Part~\ref{ball_partz_equal} of Theorem \ref{th:ballargumentboundz}, which yields a sub-Weibull tail with shape parameter $1/z$. However, in both cases the tail is heavier than sub-Gaussian. Furthermore, Theorem \ref{thm:impossibility}(2) implies that the error cannot be upper bounded by something of order $\exp\big(K_2'\log\left(1/\delta\right)^{\beta'}\big)/k^z$ for any $\beta'<1-z$. This is again slightly lighter than our high probability bound of Part~\ref{ball_partz_exp} of Theorem \ref{th:ballargumentboundz}, which yields an upper bound of order $\exp\big(K_2'\log(1/\delta)^{(1-z)/z}\big)/k^z$. However, in both cases the tail is heavier than any Weibull and lighter than any Pareto for all $z\in (1/2,1)$. The proof of Theorem \ref{thm:impossibility} relies on obtaining lower bounds on the MGF of a 1-dimensional SA example, and is deferred to Appendix~\ref{app:impossibility}.

\subsection{Illustrations and Applications}\label{sec:illustrations}
A wide range of algorithms in optimization and machine learning % in optimization, contractive SA, linear SA, semi-norm contractive SA, SA with dissipativity assumptions, etc. 
can be modeled with the general setup of Section~\ref{sec:model} for specific choices for the operator $F$. In this subsection, we show the applicability of our concentration results in theorems~\ref{th:ballargumentbound} and~\ref{th:ballargumentboundz} to several relevant settings.

\subsubsection{Contractive Stochastic Approximation}\label{sec:contractiveSA}
In this section, we show that contractive SA is a special case of the setting considered in Section~\ref{sec:model}, and present the concentration results for the iterates of a contractive SA algorithm. In particular, in this section, we make the following assumption on the average operator $\bar{F}$, in addition to assumptions~\ref{assmp:boundedmultnoise} and~\ref{assmp:boundedaddnoise} on the underlying noise. 

\begin{assmp}[Contraction of average operator $\bar{F}(\cdot)$]\label{assmp:contraction} There exists a constant $\gamma_c \in [0,1)$ such that the average operator $\bar{F}(\cdot) := \E{ F(\cdot,S)}$ where ${S\sim\pi}$, is a $\gamma_c$-contraction with respect to the $\|\cdot\|_c$-norm. That is, we have
    \[ \|\bar{F}(x_1) - \bar{F}(x_2)\|_c \le \gamma_c \|x_1 - x_2\|_c, \qquad \forall \, x_1, x_2 \in \R^d.  \]
Moreover, it is uniformly Lipschitz, i.e., there exists a constant $L_F$ such that 
\[ \|F(x_1, s) - F(x_2, s)\|_c \le L_F\|x_1 - x_2\|_c \quad \forall x_1, x_2 \in \R^d, \quad \text{and} \quad \forall s\in \mathcal S. \]
\end{assmp}

In Lemma~\ref{lem:assmpbarf_cont}, we show that Assumption~\ref{assmp:barf} follows from assumptions~\ref{assmp:boundedmultnoise},~\ref{assmp:boundedaddnoise},and~\ref{assmp:contraction}, with 
\begin{align*}
    A_3 = \gamma_c, \quad B_3 = 0, \quad L_1 = (L_F + \gamma_c)(1+\max_{s\in\mathcal S} \E{\tau_s | S_0 = s}), \quad L_2 = L_1 \|x^*\|_c + \max_{s\in\mathcal{S}} \|V_0(s)\|_c, \numberthis\label{eq:cont_const1}
\end{align*}
where $\tau_s := \min\{ n > 0: S_n = s \}$.

\medskip

Arguably the hardest assumption to verify is Assumption~\ref{assmp:lyapunov}, which requires the existence of a Lyapunov function that satisfies certain properties. To show that this assumption is satisfied for contractive average operators, we introduce the following. For $x\in\R^d$ and $\mu > 0$ (to be chosen later), we have the generalized Moreau envelope (\cite{beck2017first})
\[ M(x) := \min\limits_{u\in\R^d}\lrset{ \frac{1}{2}\|u\|^2_c + \frac{1}{2\mu} \|x-u\|^2_s }. \numberthis \label{eq:moreau} \]
Here, $\|\cdot\|_s$ is a norm on $\R^d$ chosen so that $\|\cdot\|^2_s/2$ is a $L$-smooth function with respect to $\|\cdot\|_s$. One can show that $M(\cdot)$ is convex and $(L/\mu)$-smooth with respect to $\|\cdot\|_s$. Additionally, there exists a norm, denoted by  $\|\cdot\|_M$, such that $M(x) = \|x\|^2_M/2$ for $x\in \R^d$. We refer the reader to \cite{beck2017first} and \cite[Lemma 2.1]{chen2020finite} for proofs of these properties of $M(\cdot)$. 

Finally, since all norms in $\R^d$ are topologically equivalent, there exist constants $\lcs$, $\ucs$, $\lcm$, and $\ucm$ such that 
\begin{align*}
    &\lcs\|\cdot\|_c \le \|\cdot\|_s \le \ucs \|\cdot\|_c \quad \text{and} \quad \lcm\|\cdot\|_c \le \|\cdot\|_M \le \ucm \|\cdot\|_c.
\end{align*}
In particular, \cite[Lemma 2.1]{chen2020finite} implies that $\lcm = 1/\sqrt{1 + \mu \ucs^2}$ and $\ucm = 1/\sqrt{1 + \mu \lcs^2}$, giving
\begin{align*}
    &\lcs\|\cdot\|_c \le \|\cdot\|_s \le \ucs \|\cdot\|_c \quad \text{ and } \quad \frac{1}{\sqrt{1 + \mu \ucs^2}}\|\cdot\|_c \le \|\cdot\|_M \le \frac{1}{\sqrt{1 + \mu \lcs^2}} \|\cdot\|_c.
\end{align*}
We choose the tunable parameter $\mu > 0$ such that
\[ \gamma_c \sqrt{\frac{1 + \mu \ucs^2}{1 + \mu \lcs^2}} < 1. \numberthis\label{eq:muassmp} \]
Note that, since $\gamma_c<1$, this is always possible for $\mu$ sufficiently large.

\medskip

Lemma~\ref{lem:assmpLyapunov_cont} in Appendix~\ref{app:conc_contraction} shows that Assumption~\ref{assmp:lyapunov} is satisfied with the Moreau envelope defined in Equation~\eqref{eq:moreau} as the Lyapunonv function, and 
\begin{align*}
    \eta = 2\left(1- \gamma_c \sqrt{\frac{1 + \mu \ucs^2}{1 + \mu \lcs^2}} \right), \quad L_s= \frac{L}{\mu}, \quad l = 2 (1 + \mu \lcs^2),\quad \text{and} \quad u = 2 (1 + \mu \ucs^2).\numberthis\label{eq:cont_const2}
\end{align*}

\begin{theorem}\label{th:contractiveconc}
    For all $k\in\N$, let $\alpha_k = \alpha/(k+h)^z$ for some $z\in(0,1]$, where $\alpha$ and $h$ satisfy the conditions in Assumption~\ref{assmp:stepsize}. Under assumptions~\ref{assmp:boundedmultnoise},~\ref{assmp:boundedaddnoise}, and~\ref{assmp:contraction}, and for $x_0 \ne x^*$, we have the bounds in theorems~\ref{th:ballargumentbound} and~\ref{th:ballargumentboundz} with constants as in equations~\eqref{eq:cont_const1} and~\eqref{eq:cont_const2}.
\end{theorem}

\begin{remark}\label{rem:bootstrapvsball}
Since i.i.d.\ noise is a special case of Markov noise, our bound above also applies to the i.i.d.\ noise setting. In particular, for $D > 0$ and $z = 1$, our bound in Theorem~\ref{th:contractiveconc} yields a time-uniform upper bound for the i.i.d.\ noise setting of order
\[ \mathcal{O}\lrp{\frac{\log(k)}{k}{\log\left(\frac{1}{\delta}\right)^{2\alpha D+1}} }. \] 
This matches the best possible exponent for $\log(1/\delta)$ established in the impossibility result \citet[Theorem 2.2]{chen2025concentration}, and it is tighter than the high-probability bound in \citet[Theorem 2.3]{chen2025concentration}, in which the exponent of $\log(1/\delta)$ is $\lceil 2\alpha D + 1\rceil$. The appearance of this ceiling is an artifact of their bootstrapping  proof technique, which increases the power of $\log(1/\delta)$ by exactly one in each iteration. Our novel ball argument, by operating in a continuum rather than through discrete bootstrapping steps, removes this artifact and thereby closes the gap between the lower and upper bounds. % in \citet[Fig. 3]{chen2025concentration}.
\end{remark}

\begin{corollary}\label{cor:additive_contbound}
In the setup of Theorem~\ref{th:contractiveconc}, consider the special case where $A_1 = 0$, i.e., where we only have additive noise. % Then $A_{13} = A_3 = \gamma_c < 1$. 
In this case, we get a time-uniform upper bound on $\|\xk - x^*\|^2_c$ of order $\mathcal{O}(\log(k/\delta)/k)$ % where we only highlight the dominant dependence on $k$ and $\delta$, 
that holds with probability at least $1-\delta$. 
\end{corollary}
Finally, when $A_1 = B_1 = A_3  = 0$, the iterates of the algorithm given by Equation~\eqref{eq:SAalgo} with step sizes $\alpha_k = 1/(k+1)$ and i.i.d.\ noise are just the empirical averages of $\{Z_k\}_{k\geq 0}$. That is, we obtain
\[ x_k = \frac{1}{k+1} \sum_{i=0}^k Z_i. \]
In this case, $D < 0$, and thus Equation~\eqref{eq:B_2dep} implies a time-uniform bound of order $\mathcal{O}(%({B^4_2}/{k})
\log(k/\delta)/k)$, which is known to be tight in the dependence on $\delta$ and $k$ up to polylog$(k)$ factors.

\subsubsection{Stationary Mean Estimation}
In this subsection, we use Theorem~\ref{th:contractiveconc} (or Corollary~\ref{cor:additive_contbound}) to recover a Hoeffding-like concentration result for the empirical average of samples from a Markov chain around the stationary mean. In the following, we first introduce the setup formally before explicitly stating this concentration result. 

\medskip

Consider the special case of the setting in Section~\ref{sec:contractiveSA}, in which the noisy operator is given by $F(x,s) = f(s)$ for some function $f:\mathcal S \rightarrow \R$. We then have $\bar{F}(x) = \bar{f}:= \E{f(S)}$, where ${S\sim\pi}$. Next, consider the iterates produced by Equation~\eqref{eq:SAalgo} with initial condition $x_0 = 0$ and step sizes $\alpha_k = 1/(k+1)$. It follows that $x_k = \sum\nolimits_{i=0}^{k-1}f(S_i)/k$, which is the empirical mean of $f(S)$. In this case, the Ergodic theorem \citep[Theorem 5.2.1]{douc2018markov} states that $x_k \xrightarrow{a.s.} \bar{f}$ as $k\rightarrow \infty$. 

In order to obtain high-probability bounds, we apply Corollary~\ref{cor:markovempmeanconc} with $\gamma_c = 0$, $A_1=  0$, $B_1 = \max_{s\in\mathcal{{S}}} |f(s) - \bar{f}|$, $B_2=  0$, $A_3 = 0$, $B_3 = 0$, $L_1 = 0$, and $L_2 = \max_{s\in\mathcal S} |V(s)|$, where $V(s)$ is a solution to the Poisson Equation $f(s) - \bar{f} = V(s) - [PV](s)$. For these constants, Corollary~\ref{cor:markovempmeanconc} simplifies to the following high-probability bound for the empirical average estimator for a stationary mean. 

\begin{corollary} \label{cor:markovempmeanconc}
    For any $\delta\in (0,1)$, we have
    \[  \mathbb{P}\lrp{ \exists k \in \N : \left|\frac{1}{k}\sum\limits_{i=0}^{k-1} f(S_i) - \bar{f}\right|^2 \ge \mathcal{O}\lrp{ \frac{\log\left(\frac{k}{\delta}\right)}{k} } }\le \delta. \]
\end{corollary}
Several previous works have established concentration results for the empirical average of samples from a DTMC in a fairly general setting. Our bound in Corollary~\ref{cor:markovempmeanconc} recovers those of, for example,  \cite{glynn2002hoeffding,boucher2009hoeffding,liu2021hoeffding,chung2012chernoff, paulin2015concentration},  when specialized to a finite state discrete-time Markov chain setting.

%{\color{red} A few previous works have established concentration results for the empirical average of samples from a DTMC. We now discuss some of these and compare our bound in Corollary~\ref{cor:markovempmeanconc} to the existing ones. References \citet[Theorem 2]{glynn2002hoeffding} and \citet[Theorem 1]{boucher2009hoeffding} prove a generalization of Hoeffding's inequality to empirical averages of samples from a Markov chain in a much more general setting than considered in this section, via Doeblin’s condition and the Drazin inverse, respectively. However, while our bound holds for any $\delta \in (0,1)$, when specialized to the setting considered in this section, their bound only holds for $\delta$ smaller than a threshold (this is presented as a constraint on the minimum number of samples  in both these works). 

%Reference \citet{liu2021hoeffding} develop extensions of results in \cite{glynn2002hoeffding} to Markov chains with general state spaces, but suffers from a similar drawback of having a threshold on $\delta$. \cite{lezaud1998chernoff} } 

\subsubsection{Linear Stochastic Approximation}\label{sec:linearSA}
We now consider the problem of finding the solution to a linear system of equations. For $A \in \R^{d\times d}$ and $b\in \R^d$. We are interested in finding $x^* \in \R^d$ such that $Ax^* + b = 0$ (or equivalently, the fixed point $x^*$ for $Ax^* + b + x^* = x^*$). However, $A$ and $b$ are unknown. Instead, we have an oracle that, when queried at time $k\geq 0$, returns $A_k x + b_k$, where $A_k = A(S_k)$ and $b_k = b(S_k)$ are noisy versions of the unknown $A$ and $b$, and $S_k$ is the state of the underlying Markov chain at time $k$ which we assume to have a finite state space and to be ergodic. Note that $(A_k, b_k)$ has a Markovian dependence over time $k$. Below, we state our assumptions on $A$ and $b$. 

\begin{assmp}\label{assmp:linearSA}
We have $\E{ A(S)} = A$ and $\E{b(S)} = b$, where ${S\sim \pi}$. Moreover, the matrix $A$ is Hurwitz, i.e., the eigenvalues of $A$ have strictly negative real parts. 
\end{assmp}

It is well known that, under Assumption~\ref{assmp:linearSA}, the following stochastic approximation algorithm converges to $x^*$:
\[ x_{k+1} = x_k + \alpha_k (A_k x_k + b_k). \numberthis\label{eq:alglinearSA} \]
In particular, Assumption~\ref{assmp:linearSA} ensures the stability of the iterates obtained by Equation~\eqref{eq:alglinearSA} (cf., \citet{srikant2019finite}). In fact, when $A$ is Hurwitz, for the $d\times d$ identity matrix $I_d$, the Lyapunov equation $A^T P+P A +I_d = 0$ has a unique positive definite solution, denoted by $\bar{P}$. Then, the asymptotic convergence of the iterates given by Equation~\eqref{eq:alglinearSA} to $x^*$ follows from the ODE method \cite{borkar2009stochastic}. %We refer the reader to \cite{srikant2019finite} and \cite[Appendix A]{agrawal2024markov} for a finite-sample upper bound on the mean-squared error for the iterates of the above linear SA algorithm. 

\medskip

In what follows, we present a tight concentration result for $\xk$. To this end, we first show that Equation~\eqref{eq:alglinearSA} is a special case of contractive SA from the previous section, that is, the noise and the operator satisfy assumptions~\ref{assmp:boundedmultnoise},~\ref{assmp:boundedaddnoise}, and~\ref{assmp:contraction}. Since $\bar{P}$ is positive definite, we define the norm $\|\cdot\|_{\bar{P}}$ with respect to $\bar{P}$ in $\R^d$ as $\|x\|^2_{\bar{P}} = x^T \bar{P} x$. In this case, we use the norm $\|\cdot\|_c = \|\cdot\|_{\bar{P}}$. Since the underlying Markov chain has a finite state space, assumptions~\ref{assmp:boundedmultnoise} and~\ref{assmp:boundedaddnoise} are satisfied with 
\[A_1 = \max_s\|A(s)\|_{\bar{P}} + \|A\|_{\bar{P}}, \quad B_1 = \max_s \|b(s)\|_{\bar{P}} + \|b\|_{\bar{P}},\quad \text{and}\quad B_2 = 0. \numberthis\label{eq:linearSA_const1} \] 
Next, let $\beta := \lambda^{-1}_{\max}(A^T\bar{P}A)/2$, where $\lambda_{\max}(\cdot)$ returns the largest eigenvalue of the input symmetric matrix. For $x\in\R^d$ and $s\in\calS$, define $F_\beta(x,s):= \beta A(s) x + \beta b(s) + x$. Then, Equation~\eqref{eq:alglinearSA} can equivalently be written as
\[ x_{k+1} = x_k + \frac{\alpha_k}{\beta} \lrp{ F_{\beta}(x_k, S_k) - x_k }. \]
Moreover, the averaged operator $\bar{F}_\beta(x) := \Exp{\pi}{F_\beta(x,S)}$ is known to be a $\gamma_c$-contraction for some $\gamma_c \in (0,1)$ (\citet[Proposition 3.1]{chen2025concentration}), implying that the first condition in Assumption~\ref{assmp:contraction} holds. Moreover,  Lemma~\ref{lem:verifyingassmp6linearsa} in Appendix~\ref{app:linearSA} shows that the noisy operator $F_\beta$ is uniformly Lipschitz, hence verifying the second condition in Assumption~\ref{assmp:contraction}. 

%In Theorem~\ref{th:linearSAconc} below, we present a time-uniform bound on the probability that the squared-error exceeds a threshold for the iterates of Algorithm~\ref{eq:alglinearSA}.

\begin{theorem}\label{th:linearSAconc}
    For all $k\in\N$, let $\alpha_k = \alpha/(h+k)^z$ for some $z\in(0,1]$, where $\alpha$ and $h$ satisfy the conditions in Assumption~\ref{assmp:stepsize}. Under Assumption~\ref{assmp:linearSA}, and for $x_0 \ne x^*$, we obtain the high-probability bounds given in theorems~\ref{th:ballargumentbound} and~\ref{th:ballargumentboundz} for the iterates defined in Equation~\eqref{eq:alglinearSA}, with the constants as in Equation~\eqref{eq:linearSA_const1}.
\end{theorem}

\begin{comment}
\begin{theorem}\label{th:linearSAconc_comp}
    For all $k\in\N$, let $\alpha_k = \alpha/(k+h)^z$ for some $z\in(0,1]$, where $\alpha > 0$ and $h \ge 3$ satisfy the conditions in Assumption~\ref{assmp:stepsize}. Then, under Assumption~\ref{assmp:linearSA}, and for $x_0 \ne x^*$, we have the following for the iterates of Algorithm~\ref{eq:alglinearSA}.
    \begin{enumerate}
        \item For $z = 1$ and $\delta \in (0,1)$,
        \begin{align*}
         &\mathbb{P}\lrp{\exists k \in \N:~\| x_k - x^*\|^2_c >  \frac{1}{k}{\lrp{\log\left(\frac{1}{\delta}\right)}^{2\alpha(A_{13} - 1)}} \mathcal{O}\lrp{\log\left(\frac{1}{\delta}\right) + \log k}} \le \delta.
        \end{align*}
        \item For $z\in(0,1)$, $\beta < 1$, and $\delta\in(0,1]$, 
        \begin{align*}
         \mathbb{P}\lrp{ \exists k \in \N: ~\|{x}_k - x^*\|^2_c > \mathcal{O}\lrp{\frac{e^{\lrp{\log\left(\frac{1}{\delta}\right)}^\frac{1-z}{ \beta z}}}{(k+h-1)^z}}} \le \delta. 
        \end{align*}
    \end{enumerate}
\end{theorem}
\end{comment}

\subsubsection{Polyak-Ruppert Averaging}
A key technique to improve the asymptotic variance of the iterates of a SA algorithms is Polyak-Ruppert averaging \citep{polyak1992acceleration}, where, given the iterates defined by Equation~\eqref{eq:SAalgo}, one computes the averaged sequence
\[ y_k := \frac{1}{k} \sum\limits_{i=0}^{k-1} x_i. \]

Polyak-Ruppert averaging has been shown to yield the optimal asymptotic covariance for SA algorithms in various settings. However, until recently, its non-asymptotic high-probability behavior remained poorly understood. Recently, \citet[Theorem 4.1]{khodadadian2025general} have shown that Polyak averaging preserves the high-probability guarantees of the unaveraged iterates up to constant factors, when the underlying noise is a Martingale-difference sequence. Since i.i.d.\ noise is a special case of Markov noise as well as of Martingale-difference noise, combining their result (Corollary 4.2) with our bounds for step sizes with $z < 1$ in Theorem \ref{th:ballargumentboundz}, we get the following corollary for the Polyak-averaged iterates of SA with i.i.d.\ noise.

\begin{corollary} 
    For all $k\in\N$, let $\alpha_k = \alpha/(k+h)^z$ for some $z\in(0,1)$, where $\alpha$ and $h$ satisfy the conditions in Assumption~\ref{assmp:stepsize}. Under assumptions~\ref{assmp:boundedmultnoise},~\ref{assmp:boundedaddnoise},~\ref{assmp:contraction}, and i.i.d.\ noise sequences $\{S_k\}_{k\geq 0}$ and $\{Z_k\}_{k\geq 0}$, for $x_0 \ne x^*$, we have the following: for  $\delta\in(0,1)$ and $k\geq 0$, with probability at least $1-\delta$, %we have
    \begin{align*}
        \|y_k - x^*\|_c^2 \leq \mathcal{O}\left(\frac{\log\left(\frac{1}{\delta}\right)}{k}\right) + \mathcal{O}\left(\frac{f_z(\delta)^2}{k^{\min\{2z,2-z\}}}\right),
    \end{align*}
    where
    \begin{equation*}
        f_z(\delta) = \begin{cases}
            \log\left(\frac{1}{\delta}\right), & \text{ if } D < 0, \\
            \log\left(\frac{1}{\delta}\right)^{1/z}, & \text{ if } D = 0, \\
            e^{\log\left(\frac{1}{\delta}\right)^\frac{1-z}{ \beta z}}, \quad \text{for } \beta\in(0,1), & \text{ if } D > 0.
        \end{cases}
    \end{equation*}
\end{corollary}

From the above corollary, we see that although Polyak averaging is known to achieve the desirable asymptotic variance, its squared error exhibits significantly heavier tail than that of the unaveraged algorithm with $z = 1$, especially when the operator is expansive with positive probability $(D > 0)$. Hence, in such expansive settings, Polyak averaging may, in fact, be an undesirable choice. A natural direction for future research is to design algorithms that, in addition to achieving the optimal asymptotic variance, improve tail behavior.

\begin{comment}
\subsubsection{Other Special Cases {\color{red}[Skip this for now]}}
\begin{enumerate}
    \item Dissipativity -- see Chapter 3 of Zaiwei's \href{https://repository.gatech.edu/server/api/core/bitstreams/224cfe46-7014-41ed-85dd-2d82a393d62d/content}{thesis}. 
    \item Optimization becomes special case of above.
    \item Semi-norm contraction. -- leave this for now.
\end{enumerate}
\end{comment}

\subsection{Proof sketch for theorems~\ref{th:ballargumentbound} and~\ref{th:ballargumentboundz}}\label{sec:proof_sketch_12}
In this section, we provide a proof sketch of theorems~\ref{th:ballargumentbound} and~\ref{th:ballargumentboundz}. We first establish the result in the simpler regime $D < 0$, and subsequently, extend the argument to the remaining cases.

\begin{itemize}
    \item {\bf Case 1 ($D < 0$).}  We first show that, if the squared error $\|x_k-x^*\|^2_c$ of the iterates generated by Equation~\eqref{eq:SAalgo} is almost surely upper bounded by a non-decreasing sequence $T_k$, then the MGF of $\alpha_k^{-1}T_k^{-1} \|x_k-x^*\|^2_c$ is upper bounded by a constant. This requires the careful bounding of the drift of this expression plus a correction term based on the solution to the Poisson equation of the Markov chain of the noise. 
    
    From the constant bound on the MGF of $\alpha_k^{-1}T_k^{-1} \|x_k-x^*\|^2_c$ it follows that, for any $\delta \in (0,1)$, the squared error $\|x_k-x^*\|^2_c$ is upper bounded by a function of order $\tilde{\mathcal{O}}(T_k\alpha_k\log(1/\delta))$ with probability at least $1-\delta$. Since $\alpha_k$ is of order $1/k^z$, this high-probability bound improves upon the original almost-sure bound $T_k$ by a factor of order $1/k^{z}$. This is given in Theorem~\ref{th:onestepbootstrap} in Appendix~\ref{app:onestepconc}. 
    
    Thus, to obtain a high-probability bound, it suffices to identify a non-decreasing almost-sure bound $T_k$. Lemma~\ref{lem:worst_case_bound_original} provides such a worst-case bound on $\|x_k-x^*\|^2_c$, which holds almost surely and it is constant in $k$ when $D < 0$. % is nondecreasing in $k$. When $A_{13} - 1 < 0$, this bound is a constant (independent of $k$). For $z = 1$, it scales as $\mathcal{O}(k^{2\alpha(A_{13} - 1)})$ when $A_{13} - 1 > 0$ and as $\mathcal{O}(\log^2(k))$ when $A_{13} - 1 = 0$. For $z < 1$, it grows as  $\mathcal{O}(e^{k^{1-z}})$ when $A_{13} - 1 > 0$, and as $\mathcal{O}(k^{2(1-z)})$ when $A_{13} - 1 = 0$. 
    Then, applying Theorem~\ref{th:onestepbootstrap} with this almost-sure bound as the initial sequence $T_k$ yields the bounds in Theorem~\ref{th:ballargumentbound}, Part~\ref{ball_part1_cont}, and Theorem~\ref{th:ballargumentboundz}, Part~\ref{ball_partz_cont}. 

    \item {\bf Case 2 ($D \ge 0$).} 
    In this case, the almost sure upper bounds for the iterates are not constant in $k$ as in the case $D < 0$. For $z = 1$, the almost sure upper bound scales as $\mathcal{O}(k^{2\alpha D})$ when $D > 0$ and as $\mathcal{O}(\log(k)^2)$ when $D = 0$. For $z < 1$, it grows as  $\mathcal{O}(e^{k^{1-z}})$ when $D > 0$, and as $\mathcal{O}(k^{2(1-z)})$ when $D = 0$. To overcome this, we establish our high-probability bounds in two steps.
    \begin{enumerate}
        \item {\bf Step 1 (Auxiliary algorithm).} We consider an auxiliary SA algorithm that projects its iterates at each time step onto a ball centered at $x^*$ with a radius $B>0$, to be determined later. Therefore, by construction, the mean square error of the new iterates is almost surely bounded by the radius $B$, for all $k\geq 0$. Applying Theorem~\ref{th:onestepbootstrap} (from Case 1 above) with the constant sequence $T_k = B^2$ yields a bound of the form $\tilde{\mathcal{O}}(B^2 \alpha_k \log(1/\delta))$ for the squared error of the auxiliary iterates, holding with probability at least $1-\delta$. 
        \item {\bf Step 2 (Auxiliary and original iterates coincide).} 
        We now relate the auxiliary iterates to the original iterates for an appropriate choice of $B$. To this end we show that, for any $\delta\in(0,1)$ we can choose the radius $B$ sufficiently large so that the projection step is never activated with probability at least $1-\delta$. Hence, the auxiliary iterates coincide with the original iterates of Equation~\eqref{eq:SAalgo} with probability $1-\delta$. Consequently, the high probability bound obtained in Step 1 also holds for the original algorithm. This yields the concentration results stated in Theorem~\ref{th:ballargumentbound}, Part~\ref{ball_part1_exp} and Theorem~\ref{th:ballargumentboundz}, parts~\ref{ball_partz_equal} and~\ref{ball_partz_exp}. %We prove the bounds for $z < 1$ similarly. 

    \end{enumerate}
\end{itemize}

We introduce the auxiliary algorithm and establish its properties in Appendix~\ref{app:auxalgo}. The proof of Theorem~\ref{th:ballargumentbound} (Part~\ref{ball_part1_exp}) and Theorem~\ref{th:ballargumentboundz} (parts~\ref{ball_partz_equal} and~\ref{ball_partz_exp}), appears in Appendix~\ref{app:proof_th:ball_part1}. We refer the reader to Appendix~\ref{app:onestepconc} for the proof of Case~1 and to Appendix~\ref{app:proofs:sec:boundednoise} for the proof of Case~2.

\section{Concentration for contractive operators with unbounded identically distributed noise}\label{sec:generalcontractive}
In this section we present a high-probability bound for the error $\|\xk - x^*\|_c$ for the case where step sizes are of order $\mathcal{O}(1/k)$, and where Assumption \ref{assmp:boundedaddnoise} on the additive noise $\{Z_k\}_{k \geq 0}$ is relaxed to include arbitrary unbounded distributions, at the cost of strengthening Assumption \ref{assmp:lyapunov} on the average operator $\bar{F}$ to the case where it is a contraction (i.e., when $\bar{F}$ satisfies Assumption~\ref{assmp:contraction}).
%Recall that in Section~\ref{sec:boundednoise} we established a high-probability bound for the error $\|\xk - x^*\|_c$ which required the noise to be almost surely bounded, as in assumptions~\ref{assmp:boundedmultnoise} and \ref{assmp:boundedaddnoise}. %for the noise, together with Assumptions~\ref{assmp:barf}–\ref{assmp:lyapunov} for the averaged operator $\bar{F}$ and the underlying Markov chain. Section~\ref{sec:contractiveSA} then specialized this analysis to the case where $\bar{F}$ is contractive, yielding a concentration bound for the iterates of Algorithm~\ref{eq:SAalgo} under Assumption~\ref{assmp:contraction} and the same bounded-noise conditions from Assumptions~\ref{assmp:boundedmultnoise}–\ref{assmp:boundedaddnoise}.
In particular, instead of Assumption \ref{assmp:boundedaddnoise}, we assume the following.

\begin{assmp}[Generalized additive noise]\label{assmp:addnoise} We have that $\{Z_k\}_{k\geq 0}$ are identically distributed, with $\mathbb{E}[Z_k\mid \mathcal{F}_{k-1}] = 0$ for all $k\geq 0$. Furthermore, there exists a non-negative random variable $W$ such that
\[ \mathbb{P}\lrp{ \|Z_0\|_c > z } \le \mathbb{P}\lrp{W > z}, \quad \forall z\in\R.  \]
\end{assmp}

Under this relaxed assumption, we derive high-probability bounds for $\|\xk - x^*\|_c$ when the averaged operator remains contractive, with the bounds expressed directly in terms of the tail behavior of $\|Z_k\|_c$. The resulting Theorem~\ref{th:generalnoiseconc} thus substantially generalizes Theorem~\ref{th:contractiveconc} for contractive average operators, beyond the bounded-noise regime.

%\todoS{Applications for motivating general noise: SGD with heavy-tail noises, any queuing applications (?), sub-Gaussian tails also not covered earlier -- so motivate saying that noise with unbounded support is important somehow.}

\subsection{Main results}\label{sec:results_generalnoise}
We now present our main concentration result for this section. First, we introduce some notation. For any $\delta' \in (0,1)$, let $B(\delta') \in \R$ denote the quantile satisfying 
\[\mathbb{P}\lrp{W \le B(\delta')} = 1-\delta'. \numberthis\label{eq:quantile} \]
We define a function $g: \R_+ \to \R_+$ such that
\[ g(\delta'):= \frac{ \E{W {\bf 1}_{\{W \ge B(\delta')\}}} }{1-\gamma_c}, \] 
where $\gamma_c \in (0,1)$ is the contraction factor from Assumption~\ref{assmp:contraction}. Note that Assumption \ref{assmp:addnoise} and Jensen's inequality imply that
\[ g(\delta') \ge \frac{ \left\| \E{Z_0 {\bf 1}_{\{\|Z_0\|_c \ge B(\delta')\}}} \right\|_c }{1-\gamma_c}. \numberthis\label{eq:bias} \] 
%Let $A_{13} := A_1 + A_3$, and for $T\in\N$ and $\delta \in (0,1)$, define the functions $\bar{a}_1(\cdot,\cdot)$, $\bar{b}_1(\cdot,\cdot)$, $B_{23}(\cdot,\cdot)$, $\bar{b}_4(\cdot, \cdot)$, and $C_1(\cdot, \cdot)$ as given in Appendix~\ref{app:const_general_contractive}. These play the same roles as the constants in Theorem~\ref{th:contractiveconc}, but in the present setting, they depend explicitly on $\delta$ and $T$ rather than being fixed constants. Their explicit forms are provided in Section~\ref{app:const_general_contractive}.
Moreover, we have the following assumption on the step sizes $\alpha_k=\alpha/(k+h)$.

\begin{assmp}[Conditions on tunable parameters $\alpha$ and $h$]\label{assmp:stepsize_generalnoise} We have $\alpha> 2/\eta$, and
\[ h \geq \max\lrset{ \frac{4 C_1}{\eta\alpha}, \frac{\eta}{2\alpha}, \frac{u L_s (L_1 + L_2)}{\lcs^2\alpha} , 8}.\]
\end{assmp}

%Since the bound on $\|x_k - x^*\|^2_c$ for $A_{13} \le 1$ can be obtained from the case $A_{13} > 1$ by setting $A_{13} = 1$ (up to a multiplicative $\log(k)$ factor), we present below only the concentration result for $A_{13} > 1$ for clarity of exposition; see also Remark~\ref{rem:reduction}.

\begin{theorem}\label{th:generalnoiseconc}
    Suppose that we have step sizes $\alpha_k=\alpha/(k+h)$, where $\alpha$ and $h$ satisfy Assumption~\ref{assmp:stepsize_generalnoise}. Under assumptions~\ref{assmp:boundedmultnoise},~\ref{assmp:contraction} and~\ref{assmp:addnoise}, for $x_0 \ne x^*$ and for $T\in \N$, the following hold. 
    \begin{enumerate}
        \item Suppose $D < 0$. For any $\delta\in(0,1)$, the following holds with probability at least $1-\delta$: for all $0 \le k\le T$,   
        \begin{align*}
            \| x_k - x^*\|^2_c 
            &\leq \frac{\bar{a}_1(\delta, T)}{k+h-1} \bigg[\bar{b}_1(\delta, T)\log\left(\tfrac{2}{\delta}\right) + B_{23}(\delta, T) + \bar{b}_4(\delta, T) \log\left(\tfrac{k+h-1}{h}\right)\bigg]  + 2 g^2\lrp{\tfrac{\delta}{2T}}.
        \end{align*}
        \item Suppose $D = 0$. For any $\delta\in(0,1)$, the following holds with probability at least $1-\delta$: for all $0 \le k\le T$,   
        \begin{align*}
            \| x_k - x^*\|^2_c 
            &\leq \frac{\bar{a}_1(\delta,T) \log(\tfrac{k-1+h}{h-1})^2}{k+h-1} \bigg[\bar{b}_1(\delta, T)\log\left(\tfrac{2}{\delta}\right) + {B}_{23}(\delta, T)  + \bar{b}_4(\delta, T) \log(\tfrac{k+h-1}{h})\bigg] + 2 g^2\lrp{\tfrac{\delta}{2T}}.
        \end{align*}
        \item Suppose $D > 0$. For any $\delta \in (0,1)$, the following holds with probability at least $1-\delta$: for all $0 \le k \le T$, 
        \begin{align*}
            \| x_k - x^*\|^2_c 
            &\leq \frac{\bar{a}_1(\delta, T) \lrs{\bar{b}_1(\delta, T)\log\left(\tfrac{2}{\delta}\right) + B_{23}(\delta, T)}^{2\alpha  D}}{k+h-1} \bigg[\bar{b}_1(\delta, T)\log\left(\tfrac{2}{\delta}\right)   + B_{23}(\delta, T) \\ 
            &\qquad \qquad \qquad \qquad \qquad \qquad \qquad \qquad \qquad \qquad + \bar{b}_4(\delta, T) \log(\tfrac{k+h-1}{h})\bigg]  + 2 g^2\lrp{\tfrac{\delta}{2T}}.
        \end{align*}
    \end{enumerate}
    {Here, $\bar{a}_1= \mathcal O(B^2(\delta/T))$, $\bar{b}_1=\mathcal O(B^2(\delta/T))$, $B_{23}=\mathcal O(B(\delta/T))$, $\bar{b}_4=\mathcal O(B^2(\delta/T))$, and $C_1=\mathcal O(B(\delta/T))$, where the explicit forms of these are given in Appendix~\ref{app:const_general_contractive}.}
\end{theorem}

\begin{comment}
\begin{theorem}\label{th:generalnoiseconc}
    For all $k\in\N$, let $\alpha_k = \alpha/(k+h)$, where $\alpha > 0$, $h \ge 3$, and $\delta \in (0,1)$ together satisfy the conditions in Assumption~\ref{assmp:stepsize}. Then, under Assumptions~\ref{assmp:contraction} and~\ref{assmp:addnoise}, for $x_0 \ne x^*$ and for $T\in \N$, the following holds with probability at least $1-\delta$, when $A_{13}  > 1$: for all $0 \le k \le T$, 
        \begin{align*}
            &\| x_k - x^*\|^2_c 
            \leq \min\left\{ \bar{a}_1(\delta/2, T)\lrs{\bar{b}_1(\delta/2, T)\log\frac{2}{\delta} + \bar{b}_2(\delta/2, T) + \bar{b}_3(\delta/2, T)}^{2\alpha (A_{13}-1)}, \right. \\
            &~~\left. \frac{\bar{a}_1(\frac{\delta}{2}, T) \lrs{\bar{b}_1(\frac{\delta}{2}, T)\log\frac{2}{\delta} + \bar{b}_2(\frac{\delta}{2}, T) + \bar{b}_3(\frac{\delta}{2}, T)}^{2\alpha (A_{13}-1)}}{k+h-1} \bigg[\bar{b}_1(\delta/2, T)\log\frac{2}{\delta} + \bar{b}_2(\delta/2, T)\right. \\ 
            &\qquad \qquad \qquad \qquad \left. + \bar{b}_3(\delta/2, T) + \bar{b}_4(\delta/2, T) \log\left(\frac{k+h-1}{h}\right)\bigg] \right\} + 2 g^2\lrp{\frac{\delta}{2T}}.
        \end{align*}
\end{theorem}
\end{comment}
The proof is given in Section \ref{sec:proof_th:generalnoiseconc}.

%\todo[inline]{Siva said dependence on $T$ is unclear. So asked to make the above change. Can we also get $g$ in terms of $B$? \\ Briefly explain the maximal here since this is different from the usual maximal. }

\medskip

The tail behavior of the random variable $W$ is implicit in the bound above in terms of the function $g(\cdot)$, and the terms $\bar{a}_1$, $\bar{b}_1$, $B_{23}$, and $\bar{b}_4$, which in turn depend on the quantile function $B(\cdot)$. %In particular, the functions $g$ and $B(\cdot)$ depend on the behavior of the tail of $W$. 
If $W$ is heavy-tailed, then $B(\cdot)$ is large. Similarly, one can see that as the tails become heavier, the bias between the fixed points of the true and the auxiliary algorithm,  captured by $g(\cdot)$, becomes larger. Note that the concentration bound is for any time $k$ between $0$ and $T$. In particular, setting $T=k$ yields finite-time concentration bounds that only depend on $k$. To make these observations concrete, we illustrate the bound in Theorem~\ref{th:generalnoiseconc} for various choices of the tail distribution $W$ in the following section.

\subsubsection{Examples}\label{sec:examples}

To illustrate the bound in Theorem~\ref{th:generalnoiseconc}, in this section we instantiate it for various choices of the tail distribution $W$. Note that, for readability, we only present the bounds for a fixed $k$ in the examples below. We postpone the details of computation for the quantile function $B(\cdot)$ and the bound $g(\cdot)$ from Equation~\eqref{eq:bias} in each example to Appendix~\ref{app:boundbias} (lemmas~\ref{lem:subweibull} and~\ref{lem:subpareto}). 
\begin{corollary}\label{cor:subweibull}
    Suppose that $W$ is a non-negative sub-Weibull random variable such that
    \[ \mathbb{P}\lrp{ W \ge x} \le p e^{-q x^\frac{1}{\theta}}, \quad \forall \, x \ge 0, \] 
    for some $\theta \geq 1/2$ and $p,q> 0$. Then, %for any $\gamma > 0$, we have $B(\gamma) = \lrp{\frac{1}{q} \log\left(\frac{p}{\gamma}\right) }^\theta$ and $g(\gamma) = \frac{c \gamma}{(1-\gamma_c) }(\log  \frac{1}{\gamma})^{\theta-1}$, where $c$ is a non-negative constant. 
    %Thus, for $k\in N$, we have
    %\begin{align*}
    %    \mathbb{P}\lrp{ \|x_k - x^*\|^2_c \le {\mathcal{O}}\lrp{ \frac{\log\left(\frac{k}{\delta}\right)^{4\theta\alpha(A_{13}-1)}\log\left(\frac{k}{\delta}\right)^{2\theta} \lrp{\log\left(\frac{1}{\delta}\right) + \log(k)\log\left(\frac{k}{\delta}\right)^{2\theta}  } }{k} + \frac{\theta^2\delta^2}{k^2} }  } \ge 1-\delta. 
    %\end{align*}
    %Moreover, 
    for $k\geq 0$, the following hold.    
    \begin{enumerate}\item\label{gauss_part1} Suppose $D < 0$. For any $\delta\in(0,1)$, with probability at least $1-\delta$, we have
    \begin{align*}
        \|x_k - x^*\|^2_c \le {\mathcal{O}}\lrp{ \frac{\log\left(\frac{k}{\delta}\right)^{4\theta+1}  }{k} + \frac{\delta^2}{k^2} \log\left(\frac{k}{\delta}\right)^{2\theta-2}}.
    \end{align*}
    \item\label{gauss_part2} Suppose $D = 0$. For any $\delta\in(0,1)$, with probability at least $1-\delta$, we have
    \begin{align*}
        \|x_k - x^*\|^2_c \le {\mathcal{O}}\lrp{ \frac{\log\left(\frac{k}{\delta}\right)^{4\theta + 1} \log(k)^2 }{k} + \frac{\delta^2}{k^2} \log\left(\frac{k}{\delta}\right)^{2\theta-2}}.
    \end{align*}
    \item \label{gauss_part3} Suppose $D > 0$. For any $\delta\in(0,1)$, with probability at least $1-\delta$, we have
    \begin{align*}
        \|x_k - x^*\|^2_c \le {\mathcal{O}}\lrp{ \frac{\lrp{\log\left(\frac{k}{\delta}\right)^{2\theta} \log\left(\frac{1}{\delta}\right)}^{2\alpha D}\log\left(\frac{k}{\delta}\right)^{4\theta+1}  }{k} + \frac{\delta^2}{k^2} \log\left(\frac{k}{\delta}\right)^{2\theta-2}}.
    \end{align*}
    \end{enumerate}
\end{corollary}

Note that, when the operator is almost surely a contraction (i.e., when $D<0$), Part \ref{gauss_part1} of Corollary~\ref{cor:subweibull} implies that the error is of order $\mathcal{O}(\log(1/\delta)^{2\theta+1/2})$, as $\delta\to 0$. That is, the tail of the error is sub-Weibull with parameter $2\theta+1/2$. Since the input noise is sub-Weibull with parameter $\theta \geq 1/2$, then the tail of the error is between two and three times heavier than the tail of the input noise. However, in this regime, we expect the tail of the error to match that of the input noise. %Our results instead show that the error tails are, in the worst case, at most three times heavier. 
Tightening these concentration bounds remains an interesting direction for future work.

%In the sub-Gaussian setting (i.e., when $\theta=1/2$), Part \ref{gauss_part1} of Corollary~\ref{cor:subweibull} implies that, for any $k \geq 0$, we have
%\[ {\|\xk - x^*\|^2_c \le \mathcal{O}\lrp{ \frac{\log\left(\frac{k}{\delta}\right)^3}{k} }} \]
%with probability at least $1-\delta$. For this special case, using a direct argument strongly tailored to sub-Gaussian noise with almost sure contractive operators, \citep[Theorem~2.4]{chen2025concentration} states a high-probability bound for the squared error of order $\mathcal{O}(\log(1/\delta))$, as $\delta\to 0$. However, their proof does not generalize to broader noise classes, while ours does. This is further showcased in the following corollary.

%\begin{remark}
    %A bound for the case where $W$ is sub-Gaussian is obtained by setting $\theta = 1/2$ in Corollary~\ref{cor:subweibull}.
%\end{remark}

\begin{corollary}\label{cor:pareto}
    Suppose that $W$ is a sub-Pareto random variable such that
    \[ \mathbb{P}\lrp{W \ge x} \le \frac{p}{x^\theta}, \quad \forall \, x \ge 0, \] 
    for some $\theta \ge 2$, and $p > 0$.
    Then, %for any $\gamma > 0$, $B(\gamma) = \lrp{\frac{p}{\gamma}}^\frac{1}{\theta}$ and $g(\gamma) = \frac{\gamma^\frac{\theta - 1}{\theta}}{1-\gamma_c} \frac{c p^\frac{1}{\theta}}{\theta}$, for a non-negative constant $c$. Moreover, 
    for $k \geq 0$, the following hold.
    \begin{enumerate}
    \item \label{pareto_part1} Suppose $D < 0$. For any $\delta\in(0,1)$, with probability at least $1-\delta$, we have
    \begin{align*}
        \|x_k - x^*\|^2_c \le {\mathcal{O}}\lrp{ \frac{\lrp{\frac{k}{\delta}}^{\frac{4}{\theta}}\log\left(\frac{k}{\delta}\right)   }{k} + \lrp{\frac{\delta}{k}}^\frac{2\theta - 2}{\theta} } .
    \end{align*}
    \item \label{pareto_part2} Suppose $D = 0$. For any $\delta\in(0,1)$, with probability at least $1-\delta$, we have 
    \begin{align*}
        \|x_k - x^*\|^2_c \le {\mathcal{O}}\lrp{ \frac{\lrp{\frac{k}{\delta}}^{\frac{4}{\theta}}\log\left(\frac{k}{\delta}\right) \log(k)^2   }{k} + \lrp{\frac{\delta}{k}}^\frac{2\theta - 2}{\theta} }.
    \end{align*}
    \item\label{pareto_part3} Suppose $D > 0$. For any $\delta\in(0,1)$, with probability at least $1-\delta$, we have 
    \begin{align*}
        \|x_k - x^*\|^2_c \le {\mathcal{O}}\lrp{ \frac{\lrp{\lrp{\frac{k}{\delta}}^\frac{2}{\theta} \log\left(\frac{1}{\delta}\right)}^{2\alpha D}\lrp{\frac{k}{\delta}}^{\frac{4}{\theta}}\log\left(\frac{k}{\delta}\right)   }{k} + \lrp{\frac{\delta}{k}}^\frac{2\theta - 2}{\theta} }.
    \end{align*}
    \end{enumerate}
\end{corollary}

When the random operator is almost surely contractive, i.e., when $D < 0$, Part~\ref{pareto_part1} of Corollary~\ref{cor:pareto} implies that the squared error is of order $\tilde{\mathcal{O}}\big(k^{4/\theta-1}\big)$. In this case, our upper bound converges to zero as $k\rightarrow \infty$ only when the additive noise is sufficiently light-tailed, with $\theta > 4$. On the other hand, when the random operator is expansive with positive probability (i.e., when $D > 0$), Part~\ref{pareto_part3} of Corollary~\ref{cor:pareto} implies that the error is of order $\tilde{\mathcal{O}}\big(k^{4[1+\alpha D]/\theta-1}\big)$. In this case the required tail condition for the convergence of the upper bound becomes even more stringent: it converges to zero as $k\rightarrow \infty$ only when $\theta > 4[1+\alpha D]$. In any case, we obtain a sub-optimal convergence rate slower than $\mathcal{O}(1/k)$.

\subsection{Proof of Theorem~\ref{th:generalnoiseconc}}\label{sec:proof_th:generalnoiseconc}

We prove the concentration result in Theorem~\ref{th:generalnoiseconc} in two steps. In the first step, we consider an auxiliary algorithm where the sequence $\{Z_k\}_{k\geq 0}$ of additive noise is truncated at an appropriate level $B>0$, so that the norm of the truncated noise is almost surely upper bounded by $B$. Since the additive noise is not necessarily symmetric around $0$, this truncation might introduce a bias, and the auxiliary algorithm might converge to the fixed point $\tilde{x}$ of a different contractive operator.  Moreover, since the noise of the auxiliary algorithm is almost surely bounded, Theorem~\ref{th:contractiveconc} yield a high-probability bound for the auxiliary iterates around $\tilde{x}$. 

As a second step, we relate the auxiliary iterates to the original iterates for an appropriate choice of the truncation level $B$. To this end we show that, for any $\delta\in(0,1)$ we can choose the truncation level $B$ sufficiently large so that the truncation is never activated with probability at least $1-\delta$, at least for the first $T$ steps.  Hence, the auxiliary iterates coincide with the original iterates of Equation~\eqref{eq:SAalgo} with probability $1-\delta$, up to time $T$. Moreover, we show that the distance between the two fixed points $x^*$ and $\tilde{x}$ is small (bounded by quantities dependent on the noise distribution). Consequently, the high probability bound obtained in the first step also holds for the original algorithm, with some adjustments for the introduced bias. This yields the concentration results stated in Theorem~\ref{th:generalnoiseconc}.

%we show that the to arrive at a concentration result for the original iterates, we analyze them on a good set, which informally consists of all the sample paths on which the norm of the additive noise at each time step is bounded. On such a set, the original and auxiliary iterates coincide, and concentrate around the auxiliary fixed point $\tilde{x}$. Recall from the Step 1 that the $x^*$ and $\tilde{x}$ are not far apart. Hence, we conclude that the original iterates concentrate with high probability on the good set. We further show that the probability of the complement of the good set can be bounded in terms of the tail of the random variable $W$, thus contributing only a negligible amount to the probability that the original iterates do not concentrate. 

%While the auxiliary algorithm may not be implementable, we only introduce it to analyze the original algorithm's performance. 

%To prove the concentration result in this section, we show a black-box reduction to the setup of Theorem~\ref{th:ballargumentbound}. The tail bound then follows from that in Theorem~\ref{th:ballargumentbound} with minor adjustments. For the said black-box reduction, we introduce an auxiliary algorithm below. 

\begin{remark}
    We use the contractive property of the average operator (Assumption~\ref{assmp:contraction}) only to bound the bias introduced in the auxiliary iterates due to truncation of the additive noise. If the additive noise at each step is symmetric around $0$, no bias is introduced in the auxiliary iterations, and its fixed point $\tilde{x}$ coincides with the original fixed point $x^*$. In this case, we don't need Assumption~\ref{assmp:contraction}. Instead, assumptions~\ref{assmp:barf} and~\ref{assmp:lyapunov} suffice, and the auxiliary iterates actually concentrate around the true fixed point~$x^*$.
\end{remark}

We now present the two steps of the proof in detail, introducing the auxiliary algorithm, bounding the bias in the fixed point arising from the noise truncation, and showing that the original and auxiliary iterates coincide with high probability.

\subsubsection{Step 1: The Auxiliary Algorithm and its Properties}\label{sec:auxalgotrunc}

We first introduce the auxiliary iterates that result from truncating the additive noise sequence at an appropriate level. Recall from Section~\ref{sec:results_generalnoise} that $B(\cdot)$ represents a quantile for $W$. From Assumption~\ref{assmp:addnoise}, it follows that for any $\delta'> 0$ and $k\geq 0$, we have  $\mathbb{P}\lrp{ \|Z_k\|_c \le B(\delta') } \ge 1-\delta'$. For any time $T > 0$, we define the collection of sample paths where the additive noise is uniformly bounded as
\[ \mathcal B_T(\delta') := \lrset{\|Z_k\|_c \le B({\tfrac{\delta'}{T}}),\,\, \forall \, k \leq T-1}. \numberthis\label{eq:BT}\]
Then,
\begin{align*} 
    \mathbb{P}\lrp{\mathcal B^c_T(\delta')} = \mathbb{P}\lrp{ \exists \, k \in \lrset{0,1,\dots, T-1}: \|Z_k\|_c > B ( \tfrac{\delta'}{T} ) } &\le \sum\limits_{k=0}^{T-1}~\mathbb{P}\lrp{ \|Z_k\|_c > B(\tfrac{\delta'}{T}) } \le \delta' \numberthis\label{eq:probbcomp}.  
\end{align*}
We consider the sequence of random vectors $\{\tilde{Z}_k\}_{k\geq 0}$ given by
\[ \tilde{Z}_k := Z_k {\bf 1}_{\{\|Z_k\|_c \le B(\frac{\delta'}{T})\}} - \mathbb{E}\left[Z_0 {\bf 1}_{\{\|Z_0\|_c \le B(\frac{\delta'}{T})\}}\right]. \]
Note that, since $\{Z_k\}_{k\geq 0}$ are identically distributed (cf. Assumption \ref{assmp:addnoise}), then the random vectors $\{\tilde{Z}_k\}_{k\geq 0}$ have zero mean. For $x\in\mathbb{R}^d$ and $s\in\mathcal{S}$, we define the operator
\[ H(\tilde{x},s) := F(\tilde{x}, s) + \E{Z_0 {\bf 1}_{\{\|Z_0\|_c \le B(\frac{\delta'}{T})\}}}, \]
%Then, the noisy operator can be rewritten as
%\begin{align*} 
%    &H(\tilde{x}, S_k) + Z_k {\bf 1}_{\{\|Z_k\|_c \le B\lrp{\frac{\delta'}{T}\}}} - \E{Z {\bf 1}_{\{\|Z\|_c \le B\lrp{\frac{\delta'}{T}}\}}},
%\end{align*}
and the corresponding averaged operator
\[\bar{H}(\tilde{x}) := \mathbb{E}[H(\tilde{x}, S)] = \bar{F}(\tilde{x}) + \E{Z_0 {\bf 1}_{\{\|Z_0\|_c \le B(\frac{\delta'}{T})\}}}, \] 
where $S\sim\pi$. For $k\geq 0$, let $\tilde{x}_k$ be the iterates defined by the stochastic recursion
\[ \tilde{x}_{k+1} = \tilde{x}_k + \alpha_k\lrp{H(\tilde{x}_k, S_k) + \tilde{Z}_k - \tilde{x}_k }.\]
%While this auxiliary algorithm may not be implementable, it is only introduced as a proof technique.
Let $\tilde{x}\in\mathbb{R}^d$ be the point that solves the fixed point equation
\[ %\E{F(\tilde{x},S) + Z {\bf 1}_{\{\|Z\|_c \le B(\frac{\delta'}{T}) \}}} = 
\bar{H}(\tilde{x}) = \tilde{x}. \numberthis\label{eq:auxfixedpoint}  \]
%Under appropriate assumptions, the auxiliary algorithm introduced above solves for the fixed point $\tilde{x}$ \citep{borkar2008stochastic}.

\begin{lemma}\label{lem:assmpcheck}
The operators $H$ and $\bar{H}$ and the noise sequences $\{S_k\}_{k\geq 0}$ and $\{\tilde{Z}_k\}_{k\geq 0}$ satisfy assumptions~\ref{assmp:boundedmultnoise},~\ref{assmp:boundedaddnoise}, and~\ref{assmp:contraction} with $B_2 = 2B(\delta'/T)$.
\end{lemma}
The proof is given in Appendix~\ref{app:proof_lem_assmpcheck}.

\medskip

%We will show that operators $H$, $\bar{H}$, noises $S$, and $Z {\bf 1}_{\|Z\|_c \le B(\delta'/T)} - \mathbb{E}[Z {\bf 1}_{\|Z\|_c \le B(\delta'/T)}]$ satisfy Assumptions~\ref{assmp:boundedmultnoise},~\ref{assmp:boundedaddnoise}, and~\ref{assmp:contraction}. 
Lemma~\ref{lem:assmpcheck} above shows that the iterates $\{\tilde{x}_k\}_{k\geq 0}$ and the underlying operator and noise sequences satisfy the assumptions in Theorem~\ref{th:contractiveconc}. Hence, the error $\|\tilde{x}_k - \tilde{x}\|_c$ satisfies the high-probability bound of Theorem~\ref{th:contractiveconc}, with the constant $B_2$ in that theorem being $2B(\delta'/T)$. In particular, we have that
\begin{equation}\label{eq:aux_hpb}
    \mathbb{P}\Big(\|\tilde{x}_k-\tilde{x}\|_c^2 \leq f(k,\delta,\delta'), \,\, \forall\, k \geq 0 \Big) \geq 1-\delta,
\end{equation}
for some appropriate function $f$.

\subsubsection{Step 2: Overall Tail Bound for Original Iterates} 

We first show that the difference between the auxiliary fixed point $\tilde{x}$ and the original fixed point $x^*$ is bounded as follows.

\begin{lemma}\label{lem:biasboundaux}
The fixed points $x^*$ and $\tilde{x}$ satisfy 
\begin{align*}
    \|x^* - \tilde{x}\|_c
    &\le \frac{\left\|\E{Z_0 {\bf 1}_{\{\|Z_0\|_c > B(\frac{\delta'}{T})\}}} \right\|_c}{1-\gamma_c}. \numberthis\label{eq:biasofaux} 
\end{align*}
\end{lemma}
This result follows from the application of the triangle inequality, and Assumption~\ref{assmp:contraction}. We refer the reader to Appendix~\ref{app:proof_lem:biasboundaux} for a complete proof.

\medskip

Recall from Section~\ref{sec:results_generalnoise} that $g(\delta'/T)$ is an upper bound on the right-hand side of Equation~\eqref{eq:biasofaux} (see Appendix~\ref{app:boundbias} for examples of such an upper bound for different choices of $W$). We are now ready to bound the overall probability that the squared error $\|\xk-x^*\|^2_c$ is large.

\begin{lemma}\label{lem:probboundtrue_aux_bad}
    For the non-negative function $f$ in Equation \eqref{eq:aux_hpb}, the following holds:
    \begin{align*}
    & \mathbb{P}\lrp{ \exists \, k \leq T : \|x_k - x^*\|^2_c  > 2 f(k, \delta'', \delta') + 2g^2(\tfrac{\delta'}{T}) } \leq 
    \mathbb{P}\lrp{ \exists \, k\leq T : \|\tilde{x}_k - \tilde{x} \|^2_c > f(k, \delta'', \delta')  } + \delta'.
    \end{align*}
\end{lemma}
%The first term in right-hand side is an upper bound on the original probability when evaluated on the set $\mathcal B_T(\delta')$, and the additive $\delta'$ term on the right-hand side above is an upper bound on the probability of the complement of the set $\mathcal B_T(\delta')$. 
The proof relies on the fact that the iterates $\xk$ and $\tilde{x}_k$ of the original and the auxiliary algorithm coincide on the set $\mathcal B_T(\delta')$, whose complement has probability at most $\delta'$, and it is given in Appendix~\ref{app:proof_lem:probboundtrue_aux_bad}.

\medskip

Note that the first term on right-hand side of Lemma~\ref{lem:probboundtrue_aux_bad} corresponds to the probability that the iterates of the auxiliary algorithm introduced in Step 1 exceed $f(k, \delta'', \delta')$, which was shown to be at most $\delta''$ in Equation \eqref{eq:aux_hpb}. % Recall that the noisy operator used in the auxiliary algorithm at time $k\leq T$ is given by $F(\tilde{x}_k, S_k) + Z_k {\bf 1}_{\{\|Z_k\|_c \le B(\delta'/T)\}} $. 
Setting $\delta' = \delta'' = \delta/2$ in Lemma~\ref{lem:probboundtrue_aux_bad}, it follows that the probability that $\|\xk - x^*\|^2_c$ ever exceeds $2f(k,\delta/2,\delta/2) + 2g^2(\delta/(2T))$ is at most $\delta/2+\delta/2=\delta$. This concludes the proof of Theorem \ref{th:generalnoiseconc}.

\section{Conclusion and Discussions}
In this paper we establish maximal concentration bounds for exponentially stable SA algorithms with general step sizes of order $\mathcal{O}(1/k^z)$ for $z\in (0,1]$, in presence of finite-state Markovian noise as well as bounded Martingale-difference additive noise. For contractive SA with $\mathcal{O}(1/k)$ step sizes, we further relax boundedness assumptions and derive the first concentration guarantees in the presence of general heavy-tailed noise models, including sub-Weibull and sub-Pareto distributions. Methodologically, our analysis introduces new proof techniques to address the combined challenges posed by Markovian dependence, general step sizes, and potentially heavy-tailed noise. The main ingredients include bounding a modified log-moment generating function derived via Poisson equation arguments, together with an analysis of coupled auxiliary iterates.

On the applications side, we instantiate our general bounds for contractive as well as linear SA. A noteworthy consequence of our results is that, when combined with the findings of \citet{khodadadian2025general}, they reveal a subtle phenomenon: although Polyak averaging achieves optimal asymptotic variance, the resulting estimation error remains heavy-tailed under typical step-size schedules such as $\mathcal{O}(1/k^z)$ for $z \in (0,1)$.

We conclude by outlining several directions for future work. First, it would be interesting to explore whether Polyak averaging can be modified to simultaneously retain optimal asymptotic variance while producing lighter-tailed error. Second, applying our concentration framework to reinforcement learning algorithms such as TD and Q-learning, to get explicit dependence on the problem parameters like state space, mixing time, etc.,  remains a compelling next step. Finally, as seen in Section~\ref{sec:examples}, our bounds in Theorem~\ref{th:generalnoiseconc} are not tight in certain regimes. Sharpening these estimates will likely require fundamentally new analytical tools and is left for future investigation.

\bibliographystyle{plainnat}
\bibliography{BibTex,references}

\appendix

\section{Technical details from Section~\ref{sec:boundednoise}}
\subsection{Constants in theorems~\ref{th:ballargumentbound} and~\ref{th:ballargumentboundz}}\label{app:const_ball}
Let $A_{13} := A_1+A_3$, $B_{123} := B_1 + B_2 + B_3$, and 
\begin{align*}
    C'_1 &:= A_{13}(L_1+L_2) + 2(L_1+L_2) + L_1\lrp{A_1\|x^*\|_c + B_{123}}, \\
    C'_2 &:= (L_1 + L_2){A_1 \|x^*\|_c + B_{123}} + (1+A_{13})L_2.\end{align*}
Define constants $C_1$ and $C_2$ as below: 
\begin{align*}
C_1 &= \tfrac{u L_s}{\lcs^2\lrp{\lcs\wedge 1}^2} \lrp{\lrp{\tfrac{3}{\alpha} + \eta}(L_1+L_2) + (3L_1+1) (1+A_{13})^2 + C'_1} + \tfrac{u L_s(1+A_{13})^2}{\lcs^2},\\
C_2 &=   \lrp{\tfrac{3}{\alpha} + \eta} \tfrac{L_sL_2}{4} + \tfrac{ L_s }{\lcs^2 \lrp{\lcs \wedge 1}^2}\lrp{(3L_1+1) (A_{1} \|x^*\|_c + B_{123})^2 + C'_2} + \tfrac{L_s}{\lcs^2}\lrp{A_1\|x^*\|_c + B_{123}}^2.\end{align*}   
Further, for $x_0 \ne x^*$, let $D_2 = C_2 + \tfrac{\eta L_sL_2}{4}$, $D_3 =  { D_2 \eta \lcs^4 }/({16 L^2_s u  \lrp{L^2_1\|x_0 - x^*\|^2_c + L^2_2 + B^2_2} })$, and 
\[\theta =  \frac{\eta \lcs^4 \|x_0 - x^*\|^2_c}{32 L^2_s u  \lrp{L^2_1\|x_0 - x^*\|^2_c + L^2_2 + B^2_2} }.\]

\subsection{A Worst-case Bound}\label{app:wc}
In the next lemma, we present a worst-case bound for the squared error of the iterates that holds almost surely. %This worst-case bounding sequence will serve as a candidate sequence $T_k(\cdot)$ in Theorem~\ref{th:onestepbootstrap}, and will be used to prove Theorem~\ref{th:ballargumentbound} Part~\ref{ball_part1_cont} and Theorem~\ref{th:ballargumentboundz} Part~\ref{ball_partz_cont}. 

\begin{lemma}\label{lem:worst_case_bound_original}
For $\alpha_k = \alpha/(k+h)$, we have $\|x_k - x^*\|_c \le B_k$ almost surely, where
\begin{align*}
        B_k := \begin{cases}
            \lrp{\frac{k-1+h}{h-1}}^{\alpha (A_{13}-1)} \lrp{ \|{x}_0 - x^*\|_c + \frac{A_1 \|x^*\|_c + B_{123}}{A_{13}-1} }, \quad & A_{13} - 1 > 0,\\
            \|x_0 - x^*\|_c + (A_1 \|x^*\|_c + B_{123}) \alpha \log \frac{k-1+h}{h-1}, \quad & A_{13} - 1= 0,\\
            \|x_0 - x^*\|_c - \frac{A_1 \|x^*\|_c + B_{123}}{A_{13}- 1}, \quad & A_{13} -1 < 0.
        \end{cases}\numberthis\label{eq:wcbound}
    \end{align*}
    Moreover, for $\alpha_k = \alpha/(k+h)^z$, with $z\in (0,1)$, we have $\|x_k - x^*\|_c \le B_k$ almost surely, where
\begin{align*}
    B_k := \begin{cases}
        e^{(A_{13} - 1) \frac{\alpha}{1-z}\lrp{ (k+h-1)^{1-z} - (h-1)^{1-z} }}\lrp{\|x_0 - x^*\|_c  + \frac{A_1 \|x^*\|_c + B_{123}}{A_{13} - 1}} , \quad &\text{for } A_{13} - 1 > 0,\\
        \|x_0 - x^*\|_c + \frac{\alpha (A_1 \|x^*\|_c + B_{123}) }{1-z}\lrp{ (k+h-1)^{1-z} - (h-1)^{1-z} },\quad &\text{for } A_{13} - 1 = 0,\\
        \|x_0 - x^*\|_c - \frac{A_1 \|x^*\|_c + B_{123}}{A_{13} - 1}, \quad &\text{for } A_{13} - 1 < 0.
    \end{cases}\numberthis\label{eq:wcboundz}
\end{align*}
\end{lemma}

\begin{proof} 
Consider the following: 
\begin{align*}
    \|{x}_{k+1} - x^*\|_c 
    &= \|{x}_k + \alpha_k(F({x}_k, S_k) + Z_k - {x}_k) - x^*\|_c \\ 
    &= \|(1-\alpha_k)({x}_k - x^*) + \alpha_k(F({x}_k, S_k) - \bar{F}({x}_k) + Z_k + \bar{F}({x}_k) - x^*)\|_c \\
    &\le (1-\alpha_k) \|{x}_k - x^*\|_c + \alpha_k \|F({x}_k, Y_k) - \bar{F}({x}_k)\|_c + \alpha_k \|Z_k\|_c + \alpha_k \| \bar{F}({x}_k) - x^*\|_c\\
    &\le (1-\alpha_k) \|{x}_k - x^*\|_c + \alpha_k \lrp{A_1 \|{x}_k\| + B_1} + \alpha_k B_2 + \alpha_k \|\bar{F}({x}_k) - x^*\|_c\\
    &\le (1-\alpha_k) \|{x}_k - x^*\|_c + \alpha_k\lrp{A_1 \|{x}_k - x^*\|_c + A_1\|x^*\|_c + B_1} \\
    &\qquad + \alpha_k  B_2 + \alpha_k \lrp{A_3 \|{x}_k - x^*\|_c + B_3 } \\
    &= \lrp{1-\alpha_k\lrp{1- A_1 - A_3}}\|{x}_k - x^*\|_c + \alpha_k \lrp{A_1 \|x^*\|_c + B_1 + B_2 + B_3}\\
    &= \lrp{1+\alpha_k\lrp{A_{13} - 1}} \|{x}_k - x^*\|_c + \alpha_k \lrp{A_1 \|x^*\|_c + B_{123}}.
\end{align*}
Using \cite[Lemma 3.1]{chen2025concentration}, we get 
\begin{align*}
    B_k \le \begin{cases}
        e^{(A_{13}-1) \sum\limits_{i=0}^{k-1}\alpha_i}\|{x}_0 - x^*\|_c  + \frac{A_1 \|x^*\|_c + B_{123}}{A_{13}-1} \lrp{e^{(A_{13}-1) \sum\limits_{i=0}^{k-1}\alpha_i}  - 1}, \quad &\text{for } A_{13} - 1  > 0,\\
        \|{x}_0 - x^*\|_c + (A_1 \|x^*\|_c + B_{123}) \sum\limits_{i=0}^{k-1}\alpha_i,\quad &\text{for } A_{13} - 1 = 0,\\
        \|{x}_0 - x^*\|_c - \frac{A_1 \|x^*\|_c + B_{123}}{A_{13}-1}, \quad &\text{for } A_{13}-1 < 0.
    \end{cases}
\end{align*}
Further, using 
    \[ \sum\limits_{i=0}^{k-1}\alpha_i  = \alpha \sum\limits_{i=0}^{k-1} \frac{1}{i+h} \le \alpha \int\limits_{-1}^{k-1} \frac{1}{x+h} dx = \alpha \log\lrp{\frac{k-1+h}{h-1}} \]
for $z=1$, and otherwise using
\[ \sum\limits_{i=0}^{k-1}\alpha_i  = \alpha \sum\limits_{i=0}^{i-1} \frac{1}{(k+h)^z} \le \alpha \int\limits_{-1}^{k-1} \frac{1}{(x+h)^z} dx = \frac{\alpha}{1-z}\lrp{ (k+h-1)^{1-z} - (h-1)^{1-z} } \]
to bound $B_k$, we get the desired bounds.
\end{proof}

\subsection{Proof of theorems~\ref{th:ballargumentbound} and~\ref{th:ballargumentboundz} (Case 1: $A_1 + A_3 - 1 < 0$)}\label{app:onestepconc}

We use the following constants throughout this appendix 
\[A_{13} := A_1+A_3, \quad B_{123} := B_1 + B_2 + B_3.\]
Recall that $x_0 \ne x^*$, and constants $D_2$, $D_3$, and $\theta$ as in Appendix~\ref{app:const_ball}.

\smallskip

Recall from Section~\ref{sec:proof_sketch_12} that we prove theorems~\ref{th:ballargumentbound} and~\ref{th:ballargumentboundz} for the two cases $(A_1 + A_3 - 1 < 0)$ and $(A_1 + A_3 - 1 \ge 0)$ separately. In this section, we detail the proof for the first case. The proof relies on an auxiliary result that assumes an almost-sure bound on the squared error, and gets a tighter high-probability bound. We begin by presenting this concentration result in Theorem~\ref{th:onestepbootstrap} below.

\begin{theorem}\label{th:onestepbootstrap}
    For all $k\in\N$, let $\alpha_k = \alpha/(k+h)^z$ for some $z\in(0,1]$, where $\alpha > 0$ and $h \ge 8$ satisfy the conditions in Assumption~\ref{assmp:stepsize}. Let $\lrset{T_k}_{k\ge 1}$ be a non-decreasing sequence in $k$ such that, for all $k\geq 0$,
    \[ \|x_k - x^*\|^2_c \leq T_k, \qquad a.s. \]
    Then, under assumptions~\ref{assmp:barf},~\ref{assmp:boundedmultnoise},~\ref{assmp:boundedaddnoise}, and~\ref{assmp:lyapunov}, the following hold.
    \begin{enumerate}
    \item \label{part1} For $z=1$, for any $K\ge 0$ and $\tilde{\delta} > 0$, the following holds with probability at least $1- \tilde{\delta}$: for all $k\ge K$, 
\begin{align*}
    \|x_k - x^*\|^2_c & \le \frac{T_k}{k+h} \left[ \bar{b}_1\log\left(\frac{1}{\tilde{\delta}}\right) + \bar{b}_2 + \bar{b}_3\lrp{\frac{h}{K+h}}^{\alpha \eta/2 - 1}    + \bar{b}_4  \log\left(\frac{k-1+h}{K-1+h}\right)  \right], 
\end{align*}
where
\begin{align*}
    &\bar{b}_1 = \frac{u\alpha}{\theta}, \quad \bar{b}_2 = \frac{u\alpha}{\theta}\lrp{ \frac{8\alpha e \theta D_2}{(\frac{\alpha\eta}{2} - 1)\|x_0 - x^*\|^2_c} + \frac{ 2u L_s (L_1 + L_2) \theta }{ l\lcs^2}},\quad \bar{b}_4 = \frac{D_3 u\alpha^2}{\theta},\\
    & \bar{b}_3 = \frac{\eta\lcs^2 \alpha}{64 l L^2_s \alpha_0 L^2_1 L_2 \theta }\lrp{2\lcs^2 L_2 +  2 u \alpha_{-1} (L_1 + L_2)  L_2 L_s + \lcs^2 \alpha_{-1} l L_s L^2_1 }.
\end{align*}
\item \label{partz} For $z\in(0,1)$, for any  $K\ge 0$ and $\tilde{\delta} > 0$, the following holds with probability at least $1 - \tilde{\delta}$: for all $k \ge K$, 
\begin{align*}
    \|x_k - x^*\|^2_c &\le  \frac{T_k}{(k+h-1)^z} \left[ \bar{c}_1\log\lrp{\frac{1}{\tilde{\delta}}} + 2\bar{c}_1\log\lrp{\frac{k+1}{\sqrt{K+1}}} \right. \\
    &\qquad \qquad \left. + \bar{c}_2 + \bar{c}_3 \lrp{\frac{K+h}{h}}^z e^{\frac{-\eta \alpha}{2(1-z)}\lrp{(k+h)^{1-z} - h^{1-z}}} \right],
\end{align*}
where
\begin{align*}
    &\bar{c}_1 = \frac{\alpha u}{\theta}, \quad \bar{c}_2 = \frac{4 \alpha u D_2 }{\frac{\eta}{2}\|x_0 - x^*\|^2_c } + \frac{u^2 L_s (L_1 + L_2) \alpha}{\lcs^2 l} + \frac{\alpha u}{\theta}\log\frac{\pi^2}{6}, \\
    &\bar{c}_3 = \frac{\eta\lcs^2 \alpha }{64 \theta l L^2_s \alpha_0 L^2_1 L_2 }\lrp{2\lcs^2 L_2 +  2 u \alpha_{-1} (L_1 + L_2)  L_2 L_s + \lcs^2 \alpha_{-1} l L_s L^2_1 }.
\end{align*}
\end{enumerate}
\end{theorem}

\begin{remark}
    As in Section~\ref{sec:boundednoise}, in the above Theorem we assume that $x_0 \ne x^*$. However, this assumption can be easily relaxed by appropriately choosing the parameters involving the term $\|x_0 - x^*\|_c$, and by carefully handling such terms in the proof. See also Remark~\ref{rem:initialerror}.
\end{remark}
To prove the concentration result in Theorem~\ref{th:onestepbootstrap}, we construct a non-negative supermartingale using a novel Lyapunov function, which is a modification of the moment-generating Lyapunov function (exponential of the squared error). A complete and detailed proof of Theorem~\ref{th:onestepbootstrap} is given in Appendix~\ref{app:proofdetails_onestepbootstrap}. 

The worst-case bound (Lemma~\ref{lem:worst_case_bound_original}) for the squared error of the $x_k$ that holds almost surely serves as a candidate sequence $T_k$ in Theorem~\ref{th:onestepbootstrap}, and is used to prove Theorem~\ref{th:ballargumentbound} Part~\ref{ball_part1_cont} and Theorem~\ref{th:ballargumentboundz} Part~\ref{ball_partz_cont} below.

\begin{proof}[Proof of Theorem~\ref{th:ballargumentbound} (Part~\ref{ball_part1_cont}) and Theorem~\ref{th:ballargumentboundz} (Part~\ref{ball_partz_cont}).]
Applying the bound in Theorem~\ref{th:onestepbootstrap}, Part~\ref{part1} with $\tilde{\delta} = \delta$, and $T_k = B_k$ from Lemma~\ref{lem:worst_case_bound_original} (Equation~\eqref{eq:wcbound} for the case $A_{13} \le 1$), proves the bound in Theorem~\ref{th:ballargumentbound}, Part~\ref{ball_part1_cont}. Similarly,  using the bound in Theorem~\ref{th:onestepbootstrap}, Part~\ref{partz} with $\tilde{\delta} = \delta$, and $T_k = B_k$ from Lemma~\ref{lem:worst_case_bound_original} (Equation~\eqref{eq:wcboundz} for the case $A_{13} < 1$), proves the bound in Theorem~\ref{th:ballargumentboundz}, Part~\ref{ball_partz_cont}, completing the proof.
\end{proof}

\subsection{Proof of theorems~\ref{th:ballargumentbound} and~\ref{th:ballargumentboundz} (Case 2: $A_1 + A_3 - 1 \ge 0$)}\label{app:proofs:sec:boundednoise}
Recall that $\{\xk\}_{k\ge 0}$ denotes the iterates of the original SA algorithm given in Equation~\eqref{eq:SAalgo} and $x^*$ denotes the fixed point for which it solves. In this appendix, we prove theorems~\ref{th:ballargumentbound} (Part~\ref{ball_part1_exp}) and~\ref{th:ballargumentboundz} (parts~\ref{ball_partz_equal} and~\ref{ball_partz_exp}), i.e., the cases when $A_1 + A_3 - 1 \ge 0$. Theorem~\ref{th:onestepbootstrap} is treated as a black box. Finally, in this appendix, we use the following notation, for brevity:  
\[ A_{13} := A_1 + A_3, \quad \text{ and } \quad B_{123} := B_1 + B_2 + B_3. \]

Recall from Section~\ref{sec:proof_sketch_12} that the proof for this case proceeds in two steps; Step 1 is presented in Appendix~\ref{app:auxalgo} below, and the Step 2 along with the proofs for the overall bound for the different cases is presented in  Appendix~\ref{app:proof_th:ball_part1}.

%We begin by giving an overview of the overall proof that we present in this section. We consider an auxiliary algorithm whose iterates are guaranteed to lie within a ball of a certain radius, centered at $x^*$. We use the radius of this ball as a bound on the norm of the distance of the auxiliary iterates and $x^*$ in Theorem~\ref{th:onestepbootstrap} as a starting point to get a high-probability bound on the distance of the auxiliary iterates from $x^*$. We then show that if the radius of this ball is large enough, then the iterates of the original algorithm and those for the auxiliary algorithm are almost surely equal. This implies that the high-probability bound that we got from the black-box call to Theorem~\ref{th:onestepbootstrap} for the auxiliary iterates, also holds for the original iterates, proving Theorem~\ref{th:ballargumentbound}. We now present the auxiliary algorithm in the following section. 
\subsubsection{Step 1: The Auxiliary Algorithm}\label{app:auxalgo}
In this appendix, we introduce the auxiliary algorithm that we use in Step 1 of the proof for the bound in Theorem~\ref{th:ballargumentbound} Part~\ref{ball_part1_exp} and Theorem~\ref{th:ballargumentboundz} parts~\ref{ball_partz_equal} and~\ref{ball_partz_exp}. We also establish the properties of its iterates, referred to in this appendix as the auxiliary iterates. 

Consider a ball of radius $B$ (to be chosen later), centered at $x^*$; denote it by $\mathfrak B$. Consider the algorithm whose iterates $\{\tilde{x}_k\}_{k\geq 0}$ are given by
\begin{align*}
    \tilde{x}_{k+1} &= \Pi_c\Big( \tilde{x}_k + \alpha_k \big[ F(\tilde{x}_k, Y_k) + Z_k - \tilde{x}_k \big] \Big), 
\end{align*}
where $\Pi_c$ is the orthogonal projection on $\mathfrak B$ with respect to the $\|\cdot\|_c$ norm. Moreover, $\tilde{x}_0 = \Pi_c(x_0)$. \\

The iterates $\tilde{x}_k$ satisfy certain properties, which are useful later in the proofs. We present these next.

\begin{lemma}\label{lem:projiteratesinball}
    The auxiliary iterates $\{\tilde{x}_k\}_{k\geq 0}$ belong to the ball $\mathfrak B$.
\end{lemma}
\begin{proof}
    This trivially follows from the projection step in the definition of $\tilde{x}_k$. 
\end{proof}

\begin{lemma}\label{lem:worst_case_bound}
For $\alpha_k = \alpha/(k+h)$, the auxiliary iterates $\{\tilde{x}_k\}_{k\geq 0}$, also satisfy the worst-case almost-sure bound from Lemma~\ref{lem:worst_case_bound_original}.
\end{lemma}
\begin{proof} 
Since
\begin{align*}
    \|\tilde{x}_{k+1} - x^*\|_c 
    &= \| \Pi_c \lrp{ \tilde{x}_k + \alpha_k \lrp{ F(\tilde{x}_k, S_k) + Z_k - \tilde{x}_k } } - x^* \|_c\\
    &\le \|  \tilde{x}_k + \alpha_k \lrp{ F(\tilde{x}_k, S_k) + Z_k - \tilde{x}_k }  - x^* \|_c,
\end{align*}
then the desired bound can be proven along the lines of that of Lemma~\ref{lem:worst_case_bound_original}.
\end{proof}

\begin{remark}\label{rem:concaux} Note that
\begin{align*}
    \| \tilde{x}_{k+1} - x^* \|_c &= \Big\| \Pi_c\Big(\tilde{x}_k + \alpha_k \big[ F(\tilde{x}_k, S_k) + Z_k - \tilde{x}_k \big] \Big) - x^* \Big\|_c \\
    &\leq \Big\| \tilde{x}_k + \alpha_k \big[ F(\tilde{x}_k, S_k) + Z_k - \tilde{x}_k \big] - x^* \Big\|_c.
\end{align*}
Moreover, since the iterates $\tilde{x}_k$, the noise $S_k$ and $Z_k$, the operator $\bar{F}$, and the Markov chain satisfy assumptions~\ref{assmp:barf},~\ref{assmp:boundedmultnoise},~\ref{assmp:boundedaddnoise}, and~\ref{assmp:lyapunov}, the error of the auxiliary iterates $\|\tilde{x}_k - x^*\|_c$ satisfy the concentration in Theorem~\ref{th:onestepbootstrap}, with the bounding sequence $T_k$ chosen from Lemma~\ref{lem:worst_case_bound} above. 
\end{remark}
We use the observation in Remark~\ref{rem:concaux} above to arrive at high-probability bounds for $\|\tilde{x}_k-x^*\|_c$ separately in the two cases: $z=1$ and $z\in(0,1)$. %We first consider the case when $z=1$ and $A_1+A_3 > 1$.
\begin{lemma}\label{lem:ballbound1}
    For $A_{13}:= A_1 + A_3 > 1$, $\alpha_k = \alpha/(k+h)$ with $\alpha$ and $h$ satisfying the conditions in Assumption~\ref{assmp:stepsize}, for any $\delta\in(0,1]$, and constants $\bar{b}_1, \bar{b}_2,$ and $\bar{b}_4$ as in Theorem~\ref{th:ballargumentbound}, the following holds  with probability at least $1-\delta$: for all $k \ge 0$,
    \begin{align*}
 \|\tilde{x}_k - x^*\|^2_c 
&\leq \min\left\{ \bar{a}_1 \lrp{\frac{B^2}{\bar{a}_1} \left[\bar{b}_1\log\left(\frac{1}{\delta}\right) + \bar{b}_2 \right]}^\frac{2\alpha (A_{13}-1)}{2\alpha (A_{13}-1) + 1}, \right. \\
&\qquad\qquad\left. \frac{B^2}{k+h-1} \bigg[\bar{b}_1\log\left(\frac{1}{\delta}\right) + \bar{b}_2 + \bar{b}_4 \log\left(\frac{k+h-1}{h-1}\right)\bigg] \right\},
\end{align*}
where
\[ \bar{a}_1 = \lrp{\frac{1}{h-1}}^{2\alpha(A_{13}-1)}\lrp{\|\tilde{x}_0 - x^*\|_c + \frac{A_1 \|x^*\|_c + B_{123}}{A_{13}-1}}^2. \]  
%and
%\[ k_1 = \left\lfloor \left( \frac{B^2}{\bar{a}_1} \big[\bar{b}_1\log\left(\frac{1}{\delta}\right) + \bar{b}_2 \big] \right)^\frac{1}{2\alpha (A_{13}-1) + 1} -h +1 \right\rfloor. \]
\end{lemma}

\begin{proof}
From Lemma \ref{lem:projiteratesinball}, we have $\|\tilde{x}_k - x^*\|^2_c \le B^2$ for all $k\ge 0$ almost surely. Hence, from Theorem~\ref{th:onestepbootstrap}~(Part \ref{part1}), for $K\ge 0$, we have the following with probability at least $1-\delta$:  for all $k\ge K$, 
\begin{align*}
\|\tilde{x}_k - x^* \|_c^2 &\leq \frac{B^2}{k+h-1} \bigg[\bar{b}_1\log\left(\frac{1}{\delta}\right) + \bar{b}_2 + \bar{b}_4 \log\left(\frac{k+h-1}{K+h-1}\right)\bigg].\numberthis\label{eq:ballbound1}
\end{align*}
Next, observe that since $h \geq 8 > e^2$, $\log(k+h-1)^2/(k+h-1)$ is a decreasing function of $k$. Thus, when $z=1$, the worst-case bounds in lemmas~\ref{lem:worst_case_bound_original} and~\ref{lem:worst_case_bound} are non-decreasing after division by $k$ when $A_{1} + A_3 > 1$. Hence, we also have 
\[ \|\tilde{x}_{k} - x^*\|^2_c \leq \bar{a}_1 \lrp{k+h-1}^{2\alpha (A_{13} - 1)}, \qquad a.s., \numberthis\label{eq:wcbound1} \]
where $\bar{a}_1$ is as in the lemma statement. 
%\[ \bar{a}_1 = \lrp{\frac{1}{h-1}}^{2\alpha(A_{13}-1)}\lrp{\|\tilde{x}_0 - x^*\|_c + \frac{A_1 \|x^*\|_c + B_{123}}{A_{13}-1}}^2 . \]
Since the bound in Equation~\eqref{eq:ballbound1} is decreasing in $k$ and that in Equation~\eqref{eq:wcbound1} is increasing for any fixed $K$, let $k_1$ be the largest integer $k$ for which the following holds:
\begin{align*}
    \bar{a}_1 \lrp{k+h-1}^{2\alpha (A_{13} - 1)} &\leq \frac{B^2}{k+h-1} \bigg[\bar{b}_1\log\left(\frac{1}{\delta}\right) + \bar{b}_2 + \bar{b}_4 \log\left(\frac{k+h-1}{k_1+h-1}\right)\bigg].
\end{align*}
In particular, for $k=k_1$, we have
\begin{align*}
    \bar{a}_1 \lrp{k_1+h-1}^{2\alpha \lrp{A_{13}-1}} &\leq \frac{B^2}{k_1+h-1} \left[\bar{b}_1\log\left(\frac{1}{\delta}\right) + \bar{b}_2 \right], 
\end{align*}
and thus we get
\[ k_1 = \left\lfloor \left( \frac{B^2}{\bar{a}_1} \left[\bar{b}_1\log\left(\frac{1}{\delta}\right) + \bar{b}_2 \right]  \right)^\frac{1}{2\alpha (A_{13} - 1) + 1} -h + 1 \right\rfloor. \]

\noindent{\bf Case 1: $k_1\geq 0$.} In this case, we have the following for all $k \ge 0$: 
\[ \|\tilde{x}_k - x^*\|^2_c 
\overset{\text{a.s.}}{\le} \bar{a}_1 \lrp{k_1+h-1}^{2\alpha (A_{13}-1)} \leq \bar{a}_1 \lrp{\frac{B^2}{\bar{a}_1} \left[\bar{b}_1\log\left(\frac{1}{\delta}\right) + \bar{b}_2 \right]}^\frac{2\alpha (A_{13}-1)}{2\alpha (A_{13}-1) + 1}, \numberthis\label{eq:bound1case1}\]
where the second inequality is due to the removal of the floor function in the definition of $k_1$. Moreover, setting $K=k_1+1$ in Equation~\eqref{eq:ballbound1}, we also have the following holding with probability at least $1-\delta$: $\forall k \ge k_1 + 1$, 
\begin{align*}
\|\tilde{x}_k - x^*\|^2_c 
&\leq \frac{B^2}{k+h-1} \bigg[\bar{b}_1\log\left(\frac{1}{\delta}\right) + \bar{b}_2  + \bar{b}_4 \log\left(\frac{k+h-1}{ k_1+h}\right)\bigg]. \numberthis\label{eq:bound2case1}
\end{align*}
Since the bound in Equation~\eqref{eq:bound1case1} is smaller than that in Equation~\eqref{eq:bound2case1} only for  $k\leq k_1$, the following holds with probability at least $1-\delta$: $\forall k \ge 0$,
\begin{align*}
 \|\tilde{x}_k - x^*\|^2_c 
&\leq \min\left\{ \bar{a}_1 \lrp{\frac{B^2}{\bar{a}_1} \left[\bar{b}_1\log\left(\frac{1}{\delta}\right) + \bar{b}_2 \right]}^\frac{2\alpha (A_{13}-1)}{2\alpha (A_{13}-1) + 1}, \right. \\
&\qquad\qquad\qquad \left. \frac{B^2}{k+h-1} \left[\bar{b}_1\log\left(\frac{1}{\delta}\right) + \bar{b}_2 + \bar{b}_4 \log\left(\frac{k+h-1}{k_1+h}\right)\right] \right\}.
\end{align*}

\noindent{\bf Case 2: $k_1 < 0$.} In this case, the following inequalities hold for all $k \ge 0$: 
\begin{align*}
    \bar{a}_1 &\lrp{\frac{B^2}{\bar{a}_1} \big[\bar{b}_1\log\left(\frac{1}{\delta}\right) + \bar{b}_2 \big] }^\frac{2\alpha (A_{13} - 1)}{2\alpha (A_{13}-1) + 1}
    \\
    &> \frac{B^2}{k+h-1} \bigg[\bar{b}_1\log\left(\frac{1}{\delta}\right) + \bar{b}_2  + \bar{b}_4 \log\left(\frac{k+h-1}{k_1+h-1}\right)\bigg]  \\
    &> \frac{B^2}{k+h-1} \bigg[\bar{b}_1\log\left(\frac{1}{\delta}\right) + \bar{b}_2  + \bar{b}_4 \log\left(\frac{k+h-1}{h-1}\right)\bigg]. 
\end{align*}
In addition, the bound  in Equation~\eqref{eq:ballbound1} also holds with $K = 0$, i.e., the following also holds with probability at least $1-\delta$: for all $ k \ge 0$,
\begin{align*}
 \|\tilde{x}_k - x^*\|^2_c 
&\leq \frac{B^2}{k+h-1} \bigg[\bar{b}_1\log\left(\frac{1}{\delta}\right) + \bar{b}_2 + \bar{b}_4 \log\left(\frac{k+h-1}{h-1}\right)\bigg].
\end{align*}
Combining the above two inequalities, we get the following with probability at least $1-\delta$: for all $k\ge 0$, 
\begin{align*}
 \|\tilde{x}_k - x^*\|^2_c 
&\leq \min\left\{ \bar{a}_1 \lrp{\frac{B^2}{\bar{a}_1} \left[\bar{b}_1\log\left(\frac{1}{\delta}\right) + \bar{b}_2 \right]}^\frac{2\alpha (A_{13}-1)}{2\alpha (A_{13}-1) + 1}, \right. \\
&\qquad\qquad\qquad \left. \frac{B^2(\delta)}{k+h-1} \bigg[\bar{b}_1\log\left(\frac{1}{\delta}\right) + \bar{b}_2 + \bar{b}_4 \log\left(\frac{k+h-1}{h-1}\right)\bigg] \right\},
\end{align*}
proving the lemma statement. 
\end{proof}

We now prove a similar high-probability bound for the case when $z \in (0,1)$ and $A_1 + A_3 > 1$.

\begin{lemma}\label{lem:ballbound2}
    For $A_{13}:= A_1 + A_3 > 1$, $z\in (0,1)$, $\alpha_k = \alpha/(k+h)^z$ with $\alpha$ and $h$ satisfying the conditions in Assumption~\ref{assmp:stepsize}, for any $\delta \in (0,1]$, $B > 0$, and  $\beta < 1$, %and constants $\bar{c}_1, \bar{c}_2,$ and $\bar{c}_3$ as in Theorem~\ref{th:onestepbootstrap} Part~\ref{partz}, 
    the following holds with probability at least $1-\delta$: for all $k \ge 0$,
    \begin{align*}
 \|\tilde{x}_k - x^*\|^2_c 
&\leq \mathcal{O}\lrp{ \frac{B^2 \log\left(\frac{1}{\delta}\right)}{ \log\lrp{B^2\log\left(\frac{1}{\delta}\right)}^{\frac{\beta z}{1-z}} } }. \numberthis\label{eq:worstcaseintersectionvalue}  
\end{align*}
\end{lemma}

\begin{proof}
From  Lemma~\ref{lem:projiteratesinball}, the following holds almost surely:
\[\|\tilde{x}_k - x^*\|^2_c \le B^2, \qquad \forall\, k\ge 0. \] 
Hence, from Theorem~\ref{th:onestepbootstrap} Part~\ref{partz}, for all $K\ge 0$, and constants $\bar{c}_1, \bar{c}_2$, and $\bar{c}_3$ from Theorem~\ref{th:onestepbootstrap} Part~\ref{partz}, the following holds with probability at least $1-\delta$: for all $k\ge K$, 
\begin{align*}
    \|\tilde{x}_k - x^*\|^2_c &\le  \frac{B^2}{(k+h-1)^z} \left[ \bar{c}_1\log\lrp{\frac{1}{\delta}} + 2\bar{c}_1\log\lrp{\frac{k+1}{\sqrt{K+1}}} + \bar{c}_2 \right. \\
    &\qquad\qquad\qquad\qquad\qquad\qquad \left. + \bar{c}_3 \lrp{\frac{K+h}{h}}^z e^{\frac{-\eta \alpha}{2(1-z)}\lrp{(k+h)^{1-z} - h^{1-z}}} \right]. \numberthis\label{eq:bootstrapbound_z}
\end{align*}
In addition, for all $k\ge 0$, Lemma~\ref{lem:worst_case_bound} implies the following almost sure bound:
\[ \|\tilde{x}_{k} - x^*\|^2_c \leq \bar{a}_2 e^{\bar{a}_3 \lrp{k + h -1}^{1-z} - \bar{a}_4}, \numberthis\label{eq:worstbound2} \]
where 
\[ 
    \bar{a}_2 := \|x_0 - x^*\|_c + \frac{A_1 \|x^*\|_c + B_{123}}{A_{13}-1}, \quad \bar{a}_3 := (A_{13}- 1)\frac{\alpha}{1-z}, \quad \text{ and } \quad \bar{a}_4 := \bar{a}_3{ (h-1)^{1-z}}.
\]
As earlier, the bound in Equation~\eqref{eq:bootstrapbound_z} is decreasing, and that in Equation~\eqref{eq:worstbound2} is increasing. Let $k_1$ be the largest value of $k$ for which the following holds:
\begin{align*}
    \bar{a}_2 e^{\bar{a}_3 \lrp{k + h -1}^{1-z} - \bar{a}_4} 
    &\leq \frac{B^2}{(k+h-1)^z}\left[ \bar{c}_1\log\lrp{\frac{1}{\delta}} + 2\bar{c}_1\log\lrp{\frac{k+1}{\sqrt{k_1+1}}} + \bar{c}_2  \right.\\
    &\qquad \left. + \bar{c}_3 \lrp{\frac{k_1 + h}{h}}^z \exp\lrp{\frac{-\eta \alpha}{2(1-z)}\lrp{(k+h)^{1-z} - h^{1-z}}} \right].
\end{align*}
In particular, for $k=k_1+1$, we have
\begin{align*}
    \bar{a}_2 e^{\bar{a}_3 \lrp{k_1 + h}^{1-z} - \bar{a}_4} 
    &> \frac{B^2}{(k_1+h)^z}\left[ \bar{c}_1\log\lrp{\frac{1}{\delta}} + 2\bar{c}_1\log\lrp{\frac{k_1+2}{\sqrt{k_1+1}}} + \bar{c}_2 \right. \numberthis\label{eq:boundz_1} \\
    &\qquad \left. + \bar{c}_3 \lrp{\frac{k_1+h}{h}}^z\exp\lrp{\frac{-\eta \alpha}{2(1-z)}\lrp{(k_1+h+1)^{1-z} - h^{1-z}}} \right].
\end{align*}
For $d_1 > 0$ and $d_2 > 0$, define
\[ \tilde{k}_1 := \lrp{d_1 \log\lrp{B^2 \log\left(\frac{1}{\delta}\right)} - 
  \frac{z d_2}{1-z} \log{\log\lrp{B^2 \log\left(\frac{1}{\delta}\right)}}}^\frac{1}{1-z}.\]
  
The ratio of the right-hand side over the left-hand side in Equation~\eqref{eq:boundz_1} when evaluated at $k_1 = \tilde{k}_1$, is upper bounded by 
\begin{align*}
    &\bar{c}_2 B^2 e^{-\bar{a}_3 T(\delta, B)} \lrs{T(\delta, B)}^{\frac{-z}{1-z}} + \bar{c}_1 B^2 e^{-\bar{a}_3 T(\delta, B)} \lrs{T(\delta, B)}^{\frac{-z}{1-z}}\log\left(\frac{1}{\delta}\right) \\ 
    & \qquad  + 2^z \bar{c}_3  B^2 e^{-\bar{a}_3 T(\delta, B) + \frac{\alpha \eta \lrp{ h^{1-z}  - T(\delta, B) }}{2-2z} } + \frac{3 \bar{c}_1 B^2 }{1-z} \log\lrp{T(\delta, B)}  e^{-\bar{a}_3 T(\delta, B)} \lrs{T(\delta, B)}^{\frac{-z}{1-z}},\numberthis\label{eq:limitobj}
\end{align*}
where
\[ T(\delta, B) = d_1 \log\lrp{B^2 \log\left(\frac{1}{\delta}\right)} - 
 \frac{z d_2}{1-z} \log{\log\lrp{B^2 \log\left(\frac{1}{\delta}\right)}}. \]
It is easy to verify that the limit of Equation~\eqref{eq:limitobj} is $0$ as $\delta\rightarrow 0$, if the following is satisfied:
\[\text{either }\quad \bar{a}_3 d_1 > 1\quad\text{ or }\quad\bar{a}_3 d_1 = 1\text{ and }\bar{a}_3 d_2 < 1. \numberthis \label{cond}\] 
Moreover, it is also $0$ for $B \rightarrow \infty$, if the condition in Equation~\eqref{cond} holds. Setting $d_1 = 1/\bar{a}_3$ and $d_2 = \beta d_1$ for $\beta < 1$ (so that the condition in Equation~\eqref{cond} is satisfied), we have
\[ k_1 = d^{\frac{1}{1-z}}_1\log\lrp{\frac{B^2 \log\left(\frac{1}{\delta}\right)}{\log \lrp{B^2\log\left(\frac{1}{\delta}\right)}^{\frac{\beta z}{1-z}}} }^\frac{1}{1-z}. \numberthis\label{eq:k1_z}\]
By using $k_1$ from Equation~\eqref{eq:k1_z}, the left-hand side of Equation~\eqref{eq:boundz_1} becomes
\[ \mathcal{O}\lrp{ \frac{B^2 \log\left(\frac{1}{\delta}\right)}{ \log\lrp{B^2\log\left(\frac{1}{\delta}\right)}^{\frac{\beta z}{1-z}} } }, \]
proving the lemma.
\end{proof}
Finally, we end this section with a bound for the case when $z \in (0,1)$ and $A_1 + A_3 = 1$. 
\begin{lemma}\label{lem:ballbound3}
    For $A_{13}:= A_1 + A_3 = 1$, $z\in (0,1)$, $\alpha_k = \alpha/(k+h)^z$ with $\alpha$ and $h$ satisfying the conditions in Assumption~\ref{assmp:stepsize},  for $\delta \in (0,1]$ and $B > 0$, and constants $\bar{c}_1, \bar{c}_2$, and $\bar{c}_3$ as in Theorem~\ref{th:onestepbootstrap} Part~\ref{partz}, %and constants $\bar{c}_1, \bar{c}_2,$ and $\bar{c}_3$ as in Theorem~\ref{th:onestepbootstrap} Part~\ref{partz}, 
    the following holds with probability at least $1-\delta$: for all $k \ge 0$,
    \[ \|\tilde{x}_k - x^*\|^2_c \le  \frac{2\bar{c}_1}{\bar{a}_6} B^2 (1-z) \log\lrp{\frac{1}{\delta}} + \frac{4\bar{c}_1}{\bar{a}_6} (1-z) B^2 \log\lrp{ \frac{2\bar{c}_1}{\bar{a}_6} (1-z) B^2  \log\lrp{\frac{1}{\delta}} } \numberthis\label{eq:worstcaseintersectionvalue_equal}, \]
    where $\bar{a}_6$ is a constant (independent of $k$, $\delta$) dependent on universal constants like $A_1$, $A_3$, $B_1$, $B_2$, $B_3$, $\alpha$, $h$, $z$, and initialization error.
\end{lemma}

\begin{proof}
    First, from  Lemma~\ref{lem:projiteratesinball}, 
    \[ \|\tilde{x}_k - x^*\|^2_c \le B^2, \qquad \forall \, k\ge 0, \] 
almost surely. Hence, from Theorem~\ref{th:onestepbootstrap}, for all $K\ge 0$, the following holds with probability at least $1-\delta$: for all $k\ge K$, 
\begin{align*}
    \|\tilde{x}_k - x^*\|^2_c &\le  \frac{B^2}{(k+h-1)^z} \left[ \bar{c}_1\log\lrp{\frac{1}{\delta}} + 2\bar{c}_1\log\lrp{\frac{k+1}{\sqrt{K+1}}} + \bar{c}_2 \right. \\
    &\qquad\qquad\qquad\qquad\qquad\qquad \left. + \bar{c}_3 \lrp{\frac{K+h}{h}}^z e^{\frac{-\eta \alpha}{2(1-z)}\lrp{(k+h)^{1-z} - h^{1-z}}} \right]. \numberthis\label{eq:bootstrapbound_zequal}
\end{align*}
In addition, for all $k\ge 0$, Lemma~\ref{lem:worst_case_bound} implies the following almost sure bound:
\[ \|\tilde{x}_{k} - x^*\|^2_c \leq \frac{\bar{a}_6}{1-z}{ (k+h-1)^{1-z}}, \numberthis\label{eq:worstbound3} \]
for a constant $\bar{a}_6$. Again, the bound in Equation~\eqref{eq:bootstrapbound_zequal} is decreasing, and that in Equation~\eqref{eq:worstbound3} is increasing. Let $k_1$ be the largest value of $k$ for which the following holds:
\begin{align*}
    \frac{\bar{a}_6}{1-z}(k+h-1)^{1-z} 
    &\leq \frac{B^2}{(k+h-1)^z}\left[ \bar{c}_1\log\lrp{\frac{1}{\delta}} + 2\bar{c}_1\log\lrp{\frac{k+1}{\sqrt{k_1+1}}} + \bar{c}_2  \right.\\
    &\qquad \left. + \bar{c}_3 \lrp{\frac{k_1 + h}{h}}^z \exp\lrp{\frac{-\eta \alpha}{2(1-z)}\lrp{(k+h)^{1-z} - h^{1-z}}} \right].
\end{align*}
In particular, for $k=k_1+1$, we have
\begin{align*}
    \frac{\bar{a}_6 (k_1+h)^{1-z}}{(1-z)}
    &> \frac{B^2}{(k_1 + h)^z}\left[\bar{c}_1\log\lrp{\frac{1}{\delta}} + 2\bar{c}_1\log\lrp{\frac{k_1+2}{\sqrt{k_1+1}}} + \bar{c}_2  \right. \numberthis\label{eq:boundz_2} \\
    &\qquad \left.  + \bar{c}_3 \lrp{\frac{k_1+h}{h}}^z\exp\lrp{\frac{-\eta \alpha}{2(1-z)}\lrp{(k_1+h+1)^{1-z} - h^{1-z}}}\right].
\end{align*}
For some $d_1 > \frac{\bar{c}_1 (1-z)}{\bar{a}_6} $ and $d_2 > 1$, define
\[ \tilde{k}_1 := d_1 B^2 \lrs{ \log\lrp{\frac{1}{\delta}} +  d_2 \log\lrp{ d_1 B^2 \log\lrp{\frac{1}{\delta}} }}.\]
The ratio of right-hand side and left-hand side in Equation~\eqref{eq:boundz_2} when evaluated at $k_1 = \tilde{k}_1$, is upper bounded by 
\begin{align*}
    1 + \frac{ d_2\log\lrp{1 + \frac{d_2\log\left(d_1 B^2 \log\left(\frac{1}{\delta}\right)\right)}{\log\left(\frac{1}{\delta}\right)}} }{\log\left(\frac{1}{\delta}\right) + d_2\log\left(d_1 B^2 \log\left(\frac{1}{\delta}\right)\right)},\numberthis\label{eq:limitobj2}
\end{align*}
It is easy to verify that the limit of Equation~\eqref{eq:limitobj2} is $1$ as $\delta\rightarrow 0$, and also as $B \rightarrow \infty$. Thus, for $z\in(0,1)$, we have
\[k_1 \le d_1 B^2 \log\lrp{\frac{1}{\delta}} + d_1 d_2 B^2 \log\lrp{ d_1 B^2 \log\lrp{\frac{1}{\delta}} }. \numberthis\label{eq:k1_z_equal}\]

Finally, using Equation~\eqref{eq:k1_z_equal} with $d_1 = \tfrac{2 \bar{c}_1 (1-z)}{\bar{a}_6} $ and $d_2 = 2$, the left-hand side of Equation~\eqref{eq:boundz_2} becomes
\[ \frac{\bar{a_6}}{1-z}{ \lrp{d_1 B^2 \log\lrp{\frac{1}{\delta}} + d_1 d_2 B^2 \log\lrp{ d_1 B^2 \log\lrp{\frac{1}{\delta}} }}^{1-z} }, \]
proving the lemma.
\end{proof}

\subsubsection{Step 2: Choice of Radius and the Overall Bound}\label{app:proof_th:ball_part1}
In this section, we prove Theorem~\ref{th:ballargumentbound}, Part~\ref{ball_part1_exp}, and Theorem~\ref{th:ballargumentboundz}, parts~\ref{ball_partz_equal} and~\ref{ball_partz_exp}. Earlier, in lemmas~\ref{lem:ballbound1},~\ref{lem:ballbound2}, and~\ref{lem:ballbound3}, we established high-probability bounds that depend on the radius $B$ of the ball $\mathfrak B$, for the auxiliary iterates. In this section (lemmas~\ref{lem:ballbound1},~\ref{lem:ballbound2}, and~\ref{lem:ballbound3}, below), we show that it is possible to choose the radius of the ball $\mathfrak B$ large enough so that the projection step in the auxiliary algorithm is not active on the sample paths in the chosen high probability set, and thus the auxiliary iterates coincide with the original iterates on these paths. Using these, we first prove the concentration bounds in theorems~\ref{th:ballargumentbound} and~\ref{th:ballargumentboundz}. In the rest of this section, we let $A_{13} = A_1 + A_3$. 

\begin{proof}[Proof of Theorem~\ref{th:ballargumentbound} (Part~\ref{ball_part1_exp}) and Theorem~\ref{th:ballargumentboundz} (parts~\ref{ball_partz_equal} and~\ref{ball_partz_exp}).]
Theorem~\ref{th:ballargumentbound}, Part~\ref{ball_part1_exp} follows by combining lemmas~\ref{lem:ballbound1} and~\ref{lem:ballradius1}. Choosing $B^2$ from Lemma~\ref{lem:ballradius2} in Equation~\eqref{eq:bootstrapbound_z} gives that $\|\tilde{x}_k-x^*\|^2_c$, which is same as $\|x_k - x^*\|^2_c$, satisfies the bound in Theorem~\ref{th:ballargumentboundz}~(Part~\ref{ball_partz_exp}). Finally, using  $B^2$  from Lemma~\ref{lem:ballradius3} in Equation~\eqref{eq:bootstrapbound_zequal}, $\|\tilde{x}_k-x^*\|^2_c$, which is the same as $\|x_k - x^*\|^2_c$, satisfies the bound in Theorem~\ref{th:ballargumentboundz} (Part~\ref{ball_partz_equal}). 
\end{proof}

We now present the choice for the ball radius $B$ in each of the three cases.

\begin{lemma}\label{lem:ballradius1}
    Fix $\delta\in(0,1]$. Suppose $z = 1$ and $A_{13} > 1$. For constants $\bar{a}_1$, $\bar{b}_1, \bar{b}_2$, and $\bar{b}_4$ as in Lemma~\ref{lem:ballbound1}, suppose 
    \[ B^2 > \bar{a}_1 \left[\bar{b}_1\log\left(\frac{1}{\delta}\right) + \bar{b}_2  \right]^{2\alpha (A_{13}-1)}. \]
    Then, with probability at least $1-\delta$, we have the following: for all $k\ge 0$, $\tilde{x}_k = x_k$.
\end{lemma}
\begin{proof}
From Lemma~\ref{lem:ballbound1}, we have the following with probability at least $1-\delta$: for all $k\ge 0$, 
\begin{equation}\label{eq:intersectionvalue}
 \|\tilde{x}_k - x^*\|^2_c 
\le \bar{a}_1 \lrp{\frac{B^2}{\bar{a}_1} \left[\bar{b}_1\log\left(\frac{1}{\delta}\right) + \bar{b}_2 \right]}^\frac{2\alpha (A_{13}-1)}{2\alpha (A_{13}-1) + 1}.
\end{equation}
Therefore, with probability at least $1-\delta$, for all $k\ge 0$,
\begin{align*}
    \|\tilde{x}_k - x^*\|^2_c &\leq \bar{a}_1 \lrp{\frac{B^2}{\bar{a}_1} \left[\bar{b}_1\log\left(\frac{1}{\delta}\right) + \bar{b}_2 \right] }^\frac{2\alpha (A_{13}-1)}{2\alpha (A_{13}-1) + 1} \\
    &< \bar{a}_1 \left[\frac{B^2}{\bar{a}_1} \left(\frac{B^2}{\bar{a}_1}\right)^\frac{1}{2\alpha (A_{13}-1)}\right]^\frac{2\alpha (A_{13}-1)}{2\alpha (A_{13}-1) + 1} \\
    &= \bar{a}_1 \left[ \left(\frac{B^2}{\bar{a}_1}\right)^\frac{2\alpha (A_{13}-1) + 1}{2\alpha( A_{13}-1)}\right]^\frac{2\alpha (A_{13}-1)}{2\alpha (A_{13}-1) + 1} \\
    &= B^2.
\end{align*}
That is, with probability at least $1-\delta$,  the auxiliary iterates lie strictly inside the ball $\mathfrak B$ and the projection step in the auxiliary algorithm is never active. Thus, with probability at least $1-\delta$, for all $k\ge 0$, $\tilde{x}_k = x_k$.
\end{proof}

\begin{lemma}\label{lem:ballradius2}
    Fix $\delta\in(0,1]$, $\beta < 1$, $z\in (0,1)$ and $A_{13}  > 1$. Suppose 
    \[ B^2 = \mathcal{O}\lrp{\frac{\exp\lrp{ \log\left(\frac{1}{\delta}\right)^\frac{1-z}{\beta z} }}{\log\left(\frac{1}{\delta}\right)}}, \]
    then with probability at least $1-\delta$, we have the following: for all $k\ge 0$, $\tilde{x}_k = x_k$.
\end{lemma}

\begin{proof}
Ensuring that the bound in  Equation~\eqref{eq:worstcaseintersectionvalue} is at most $B^2$ guarantees that the projection step in our auxiliary algorithm is never active, and the iterates of the auxiliary algorithm are the same as those for the original algorithm on $1-\delta$ probability paths, i.e., $\tilde{x}_k = x_k$ with probability at least $1-\delta$. For this, one can verify that for every $\delta$, there exists a large enough $B$ so that this is satisfied. In particular,  choosing $B^2$ as the following function of $\delta$ 
\[ B^2 = \mathcal{O}\lrp{\frac{\exp\lrp{ \log\left(\frac{1}{\delta}\right)^\frac{1-z}{\beta z} }}{\log\left(\frac{1}{\delta}\right)}} \]
suffices.
\end{proof}

\begin{lemma}\label{lem:ballradius3}
    Fix $\delta\in(0,1]$  and $z\in (0,1)$. Suppose 
    \[ B^2 = \mathcal{O}\lrp{\log\left(\frac{1}{\delta}\right)^\frac{1-z}{z}}.\]
    Then, with probability at least $1-\delta$, we have the following: for all $k\ge 0$, $\tilde{x}_k = x_k$.
\end{lemma}

\begin{proof}
Ensuring that the bound in Equation~\eqref{eq:worstcaseintersectionvalue_equal} is at most $B^2$ guarantees that the projection step in our auxiliary algorithm is never active on the $1-\delta$ samples paths where the bound holds. Hence, on these paths, the iterates of the auxiliary algorithm are the same as those for the original algorithm. To this end, one can verify that for every $\delta$, there exists a large enough $B$ so that Equation~\eqref{eq:worstcaseintersectionvalue_equal} is bounded by $B^2$. In particular, choosing $B^2$ as in the lemma statement suffices.
\end{proof}

\subsection{Proof of Theorem~\ref{thm:impossibility} }\label{app:impossibility}
\begin{proof}
The proof is obtained through the following example.

\begin{example}
Consider a $1$-dimensional linear SA as in Equation \eqref{eq:SAalgo}, with $F(x,s)=sx+b(s-a)$ for all $x\in\mathbb{R}$ and $s\in\mathcal{S}$, with $\{S_k\}_{k\geq 0}$ being an i.i.d.\ sequence of real-valued random variables s.t. 
$\mathbb{P}\left(S_k=a+N\right)=1/(N+1)$ and $\mathbb{P}\left(S_k=a-1\right)=N/(N+1)$, where $a\in (0,1)$, $b\geq 0$, and $N>1$ are tunable parameters. Note that the update equation can be equivalently written as
\begin{align}\label{eq:ex1}
    x_{k+1} =(1+(S_k-1)\alpha_k)x_k + \alpha_k (S_k-a) b.
\end{align}
\end{example}

In the example above, it can be easily verified that $\bar{F}(x)=ax$ is a contraction with the unique fixed point $x^*=0$, and that assumptions \ref{assmp:boundedmultnoise}--\ref{assmp:lyapunov} are satisfied. We assume that $x_0>0$, $\alpha_k=\alpha/(k+h)^z$, where $z\in(0,1)$ and $h$ is large enough so that Assumption \ref{assmp:stepsize} is satisfied.\\

We start with part (1). Fix $b=1$ and $N=1-a$. For any $\lambda>0$, and $\beta>0$, we have
\begin{align}
    &\mathbb{E}\left[\exp\left(\lambda \left[(k+h)^{z} \| x_k -x^* \|^2_c \right]^\beta\right)\right] \nonumber \\
    & = \mathbb{E}\left[\exp\left(\lambda \left[(k+h)^{z} x_k^2 \right]^\beta\right)\right] \nonumber \\
    & = \mathbb{E}\left[\exp\left(\lambda (k+h)^{{\beta}z} \left[ x_0 \prod_{i=0}^{k-1}(1+\alpha_i( S_i-1)) + \sum_{i=0}^{k-1} \alpha_i(S_i-a) \prod_{j=i+1}^{k-1} (1+\alpha_j(S_j-1)) \right]^{2\beta} \right) \right] \label{eq1:thm:impossibility} \\
    & \geq \frac{1}{(2-a)^k}\exp\left(\lambda (k+h)^{\beta z} \left[ x_0 + \sum_{i=0}^{k-1} \alpha_i(1-a) \right]^{2\beta}  \right) \label{eq:example_before} \\
    & \geq \frac{1}{(2-a)^k} \exp\left(\lambda (k+h)^{\beta z} \left[ x_0 + \frac{\alpha(1-a)}{1-z}((k+h)^{1-z}-h^{1-z}) \right]^{2\beta} \right), \label{eq:final_counter} \\
    & = \exp\left(\lambda (k+h)^{\beta z} \left[ x_0 + \frac{\alpha(1-a)}{1-z}((k+h)^{1-z}-h^{1-z}) \right]^{2\beta} - k \log(2-a) \right), \nonumber
\end{align}
where Equation \eqref{eq1:thm:impossibility} follows from the update Equation \eqref{eq:ex1} with $b=1$ and $N=1-a$, Equation \eqref{eq:example_before} follows from the distribution of $S_k$, and Equation \eqref{eq:final_counter} follows from
\begin{align*}
    \sum_{i=0}^{k-1}\alpha_i
    \geq \int_{0}^k\frac{\alpha}{(x+h)^z}dx
    = \frac{\alpha}{1-z}((k+h)^{1-z}-h^{1-z}).
\end{align*}
Note that, if $2\beta(1-z/2)>1$, then
\begin{align*}
        \liminf\limits_{k\to\infty} \mathbb{E}\left[\exp\left(\lambda \left[(k+h)^{z} \|x_k-x^*\|_c^2\right]^\beta\right)\right] = \infty, \qquad \forall\,\lambda>0.
\end{align*}
As a result, for any $\beta>1/(2-z)$, there do not exist constants $K_1,K_2>0$ s.t. 
\begin{align*}
    \mathbb{P}\Big((k+h)^z \|x_k-x^*\|_c^2 > \epsilon \Big) \leq K_1\exp\left(-K_2\epsilon^\beta\right), \qquad \forall\,\epsilon>0, \quad k\geq 0.
\end{align*}
Equivalently, for any $\beta'<2-z$, there do not exist constants $K_1',K_2'>0$ s.t. 
\begin{align*}
    \mathbb{P}\left(\|x_k-x^*\|_c^2 \leq \frac{K_1'\log\left(\frac{K_2'}{\delta}\right)^{\beta'}}{(k+h)^z} \right) \geq 1- \delta,\qquad \forall\,\delta\in(0,1), \quad k\geq 0.
\end{align*}
This completes the proof of part (1).\\

We now prove part (2). Fix $b=0$ and $N>1-a$. For any $\lambda>0$, and $\beta>0$, we have
\begin{align}
    &\mathbb{E}\left[\exp\left(\lambda \log \left[(k+h)^{z} \| x_k -x^* \|^2_c \right]^\beta\right)\right] \nonumber \\
    & = \mathbb{E}\left[\exp\left(\lambda \log\left[(k+h)^{z} x_k^2 \right]^\beta\right)\right] \nonumber \\
    & = \mathbb{E}\left[\exp\left(\lambda \log \left[ (k+h)^z x_0^2 \prod_{i=0}^{k-1}(1+\alpha_i( S_i-1))^2 \right]^{\beta} \right) \right] \label{eq1:thm:impossibility2} \\
    & \geq \frac{1}{(N+1)^k}\exp\left(\lambda \log \left[ (k+h)^z  x_0^2 \prod_{i=0}^{k-1}(1+\alpha_i( a+N-1))^2 \right]^{\beta}  \right) \label{eq:example_before_2} \\
    & \geq \frac{1}{(N+1)^k} \exp\left(\lambda \log \left[ (k+h)^z  x_0^2 \exp\left( \frac{2}{1+\epsilon}\sum_{i=k_0(\epsilon)}^{k-1} \alpha_i(a+N-1) \right) \right]^{\beta}  \right) \label{eq:example_before_3} \\
    & \geq \frac{1}{(N+1)^k} \exp\left(\lambda \log \left[ (k+h)^z  x_0^2 \exp\left( \frac{2(a+N-1)\alpha}{(1-z)(1+\epsilon)} \Big((k+h)^{1-z}-(k_0(\epsilon)+h)^{1-z}\Big) \right) \right]^{\beta} \right) \label{eq:example_before_4} \\
    & = \exp\left(\lambda \left[ \log( (k+h)^z  x_0^2) +  \frac{2(a+N-1)\alpha}{(1-z)(1+\epsilon)}\Big((k+h)^{1-z}-(k_0(\epsilon)+h)^{1-z}\Big) \right]^{\beta}  - k \log(N+1) \right), \nonumber
\end{align}
where Equation \eqref{eq1:thm:impossibility2} follows from the update equation \eqref{eq:ex1} with $b=0$, Equation \eqref{eq:example_before_2} follows from the distribution of $S_k$, Equation \eqref{eq:example_before_3} follows from the facts that $a+N-1>0$ and $1+x\geq e^{x/(1+\epsilon)}$ for all $x>0$ small enough, and Equation \eqref{eq:example_before_4} follows from
\begin{align*}
    \sum_{i=k_0(\epsilon)}^{k-1}\alpha_i
    \geq \int_{0}^k\frac{\alpha}{(x+h)^z}dx
    = \frac{\alpha}{1-z}\Big((k+h)^{1-z}-(k_0(\epsilon)+h)^{1-z}\Big).
\end{align*}
Note that, if $\beta>1/(1-z)$, then
\begin{align*}
        \liminf\limits_{k\to\infty} \mathbb{E}\left[\exp\left(\lambda \log\left[(k+h)^{z} \|x_k-x^*\|_c^2\right]^\beta\right)\right] = \infty, \qquad \forall\,\lambda>0.
\end{align*}
As a result, for any $\beta>1/(1-z)$, there do not exist constants $K_1,K_2>0$ s.t. 
\begin{align*}
    \mathbb{P}\Big((k+h)^z \|x_k-x^*\|_c^2 > \epsilon \Big) \leq K_1\exp\left(-K_2\log(\epsilon)^\beta\right), \qquad \forall\,\epsilon>0, \quad k\geq 0.
\end{align*}
Equivalently, for any $\beta'<1-z$, there do not exist constants $K_1',K_2'>0$ s.t. 
\begin{align*}
    \mathbb{P}\left(\|x_k-x^*\|_c^2 \leq \frac{K_1' \exp\Big(K_2'\log\left(\frac{1}{\delta}\right)^{\beta'}\Big)}{(k+h)^z} \right) \geq 1- \delta,\qquad \forall\,\delta\in(0,1), \quad k\geq 0,
\end{align*}
completing the proof.
\end{proof}

\subsection{Concentration for contractive operators}\label{app:conc_contraction}
Theorem~\ref{th:contractiveconc} in Section~\ref{sec:contractiveSA} specializes theorems~\ref{th:ballargumentbound} and~\ref{th:ballargumentboundz} to the setting where the average operator $\bar{F}$ is contractive. In this appendix, we verify the assumptions required for this specialization, including those on the noise sequence, the operator $\bar{F}$, and the associated stability conditions.
%Let $\tau := \max_{s_0 \in \mathcal S}\min\lrset{n > 0: S_n = s_0}$, and $V_x \in \R^{|\mathcal S|}$ be a solution of the Poisson's Equation~\eqref{eq:PE}. The existence of $V_x$ is guaranteed since we have a finite state, discrete time Markov chain with a stationary distribution $\pi$, and for all $x\in\R^d$, $F(x,Y(\cdot)) + Z(\cdot) - \bar{F}(x)$ has stationary mean $0$. Then, set $A_3 = \gamma_c$, $B_3 = 0$, $L_1 := (L_F + \gamma_c) (1+\max_{s\in\mathcal S} \E{\tau | S_0 = s})$, $L_2 = L_1\|x^*\|_c + \max_s \|V_0(s)\|$,  $\eta = 2(1-\frac{\ucm \gamma_c}{\lcm})$, $L_s = \frac{L}{\mu}$, $l = \frac{2}{\ucm^2}$, and $u = \frac{2}{\lcm^2}$. 

\begin{lemma}\label{lem:assmpbarf_cont}
    $\bar{F}$ satisfies Assumption~\ref{assmp:barf} with $A_3 = \gamma_c$, $B_3 = 0$, $L_1 := (L_F + \gamma_c) (1+ \max_{s\in\mathcal S} \E{\tau | S_0 = s})$ and $L_2 = L_1\|x^*\|_c + \max_s \|V_0(s)\|_c$, where for any given fixed state $s_0$, $\tau := \min\lrset{n > 0: S_n = s_0}$, and $V_x \in \R^{|\mathcal S|}$ is a solution of the Poisson's Equation~\eqref{eq:PE}.
\end{lemma}
\begin{proof}
    We first prove that Equation~\eqref{eq:barfassump} is satisfied with $A_3 = \gamma_c$ and $B_3 = 0$. To this end, consider the following inequalities. 
    \begin{align*}
        \|\bar{F}(x) - x^*\|_c 
        &= \|\bar{F}(x) - \bar{F}(x^*)\|_c \le \gamma_c \|x - x^*\|_c \tag{From Assumption~\ref{assmp:contraction}}.
    \end{align*}
    
    Finally, in the following, we show that Equation~\eqref{eq:Lipschitzness} holds with $L_1 := (L_F + \gamma_c) (1+\max_s\E{\tau | S_0 = s})$ and $L_2 = L_1\|x^*\|_c + \max_s \|V_0(s)\|_c$. The proof for this follows along the lines of that for \cite[Proposition 5.1]{haque2024stochastic}. For a fixed state $s_0$, define $$\tau := \min\lrset{n > 0: S_n = s_0}. $$ 
    Then, for $s \in \calS$, and $x\in\R^d$
    \[ V^*_x(s) = \E{\sum\limits_{k=0}^{\tau - 1} \lrp{F(x, S_k) - \bar{F}(x)} | S_0 = s} \]
    is a solution to the Poisson's equation~\eqref{eq:PE}. Moreover, it is well-known that all the solutions of Equation~\eqref{eq:PE} are of the form $V_x := V^*_x + c {\bf 1}$ \citep[Lemma 21.2.2]{douc2018markov}. Thus, it suffices to prove the desired properties for $V^*_x$.
    
    Since $V^*_x$ satisfies Equation~\eqref{eq:PE}, we have
    \[ V^*_x(s) - V^*_y(s) = F(x,s) - \bar{F}(x) + (PV^*_x)(s) - F(y, s) + \bar{F}(y) - (PV^*_y)(s). \]
    This implies
    \begin{align*}
        \|V^*_x(s) - V^*_y(s)\|_c 
        &\le \|F(x,s) - F(y,s)\|_c + \|\bar{F}(x) - \bar{F}(y)\|_c + \| (PV^*_x)(s) - (PV^*_y)(s) \|_c\\
        &\le L_F \|x-y\|_c + \gamma_c \|x-y\|_c + \| (PV^*_x)(s) - (PV^*_y)(s) \|_c.   \numberthis\label{eq:boundT2temp}
    \end{align*}
    We now bound the last term in the right-hand side above by using the form of $V^*_x$. To this end, observe that 
    \[ (PV^*_x)(s) = \E{\sum\limits_{k=1}^{\tau} \lrp{F(x, S_k) - \bar{F}(x)} | S_0 = s}, \]
    and hence, 
    \begin{align*}
        \| (PV^*_x)(s) - (PV^*_y)(s) \|_c &= \left\|\E{\sum\limits_{k=1}^{\tau} \lrp{F(x, S_k) - \bar{F}(x)} | S_0 = s} - \E{\sum\limits_{k=1}^{\tau} \lrp{F(y, S_k) - \bar{F}(y)} | S_0 = s} \right\|_c\\
        &= \left\| \E{\sum\limits_{k=1}^{\tau} \lrp{F(x, S_k) - F(y, S_k) + \bar{F}(y) - \bar{F}(x)} | S_0 = s} \right\|_c\\
        &\le \E{\sum\limits_{k=1}^{\tau} \|F(x, S_k) - F(y, S_k)\|_c + \|\bar{F}(y) - \bar{F}(x)\|_c ~|~ S_0 = s}\\
        &\le \E{\sum\limits_{k=1}^{\tau} \lrp{L_F \|x-y\|_c + \gamma_c\|x-y\|_c} ~|~ S_0 = s} \tag{Assumption~\ref{assmp:contraction}} \\
        &\le (L_F + \gamma_c) \E{\tau|S_0 = s} \|x-y\|_c\\
        &\le (L_F + \gamma_c) \lrp{\max\limits_{s\in\mathcal S}\E{\tau|S_0 = s}} \|x-y\|_c.
    \end{align*}
    Using this in the bound in Equation~\eqref{eq:boundT2temp} and rearranging, we get
    \[ \|V^*_x(s) - V^*_y(s)\|_c \le (L_F + \gamma_c)\lrp{1+ \max_s\E{\tau| S_0 = s}} \|x-y\|_c, \]
    proving that for each $s$, the vector $V^*_x(s)$ is $L_1$-Lipschitz in $x$, with $L_1 := (L_F + \gamma_c) (1+\max_s\E{\tau | S_0 = s})$. Since every solution $V_x(s)$ of the Poisson's equation is of the form $V_x = V^*_x + c {\bf 1}$ for some constant $c\in\R$, we also have that all the solutions $V_x(s)$ are $L_1$-Lipschitz.
    
    Finally, for $s\in\mathcal S$, any solution $V_{x^*}$ of the Poisson's Equation satisfies
    \begin{align*}
        \|V_{x^*}(s)\|_c  
        &= \|V_{x^*}(s) - V_0(s) + V_0(s)\|_c\\
        &\le \|V_{x^*}(s) - V_0(s)\|_c + \|V_0(s)\|_c\\
        &\le L_1\|x^*\|_c + \|V_0(s)\|_c\\
        &\le L_1\|x^*\|_c + \sup\limits_{s\in\mathcal S}\|V_0(s)\|_c. 
    \end{align*}
    Thus, $L_2 = L_1\|x^*\|_c + \max_s \|V_0(s)\|_c$. 
\end{proof}

\begin{lemma}\label{lem:assmpLyapunov_cont}
    The Moreau envelope, defined in Equation~\eqref{eq:moreau} is the Lyapunov function satisfying Assumption~\ref{assmp:lyapunov} with $\eta = 2(1- \gamma_c (1 + \mu \ucs^2)^{\frac{1}{2}} /(1 + \mu \lcs^2)^{\frac{1}{2}})$, $L_s= L/\mu$, $l = 2 (1 + \mu \lcs^2)$, and $u = 2 (1 + \mu \ucs^2)$.  
\end{lemma}
\begin{proof}
    $M$ defined in Equation~\eqref{eq:moreau} is $\frac{L}{\mu}$-smooth with respect to $\|\cdot\|_s$ \cite[Lemma 2.1]{chen2020finite}, implying that Equation~\eqref{eq:Smoothness} holds with $L_s = L/\mu$. 
    Next, Lemma~\ref{lem:negdriftofMoreau} implies that Equation~\eqref{eq:negdrift} holds with $\eta = 2(1- \gamma_c \ucm / \lcm)$.  
    
    To see that Equation~\eqref{eq:0gradatmin} holds, recall that for $x\in\R^d$, there exists a norm such that $M(x) = \|x\|^2_M/2$. With this, 
    \[\nabla M(x) = \frac{1}{2}\nabla \|x\|^2_M = \|x\|_M \nabla \|x\|_M \implies \nabla M(x) \rvert_{x=0}  = 0.\] 
    Finally, to see that Equation~\eqref{eq:normMcM} holds, consider the following: 
    \begin{align*}
        \frac{\lcm^2}{2} \|x\|^2_c \le M(x) = \frac{1}{2} \|x\|^2_M \le \frac{\ucm^2}{2} \|x\|^2_c. 
    \end{align*}
    The above implies that 
    \[ \frac{2}{\ucm^2} M(x) \le \|x\|^2_c \le \frac{2}{\lcm^2} M(x),\]
    proving Equation~\eqref{eq:normMcM}, with $l = 2/\ucm^2$ and $u = 2/\lcm^2$. Substituting $\lcm = (1 + \mu \ucs^2)^{-\frac{1}{2}}$ and $\ucm = (1 + \mu \lcs^2)^{-\frac{1}{2}}$, we get the result.
\end{proof}

\begin{lemma}\label{lem:negdriftofMoreau}
    For $x\in\R^d$ and the fixed point $x^*$, the Moreau envelope defined in Equation~\eqref{eq:moreau} satisfies 
    \[ \langle \nabla M(x-x^*), \bar{F}(x) - x \rangle \le - 2 \lrp{ 1 -  \gamma_c\lrp{\frac{1 + \mu \ucs^2 }{1 + \mu \lcs^2}}^\frac{1}{2}} M(x - x^*). \]
\end{lemma}
\begin{proof}
    Consider the following inequalities:
    \begin{align*}
        & \langle \nabla M(x-x^*) , \bar{F}(x) - x \rangle \\
        &= \|x-x^*\|_M \langle \nabla \|x-x^*\|_M, \bar{F}(x) - x \rangle\\
        &= \|x-x^*\|_M \langle \nabla \|x-x^*\|_M, \bar{F}(x) - \bar{F}(x^*) + \bar{F}(x^*) - x \rangle\\
        &= \|x-x^*\|_M  \langle \nabla \|x-x^*\|_M, \bar{F}(x) - \bar{F}(x^*) \rangle - \|x-x^*\|_M  \langle \nabla \|x-x^*\|_M, x - x^* \rangle \\
        &\le \|x-x^*\|_M \| \nabla \|x-x^*\|_M \|^*_M \|\bar{F}(x) - \bar{F}(x^*)\|_M - \|x-x^*\|^2_M  \tag{Cauchy-Schwarz and Convexity of $\|\cdot\|_M$} \\
        &\le \|x-x^*\|^2_M \frac{\ucm\gamma_c}{\lcm} - \|x-x^*\|^2_M \tag{$\| \nabla \|\cdot\|_M \|^*_M \le 1$} \\
        &= - \lrp{1-\frac{\ucm \gamma_c}{\lcm}} \|x-x^*\|^2_M \\
        &= - 2\lrp{1-\frac{\ucm \gamma_c}{\lcm}} M(x-x^*).
    \end{align*}
    On substituting $\ucm = (1+\mu\lcs^2)^{-1/2}$ and $\lcm = (1+ \mu \ucs^2)^{-1/2}$, we get the desired inequality.
\end{proof}

\subsection{Concentration for linear SA}\label{app:linearSA}
\begin{lemma}\label{lem:verifyingassmp6linearsa}
    The operator $F_\beta(\cdot, \cdot)$ defined in Section~\ref{sec:linearSA} is Lipschitz with respect to $\|\cdot\|_{\bar{P}}$ in the first argument, uniformly over the second argument, that is,  there exists a constant $L_F > 0$ such that
    \[ \sup\limits_{s\in\calS} \big\{ \|F(x_1, s) - F(x_2, s)\|_{\bar{P}} \big\} \le L_F\|x_1 - x_2\|_{\bar{P}}. \]
\end{lemma}
\begin{proof}Consider the following inequalities:
    \begin{align*}
        \|F_\beta(x_1,s) - F_\beta(x_2, s)\|_{\bar{P}}
        &= \|\beta A(s)(x_1-x_2) + x_1 - x_2\|_{\bar{P}}\\
        &= \|(\beta A(s) + I_n) (x_1 - x_2)\|_{\bar{P}}\\
        &\le \|\beta A(s) + I_n\|_{\bar{P}} \|x_1 - x_2\|_{\bar{P}}\\
        &\le \lrp{\beta \|A(s)\|_{\bar{P}} + \|I_n\|_{\bar{P}}} \|x_1 - x_2\|_{\bar{P}},
    \end{align*}
    where the inequality  in the lemma statement holds with $L_F = \sup_s (\beta \|A(s)\|_{\bar{P}} + \|I_n\|_{\bar{P}})$.
\end{proof}

\section{Proof details for Theorem~\ref{th:onestepbootstrap}}\label{app:proofdetails_onestepbootstrap}
In this section, we prove Theorem~\ref{th:onestepbootstrap}. Our proof uses Lyapunov drift analysis to construct a non-negative supermartingale, which then gives the desired concentration inequalities. We begin by introducing some notation that is used throughout this appendix.

Recall the constants from Appendix~\ref{app:const_ball}. Let $\lrset{T_k}_{k\ge 0}$ be the given non-decreasing sequence such that, for all $k\geq 0$,
\[ \|x_k - x^*\|^2_c \le T_k, \qquad a.s. \]
For $\theta > 0$ (to be chosen later) and $\alpha_k = \alpha / (k+h)^z$ for $z\in (0,1]$, define
\[ \lambda_k :=  \tfrac{\theta}{\alpha_k T_k}. \]
Using Equation~\eqref{eq:PE} with $x = x_k$ and $s = S_k$, we also get
\begin{align*} 
    V_{x_k}(S_k) 
    &= F(x_k,S_k) - \bar{F}(x_k) + (PV_{x_k})(S_k) = F(x_k,S_k) - \bar{F}(x_k) + \E{V_{x_k}(S_{k+1})| S_{k}, x_{k}}.\numberthis\label{eq:PETerm}
\end{align*}
With $V_{x_k}(\cdot)$ as above, define 
\[ d_k := \langle \nabla M(x_k - x^*), \E{V_{x_k}(S_k) | S_{k-1}, x_{k-1}, Z_{k-1}} \rangle, \numberthis\label{eq:defdk}\]
and let 
\[ Z_k := \log \E{  e^{\lambda_{k}\lrp{M(x_{k} - x^*) + \alpha_{k-1} d_{k} + \frac{\alpha_{k-1} L_sL_2 }{4} }}  }, \numberthis\label{eq:defZk}\]
where, recall that $M(\cdot)$ is the Lyapunov function from Assumption~\ref{assmp:lyapunov}. Note that the Lyapunov function $Z_k$ introduced above is (the logarithm of) a modified version of the moment-generating function of the squared error, where the adjustment term $\alpha_{k-1}d_k + \alpha_{k-1}L_sL_2/4$ is introduced to handle the bias present at each time due to the presence of the Markov noise in our setting.

The following Proposition establishes a simple recursive inequality for the Lyapunov function $Z_k$.

\begin{proposition}\label{lem:rec2}
    For $k\ge 0$, we have
    \[ Z_{k+1}\le \frac{\alpha_k}{\alpha_{k+1}}\left(1- \frac{\alpha_k \eta}{2}\right) Z_k + 2 \lambda_k \alpha^2_k  D_2.\]
    %\[ Z_{k+1 } \le e^{-\alpha_k \frac{\alpha\tilde{\gamma}_c - 1}{\alpha}} Z_k + 2 \lambda_k \alpha^2_k (1+\|x^*\|_c)^2 D_2. \]
\end{proposition}

Here, the choice of the Lyapunov function is the key to obtaining such a simple recursive inequality. In order to prove Proposition~\ref{lem:rec2} above, we first obtain an intermediate inequality in the following lemma, which does most of the heavy lifting in the proof.

\begin{lemma}\label{lem:rec1}
    For $k\in\N$, let $\alpha_k = \alpha/(k+h)^z$ for some $z\in (0,1]$, where $\alpha \ge 0$, $h \ge 8$, satisfy the conditions in Assumption~\ref{assmp:stepsize}. Let $\mathcal F_k := \sigma(x_{k-1}, S_{k-1}, Z_{k-1})$. Then, the following inequality holds. 
\begin{align*}
        &\E{ \exp\lrset{{\lambda_{k+1}M(x_{k+1} - x^*) + \alpha_k \lambda_{k+1} d_{k+1} + \lambda_{k+1}\tfrac{\alpha_k L_sL_2 }{4 } }} \big| \mathcal F_k }\\
        &\qquad \le \lrp{\exp\lrset{ { \lambda_{k}M(\xk - x^*) + \alpha_{k-1}\lambda_{k} d_k + \lambda_{k}\tfrac{\alpha_{k-1} L_sL_2}{4}}}}^{\frac{\alpha_{k}}{\alpha_{k+1}}\left(1-\tfrac{\alpha_k \eta}{2}\right)} \exp\lrset{2 \lambda_k \alpha^2_k \lrp{C_2 + \tfrac{\eta L_sL_2}{4} }},%\numberthis\label{eq:expbound4}
\end{align*}
In addition, for all $k\ge 0$, we have
\[ \frac{\alpha_k}{\alpha_{k+1}}\left(1-\frac{\alpha_k \eta}{2}\right) \in [0,1]. \]
\end{lemma}
The proof is given in Appendix~\ref{app:proof_lem:rec1}. 

\begin{proof}[Proof of Proposition~\ref{lem:rec2}]
Taking expectation on both sides of the inequality in Lemma~\ref{lem:rec1}, and using Jensen's inequality for bounding the average of the first term we get 
\begin{align*}
        &\E{ \exp\lrset{{\lambda_{k+1}M(x_{k+1} - x^*) + \alpha_k \lambda_{k+1}d_{k+1} + \lambda_{k+1}\tfrac{\alpha_k L_sL_2 }{4} }}}\\
        &\quad \le \lrp{\E{\exp\lrset{ {\lambda_{k} M(\xk - x^*) + \alpha_{k-1}\lambda_{k}d_k + \lambda_{k}\tfrac{\alpha_{k-1} L_sL_2}{4} }}}}^{\frac{\alpha_k}{\alpha_{k+1}}\left(1-\tfrac{\alpha_k \eta}{2}\right)}  \exp\lrset{2 \lambda_k \alpha^2_k D_2}.
\end{align*}  
Now, taking the logarithm on both sides, we get the desired recursion. 
\end{proof}

By opening the recursive inequality given in Proposition \ref{lem:rec2}, we obtain the following corollary.

\begin{corollary}\label{lem:boundonZk}
    For $k\ge 0$, $z\in(0,1]$, and  $\alpha_k = \alpha/(k+h)^z$, $Z_k$ defined in Equation~\eqref{eq:defZk} satisfies the following:  
\begin{align*}
    Z_k \le 
    \begin{cases}
        Z_0\lrp{ \frac{h}{k+h} }^{\alpha\tilde{\gamma}_c - 1} + \frac{8 e \alpha\theta D_2 }{\|x_0-x^*\|^2_c \lrp{\tilde{\gamma}_c\alpha - 1} }, & \text{ if } z = 1,\,\, \alpha > \frac{1}{\tilde{\gamma}_c},\\
        Z_0\lrp{ \frac{k+h}{h} }^z ~ e^{-\frac{\tilde{\gamma}_c\alpha}{1-z}~\lrp{(k+h)^{1-z} - h^{1-z}} } +  \frac{4 D_2\theta}{\tilde{\gamma}_c\|x_0-x^*\|^2_c  }, & \text{ if } z\in (0,1), \,\, \alpha > 0, \,\, h \ge \lrp{\frac{2z}{\alpha\tilde{\gamma}_c}}^\frac{1}{1-z}.
    \end{cases}
\end{align*}
\end{corollary}
We present the proof in Appendix~\ref{app:proof_lem:boundonZk} for completeness. 

\smallskip

\noindent Furthermore, we have the following bound on $Z_0$. 

\begin{lemma}\label{lem:boundonZ0} We have
    \[ Z_0 \le \frac{\eta\lcs^2}{64 ul L^2_s \alpha_0 L^2_1 L_2 }\lrp{2\lcs^2 L_2 +  2 u \alpha_{-1} (L_1 + L_2)  L_2 L_s + \lcs^2 \alpha_{-1} l L_s L^2_1 }. \]
\end{lemma}
\begin{proof}
    Recall, 
\begin{align*}
    Z_0 &= \log \E{ e^{\lambda_0 \lrp{ M(x_0-x^*)  + \alpha_{-1} d_0 + \frac{\alpha_{-1} L_sL_2}{4}} }}\\
    %&= \lambda_0 \lrp{ M(x_0 - x^*)  + \alpha_{-1} d_0 + \frac{\alpha_{-1} L_SL_2}{4}} \\
    &= \frac{\theta}{\alpha_0 T_0} \lrp{ M(x_0 - x^*)  + \alpha_{-1} d_0 + \frac{\alpha_{-1} L_sL_2}{4}} \tag{$x_0$ is deterministic}\\
    %&\le \frac{\theta}{\alpha_0 \|x_0-x^*\|^2_c} \lrp{ M(x_0-x^*)  + \alpha_{-1} \abs{d_0} + \frac{\alpha_{-1} L_sL_2}{4}}\\
    &\le \frac{\theta}{\alpha_0 \|x_0-x^*\|^2_c} \lrp{ M(x_0-x^*)  \lrp{1+\frac{\alpha_{-1} u L_s (L_1+L_2)}{\lcs^2}} + \frac{\alpha_{-1}  L_sL_2}{2}}\tag{Lemma~\ref{lem:boundonmoddk}}\\
    &\le \frac{\eta \lcs^4 \|x_0 - x^*\|^2_c}{32 L^2_s ul  \alpha_0\lrp{L^2_1\|x_0 - x^*\|^2_c + L^2_2 + B^2_2} } \lrp{1+\frac{\alpha_{-1} u L_s (L_1 + L_2)}{\lcs^2}} \\
    &\qquad + \frac{\eta \lcs^4 \alpha_{-1} L_2L_s}{64 L^2_s u \alpha_0 \lrp{L^2_1\|x_0 - x^*\|^2_c + L^2_2 + B^2_2}  }\\
    %&= \frac{\eta \lcs^4 \|x_0 - x^*\|^2_c}{32 ul L^2_s \alpha_0 L^2_1 \lrp{\|x_0 - x^*\|^2_c + L^2_2/L^2_1} } + \frac{\eta \lcs^2 \alpha_{-1} \|x_0 - x^*\|^2_c (L_1 + L_2)}{32 l L_s \alpha_0 L^2_1 \lrp{\|x_0 - x^*\|^2_c + L^2_2/L^2_1} } \\
    %&\qquad + \frac{\eta \lcs^4 \alpha_{-1}}{64 u L_s L_2 \lrp{L^2_1\|x_0 - x^*\|^2_c/L^2_2 + 1} \alpha_0 }\\
    &\le \frac{\eta \lcs^4 }{32 ul L^2_s \alpha_0 L^2_1  } + \frac{\eta \lcs^2 \alpha_{-1} (L_1 + L_2)}{32 l L_s \alpha_0 L^2_1  } + \frac{\eta \lcs^4 \alpha_{-1} }{64 u L_s L_2  \alpha_0 }\\
    &=  \frac{\eta\lcs^2}{64 ul L^2_s \alpha_0 L^2_1 L_2 }\lrp{2\lcs^2 L_2 +  2 u \alpha_{-1} (L_1 + L_2)  L_2 L_s + \lcs^2 \alpha_{-1} l L_s L^2_1 },
\end{align*}
completing the proof.
\end{proof}

Finally, the recursion in Lemma~\ref{lem:rec1} suggests a non-negative supermartingale. This is used to arrive at a maximal concentration inequality (after an application of Ville's inequality). Define 
\[ \bar{M}_k := \exp\lrset{  \lambda_{k} M(\xk - x^*) + \alpha_{k-1}\lambda_{k}d_k + \lambda_{k}\frac{\alpha_{k-1} L_s L_2}{4} }  \exp\lrset{-D_3\sum\limits_{i=0}^{k-1}\alpha_i }. \numberthis\label{eq:defMbar} \]

\begin{lemma}\label{lem:martingale}  Process $\lrset{\bar{M}_k}_{k\geq 1}$ defined above is a non-negative supermartingale with respect to the filtration $\{\mathcal F_k\}_{k\geq 1}$ defined as $\mathcal F_k := \sigma(S_{k-1}, x_{k-1}, Z_{k-1},\dots,S_0, x_0, Z_0)$.
\end{lemma}
\begin{proof}
The process $\lrset{\bar{M}_k}_{k\geq 1}$ is nonnegative and adapted to $\{\mathcal F_k\}_{k\geq 1}$ since for each $k\ge 0$, $E_k$, $x_k$, and $d_k$ are $\mathcal F_k$-measurable. Next, from Lemma~\ref{lem:rec1}, we have 
    \begin{align*}
        &\E{ e^{\lambda_{k+1}\lrp{M(x_{k+1} - x^*) + \alpha_k d_{k+1} + \frac{\alpha_k L_sL_2 }{4} }} \bigg| \mathcal F_k }\\
        &\le e^{ \lambda_{k}\lrp{ M(\xk - x^*) + \alpha_{k-1}d_k + \frac{\alpha_{k-1} L_sL_2}{4} }}  e^{2 \lambda_k \alpha^2_k D_2 }\\
        &= e^{ \lambda_{k}\lrp{ M(\xk - x^*) + \alpha_{k-1}d_k + \frac{\alpha_{k-1} L_sL_2}{4} }}  e^{\frac{2\alpha^2_k D_2\theta}{\alpha_k T_k} } \tag{$\because  \lambda_k = \frac{\theta}{\alpha_k T_k}$}\\
        &\le e^{ \lambda_{k}\lrp{ M(\xk - x^*) + \alpha_{k-1}d_k + \frac{\alpha_{k-1} L_sL_2}{4} }}  e^{\frac{2\alpha_k D_2\theta}{T_0} } \tag{$T_k$ are non-decreasing}\\
        &\le e^{ \lambda_{k}\lrp{ M(\xk - x^*) + \alpha_{k-1}d_k + \frac{\alpha_{k-1} L_sL_2}{4} }}  e^{\frac{2\alpha_k D_2 }{\|x_0- x^*\|^2_c} \frac{\eta \lcs^4 \|x_0 - x^*\|^2_c}{32 L^2_s u  \lrp{L^2_1\|x_0 - x^*\|^2_c + L^2_2 + B^2_2} }} \\
        &\le e^{ \lambda_{k}\lrp{ M(\xk - x^*) + \alpha_{k-1}d_k + \frac{\alpha_{k-1} L_sL_2}{4} }}  e^{\alpha_k\frac{ D_2 \eta \lcs^4 }{16 L^2_s u  \lrp{L^2_1\|x_0 - x^*\|^2_c + L^2_2 + B^2_2} }}\\
        &= e^{ \lambda_{k}\lrp{ M(\xk - x^*) + \alpha_{k-1}d_k + \frac{\alpha_{k-1} L_sL_2}{4} }}  e^{\alpha_k D_3}.
    \end{align*}
 From the above inequality, it follows that $\E{\bar{M}_{k+1} | \mathcal F_k} \le \bar{M}_k$ for all $k\ge 1$.
\end{proof}

With the above results for the Lyapunov function $Z_k$ and the supermartingale construction above, we are now ready to prove Theorem~\ref{th:onestepbootstrap}. We prove the two parts separately; Part~\ref{part1} is proven in Appendix~\ref{app:proof_onestepbootstrappart1}, and Part~\ref{partz} in Appendix~\ref{app:proof_onestepbootstrappartz}. 
\subsection{Proof for Theorem~\ref{th:onestepbootstrap}, Part~\ref{part1} ($z = 1$)}\label{app:proof_onestepbootstrappart1}
\begin{proof}
    We let $\tilde{\gamma}_c := \eta/2$.  For any $\epsilon > 0$, $K \ge 0$, and for $d_k$, $D_3$, $\bar{M}_k$, and $Z_k$ defined in Equation~\eqref{eq:defdk}, Appendix~\ref{app:const_ball}, and equations~\eqref{eq:defMbar} and~\eqref{eq:defZk},  respectively, the following inequalities hold: 
    \begin{align*}
        &\Prob{\bigcup\limits_{ k \ge K} \lrset{ \lambda_k \lrp{ M(x_k - x^*) + \alpha_{k-1} d_k + \frac{\alpha_{k-1} L_sL_2}{4}} - D_3\sum\limits_{i=0}^{k-1} \alpha_i   > \epsilon } }\\
        &= \Prob{\bigcup\limits_{ k \ge K} \lrset{ e^{\lambda_k \lrp{ M(x_k - x^*) + \alpha_{k-1} d_k + \frac{\alpha_{k-1} L_sL_2}{4}} - D_3\sum\limits_{i=0}^{k-1} \alpha_i }  > e^{\epsilon}} } \\
        &\le \E{\bar{M}_K} e^{-\epsilon} \\
        &= e^{Z_K - D_3 \sum\limits_{i=0}^{K-1}\alpha_i - \epsilon} \\
        &\le e^{Z_0\lrp{\frac{h}{K+h}}^{\alpha\tilde{\gamma}_c - 1} + \frac{8 \alpha e \theta D_2}{(\alpha \tilde{\gamma} - 1)\|x_0 - x^*\|^2_c}  - D_3 \sum\limits_{i=0}^{K-1}\alpha_i - \epsilon} \tag{Corollary~\ref{lem:boundonZk}}.
    \end{align*}
Choose $\epsilon$ so that the above equals $\tilde{\delta}$, i.e., 
\[ \epsilon = \log\left(\frac{1}{\tilde{\delta}}\right) + Z_0\lrp{\frac{h}{K+h}}^{\alpha\tilde{\gamma}_c - 1} + \frac{8\alpha e \theta D_2}{(\alpha \tilde{\gamma} - 1)\|x_0 - x^*\|^2_c}  - D_3 \sum\limits_{i=0}^{K-1}\alpha_i. \]
With this choice, we have that with probability at least $1-\tilde{\delta}$, for all $k \ge K$, 
\begin{align*}
    &\lambda_k \lrp{M(x_k - x^*) + \alpha_{k-1} d_k + \frac{\alpha_{k-1} L_sL_2}{4} } \\
    &\qquad \le \log\left(\frac{1}{\tilde{\delta}}\right) + Z_0\lrp{\frac{h}{K+h}}^{\alpha\tilde{\gamma}_c - 1}+ \frac{8\alpha e \theta D_2}{(\alpha \tilde{\gamma} - 1)\|x_0 - x^*\|^2_c} + D_3 \sum\limits_{i=K}^{k-1}\alpha_i \\ 
&\qquad \le \log\left(\frac{1}{\tilde{\delta}}\right) + Z_0\lrp{\frac{h}{K+h}}^{\alpha\tilde{\gamma}_c - 1}+ \frac{8\alpha e \theta D_2}{(\alpha \tilde{\gamma} - 1)\|x_0 - x^*\|^2_c}  + D_3 \alpha \log\left(\frac{k-1+h}{K-1+h}\right),
\end{align*}
which follows, since
\[ \sum\limits_{i=K}^{k-1} \alpha_i \le \alpha \log\left(\frac{k-1+h}{K-1+h}\right), \]
and $h\ge 1$.
The above bound then implies that with probability at least $1 - \tilde{\delta}$, for all $k\ge K$, 
\begin{align*}
    M(x_k - x^*) & \le \frac{\alpha_k T_k}{\theta} \left[ \log\left(\frac{1}{\tilde{\delta}}\right) + Z_0\lrp{\frac{h}{K+h}}^{\alpha\tilde{\gamma}_c - 1} + \frac{8\alpha e \theta D_2}{(\alpha \tilde{\gamma} - 1)\|x_0 - x^*\|^2_c}  \right. \\
    &\qquad\qquad\qquad \left. + {D_3 \alpha } \log\left(\frac{k-1+h}{K-1+h}\right) \right] + \alpha_{k-1}|d_k| - \frac{\alpha_{k-1} L_sL_2}{4}. \numberthis\label{eq:tempboundM} 
\end{align*}
We now bound the last two terms on the right-hand side.
\begin{align*}
    \alpha_{k-1}|d_k| - \frac{\alpha_{k-1} L_sL_2 }{4}
    &\le \alpha_{k-1} \frac{u L_s (L_1 + L_2)}{\lcs^2} M(x_k - x^*) \tag{Lemma~\ref{lem:boundonmoddk}}\\
    &\le \alpha_{k}\frac{2 u L_s (L_1 + L_2)}{ \lcs^2} M(x_k - x^*)  \tag{$\alpha_{k-1}\le 2\alpha_k$}\\
    &\le \alpha_{k}\frac{2 u L_s (L_1 + L_2)}{ l\lcs^2} \|x_k - x^*\|^2_c\\
    &\le \alpha_{k}\frac{ 2 u L_s (L_1 + L_2)}{l \lcs^2} T_k.
\end{align*}
Substituting this in the bound on $M(x_k - x^*)$ in Equation~\eqref{eq:tempboundM}, and re-arranging terms, we get 
\begin{align*}
    M(x_k - x^*) & \le \frac{\alpha_k T_k}{\theta} \left[ \log\left(\frac{1}{\tilde{\delta}}\right) + Z_0\lrp{\frac{h}{K+h}}^{\alpha\tilde{\gamma}_c - 1} + \frac{8\alpha e \theta D_2}{(\alpha \tilde{\gamma} - 1)\|x_0 - x^*\|^2_c}  \right. \\
    &\qquad\qquad \left. + {D_3 \alpha } \log\left(\frac{k-1+h}{K-1+h}\right) + \frac{ 2u L_s (L_1 + L_2) \theta }{ l\lcs^2} \right]
\end{align*}
with probability at least $1 - \tilde{\delta}$. Next, since $\|x_k - x^*\|^2_c \le u M(x_k - x^*)$, the following also holds with probability at least $1- \tilde{\delta}$:  for all $k\ge K$,
\begin{align*}
    \|x_k - x^*\|^2_c & \le \frac{u\alpha_k T_k}{\theta} \left[ \log\left(\frac{1}{\tilde{\delta}}\right) + Z_0\lrp{\frac{h}{K+h}}^{\alpha\tilde{\gamma}_c - 1} + \frac{8\alpha e \theta D_2}{(\alpha \tilde{\gamma} - 1)\|x_0 - x^*\|^2_c}  \right. \\
    &\qquad\qquad \left. + {D_3 \alpha } \log\left(\frac{k-1+h}{K-1+h}\right) + \frac{ 2u L_s (L_1 + L_2) \theta }{ l\lcs^2} \right] . 
\end{align*}
Now, recall from Lemma~\ref{lem:boundonZ0} that 
\[ Z_0 \le \frac{\eta\lcs^2}{64 ul L^2_s \alpha_0 L^2_1 L_2 }\lrp{2\lcs^2 L_2 +  2 u \alpha_{-1} (L_1 + L_2)  L_2 L_s + \lcs^2 \alpha_{-1} l L_s L^2_1 }. \]
Using the above, we have
\begin{align*}
    \|x_k - x^*\|^2_c & \le \frac{T_k}{k+h} \left[ \bar{b}_1\log\left(\frac{1}{\tilde{\delta}}\right) + \bar{b}_2 + \bar{b}_3\lrp{\frac{h}{K+h}}^{\alpha\tilde{\gamma}_c - 1}    + \bar{b}_4  \log\left(\frac{k-1+h}{K-1+h}\right)  \right], 
\end{align*}
where
\begin{align*}
    &\bar{b}_1 = \frac{u\alpha}{\theta}, \quad \bar{b}_2 = \frac{u\alpha}{\theta}\lrp{ \frac{8\alpha e \theta D_2}{(\alpha \tilde{\gamma} - 1)\|x_0 - x^*\|^2_c} + \frac{ 2u L_s (L_1 + L_2) \theta }{ l\lcs^2}},\quad \bar{b}_4 = \frac{D_3 u\alpha^2}{\theta},\\
    & \bar{b}_3 = \frac{\eta\lcs^2 u\alpha}{64 ul L^2_s \alpha_0 L^2_1 L_2 \theta }\lrp{2\lcs^2 L_2 +  2 u \alpha_{-1} (L_1 + L_2)  L_2 L_s + \lcs^2 \alpha_{-1} l L_s L^2_1 },
\end{align*}
completing the proof for Part~\ref{part1}. 
\end{proof}
\subsection{Proof for Theorem~\ref{th:onestepbootstrap}, Part~\ref{partz} ($z \in (0,1)$)}\label{app:proof_onestepbootstrappartz}
\begin{proof}
As earlier, we let $\tilde{\gamma}_c := \eta/2$.   Observe that for  $\epsilon_k > 0$ for all $k\ge 0$, and $K \ge 0$, for $d_k$ and $Z_k$ defined in equations~\eqref{eq:defdk} and~\eqref{eq:defZk}, respectively, the following inequalities hold: 
\begin{align*}
    &\Prob{\bigcup\limits_{k\ge K}\left\{ \lambda_k \lrp{M(x_k- x^*) + \alpha_{k-1}d_k + \frac{\alpha_{k-1} L_sL_2}{4} }  \ge \epsilon_k \right\} }\\
    & = \Prob{\bigcup\limits_{k\ge K} \left\{ e^{\lambda_k \lrp{M(x_k- x^*) + \alpha_{k-1}d_k + \frac{\alpha_{k-1} L_s L_2}{4} } } \ge e^{\epsilon_k} \right\} } \\
    &\le \sum\limits_{k=K}^\infty \Prob{ e^{\lambda_k \lrp{M(x_k- x^*) + \alpha_{k-1}d_k + \frac{\alpha_{k-1} L_s L_2}{4} }} \ge e^{\epsilon_k} }  \tag{Union bound}\\
    &\le \sum\limits_{k=K}^\infty \E{ e^{\lambda_k \lrp{M(x_k- x^*) + \alpha_{k-1}d_k + \frac{\alpha_{k-1} L_s L_2}{4} }}}e^{-\epsilon_k} \\
    &= \sum\limits_{k=K}^\infty e^{Z_k - \epsilon_k} . \numberthis\label{eq:boundonP_union}
\end{align*}
Now, we choose $\epsilon_k$ so that the above bound is at most $\tilde{\delta} $, i.e., we choose
\[ \epsilon_k = \log\left(\frac{(k+1)^2 \pi^2}{6\tilde{\delta}(K+1)}\right) + Z_k. \]
With this choice, from Equation~\eqref{eq:boundonP_union}, we have  with probability at least $1-\tilde{\delta}$, for all $k\ge K$
\begin{align*} 
    M(x_k - x^*) 
    &\le \frac{Z_k + \log\left(\frac{(k+1)^2 \pi^2}{6\tilde{\delta} (K+1)}\right)}{\lambda_k} 
    + \alpha_{k-1}|d_k| - \frac{\alpha_{k-1} L_sL_2}{4} \\
    &\le \frac{Z_k + \log\left(\frac{(k+1)^2 \pi^2}{6\tilde{\delta} (K+1)}\right)}{\lambda_k} 
    + \alpha_{k-1}\frac{u L_s (L_1 + L_2)}{\lcs^2} M(x_k - x^*) \tag{Lemma~\ref{lem:boundonmoddk}}\\
    &\le \frac{T_k \alpha_k}{\theta}\lrp{Z_k + \log\left(\frac{(k+1)^2 \pi^2}{6\tilde{\delta} (K+1)}\right)} 
    + \alpha_{k-1}\frac{u L_s (L_1 + L_2)}{\lcs^2 l} \|x_k - x^*\|^2_c \tag{Choice of $\lambda_k$}\\
    &\le \frac{T_k \alpha_k}{\theta}\lrp{Z_k + \log\left(\frac{(k+1)^2\pi^2}{6\tilde{\delta} (K+1)}\right)} 
    + \alpha_{k-1}\frac{u L_s (L_1 + L_2) T_k}{\lcs^2 l} \tag{$\|x_k - x^*\|^2_c \leq T_k$}\\
    &\le \frac{T_k \alpha_{k-1}}{\theta}\lrp{Z_k + \log\left(\frac{(k+1)^2\pi^2}{6\tilde{\delta} (K+1)}\right)} 
    + \alpha_{k-1}\frac{ u L_s (L_1 + L_2) T_k}{\lcs^2 l} \tag{$\because \alpha_k \le \alpha_{k-1}$}\\
    &= \frac{T_k \alpha_{k-1}}{\theta}\lrp{Z_k + \log\left(\frac{(k+1)^2\pi^2}{6\tilde{\delta} (K+1)}\right) 
    + \frac{u L_s (L_1 + L_2) \theta}{\lcs^2 l} }\\
    &\le \frac{T_k \alpha_{k-1}}{\theta}\left[ Z_0\lrp{ \frac{K+h}{h} }^z ~ e^{-\frac{\tilde{\gamma}_c\alpha}{1-z}~\lrp{(k+h)^{1-z} - h^{1-z}} } +  \frac{4 D_2\theta  }{\tilde{\gamma}_c\|x_0-x^*\|^2_c  } \right. \\
    & \qquad \qquad \qquad \qquad \left. + \log\left(\frac{(k+1)^2\pi^2}{6\tilde{\delta} (K+1)}\right) 
    + \frac{u L_s(L_1+L_2) \theta}{\lcs^2 l} \right] \tag{Corollary~\ref{lem:boundonZk}}.%\\
    %&\le \frac{T_k \alpha_{k-1}}{\theta}\left[ Z_0\lrp{ \frac{K+h}{h} }^z ~ e^{-\frac{\tilde{\gamma}_c\alpha}{1-z}~\lrp{(k+h)^{1-z} - h^{1-z}} } +  \frac{4 D_2\theta}{\tilde{\gamma}_c\|x_0-x^*\|^2_c  } \right. \\
    %&  \qquad \qquad \qquad \left. + %\log\frac{\lrp{k+1}^2\pi^2}{6\tilde{\delta} (K+1)} + \frac{u L_s (L_1+ L_2) \theta}{\lcs^2 l} \right] \tag{$\because k(k+1) < (k+1)^2$}.
 \end{align*}
Next, since
\[ \frac{1}{u} \|x_k - x^*\|^2_c \le M(x_k-x^*),\]
we get that for $K\ge 0$, with probability at least $1 - \tilde{\delta}$, for all $k\ge K$, 
\begin{align*}
    \|x_k-x^*\|^2_c &\le \frac{T_k \alpha_{k-1} u}{\theta}\left[ Z_0\lrp{ \frac{K+h}{h} }^z ~ e^{-\frac{\tilde{\gamma}_c\alpha}{1-z}~\lrp{(k+h)^{1-z} - h^{1-z}} } +  \frac{4 D_2\theta}{\tilde{\gamma}_c\|x_0-x^*\|^2_c  } \right. \\
    &  \qquad \qquad \qquad \left. + \log\left(\frac{\lrp{k+1}^2\pi^2}{6\tilde{\delta} (K+1)}\right) 
    + \frac{u L_s (L_1+ L_2) \theta}{\lcs^2 l} \right].
\end{align*}

Also, recall from Lemma~\ref{lem:boundonZ0}, that
\[Z_0 \le \frac{\eta\lcs^2}{64 ul L^2_s \alpha_0 L^2_1 L_2 }\lrp{2\lcs^2 L_2 +  2 u \alpha_{-1} (L_1 + L_2)  L_2 L_s + \lcs^2 \alpha_{-1} l L_s L^2_1 }.\]
Using the above bound, we get that
\begin{align*}
    \|x_k &- x^*\|^2_c \\
    &\le\frac{T_k}{(k+h-1)^z} \left[ \bar{c}_1\log\lrp{\frac{1}{\tilde{\delta}}} + 2\bar{c}_1\log\lrp{\frac{k+1}{\sqrt{K+1}}} + \bar{c}_2 + \bar{c}_3 \lrp{\frac{K+h}{h}}^z e^{\frac{-\tilde{\gamma}_c \alpha}{1-z}\lrp{(k+h)^{1-z} - h^{1-z}}} \right],
\end{align*}
where
\begin{align*}
    &\bar{c}_1 = \frac{\alpha u}{\theta}, \quad \bar{c}_2 = \frac{4 \alpha u D_2 }{\tilde{\gamma}_c\|x_0 - x^*\|^2_c } + \frac{u^2 L_s (L_1 + L_2) \alpha}{\lcs^2 l} + \frac{\alpha u}{\theta}\log\frac{\pi^2}{6}, \\
    &\bar{c}_3 = \frac{\eta\lcs^2 \alpha }{64 \theta l L^2_s \alpha_0 L^2_1 L_2 }\lrp{2\lcs^2 L_2 +  2 u \alpha_{-1} (L_1 + L_2)  L_2 L_s + \lcs^2 \alpha_{-1} l L_s L^2_1 },
\end{align*}
completing the proof.
\end{proof}

\subsection{Constructing the Bias-adjusted Lyapunov Function: Supporting Results}\label{app:LyapunovFunction}
Recall the Lyapunov function $M$  from Assumption~\ref{assmp:lyapunov}, and the constants $A_{13} := A_1 + A_3$ and $B_{123} := B_1 + B_2 + B_3$. Lemmas~\ref{lem:CrudeBoundOnM}-~\ref{lem:nonnegLyapunov}, presented below, are used to prove Lemma~\ref{lem:rec1} (Appendix~\ref{app:proof_lem:rec1}), the key recursion that gave the recursive inequality for $Z_k$, the Lyapunov function. 

\begin{lemma}\label{lem:CrudeBoundOnM}
The following holds:  
    \begin{align*}
        M(x_{k+1} - x^*) &\le M(\xk - x^*) \lrp{1-\alpha_k \eta + \alpha^2_k  \frac{uL_s (1+A_{13})^2}{\lcs^2} } + \alpha^2_k \frac{L_s }{\lcs^2} \lrp{A_{1}\|x^*\|_c + B_{123}}^2 + M_1 ,
    \end{align*}
    where $M_1 := \alpha_k \langle \nabla M(\xk - x^*), F(\xk, S_k) + Z_k - \bar{F}(x_k) \rangle$. 
\end{lemma}
The bound in the above lemma follows from the Assumption~\ref{assmp:lyapunov} on the Lyapunov function $M$. We refer the reader to Appendix~\ref{app:proof_lem:CrudeBoundOnM} for a proof.  We next bound the term $M_1$ introduced in Lemma~\ref{lem:CrudeBoundOnM} above. Using Equation~\eqref{eq:PETerm}, we get
\begin{align*}
    M_1 &= \alpha_k \langle \nabla M(\xk - x^*), V_{x_k}(S_k) - \E{V_{x_k}(S_{k+1}) | S_{k}, x_{k}} \rangle +  \alpha_k \langle \nabla M(\xk - x^*), Z_k \rangle\\
    &=  M_{11} + M_{12} + M_{13}, \numberthis\label{eq:T1intermediate1}
\end{align*} 
where, 
    \begin{align*}
        &M_{11} := \alpha_k \langle \nabla M(\xk - x^*), V_{x_k}(S_k) - \E{V_{x_k}(S_k) | S_{k-1}, x_{k-1}, Z_{k-1}} \rangle, \\
        & M_{12} := \alpha_k \langle \nabla M(\xk - x^*), \E{V_{x_k}(S_k) | S_{k-1}, x_{k-1}, Z_{k-1}} - \E{V_{x_k}(S_{k+1}) | S_{k}, x_{k}} \rangle, \numberthis \label{eq:T12}\\
        & M_{13} := \alpha_k \langle \nabla M(\xk - x^*), Z_k \rangle.
    \end{align*}
\begin{lemma}\label{lem:T11Properties}
    The terms $M_{11}$ and $M_{13}$, defined in Equation~\eqref{eq:T12}, satisfy 
    \begin{align*} 
    &\E{ M_{11} | S_{k-1}, x_{k-1}, Z_{k-1} }  = 0 \qquad \text{and} \qquad \E{M_{13} | S_{k-1}, x_{k-1}, Z_{k-1}} = 0.
    \end{align*}
    Moreover, they are bounded as below:
    \begin{align*} 
        &|M_{11}| \le \frac{2\alpha_k L_s L_1}{\lcs^2}\|\xk- x^*\|^2_c + \frac{2\alpha_k L_s L_2}{\lcs^2}\|\xk-x^*\|_c, \quad \text{and} \quad  |M_{13}| \le \alpha_k \frac{L_s B_2}{\lcs^2} \|x_k -x^*\|_c.
    \end{align*}
    Furthermore, $M_{11}$ and $M_{13}$ are conditionally independent, conditioned on $S_{k-1}, x_{k-1}, Z_{k-1}$. 
\end{lemma}
To see the first part of the lemma, observe that since $x_k$ is a deterministic function of $x_{k-1}$, $S_{k-1}$, and $Z_{k-1}$, $\E{\cdot | S_{k-1}, x_{k-1}, Z_{k-1}} = \E{\cdot | S_{k-1}, x_{k-1}, Z_{k-1}, x_k}$.   Thus, 
\begin{align*}
    \E{ M_{11} | S_{k-1}, x_{k-1}, Z_{k-1} } &= \alpha_k \langle \nabla M(\xk - x^*), 0 \rangle = 0.
\end{align*}
The fact that $\E{M_{13} | S_{k-1}, x_{k-1}, Z_{k-1}} = 0$ follows similarly from conditioning and Assumption~\ref{assmp:boundedaddnoise}. We refer the reader to Appendix~\ref{app:proof_lem:T11Properties} for a complete proof of the Lemma. Later, we bound the exponent of $M_{11}+M_{13}$ using the conditional Hoeffding's lemma, with the above bounds as the sub-Gaussian parameter. 

Finally, let us now bound the term $M_{12}$ from Equation~\eqref{eq:T12}. 
\begin{align*}
    M_{12} &=  \alpha_k \langle \nabla M(x_k - x^*), \E{ V_{x_k}(S_k) | S_{k-1}, x_{k-1}, Z_{k-1} } - \E{V_{x_k}(S_{k+1}) | S_k, x_k} \rangle  \\ 
    &=  \alpha_k d_k -  \alpha_k \langle \nabla M(x_k - x^*), \E{ V_{x_k}(S_{k+1}) | S_k, x_k} \rangle \\
    &= \alpha_k (d_k-d_{k+1}) + \alpha_k \lrp{d_{k+1} - \langle \nabla M(x_k - x^*), \E{ V_{x_k}(S_{k+1}) | S_k, x_k} \rangle}. \numberthis \label{eq:T12intermediate1}
\end{align*}

\begin{lemma}\label{lem:boundonmoddk} $|d_k|$, defined in Equation~\eqref{eq:defdk}, is bounded as below. 
    \begin{align*}
        |d_k| \le \frac{uL_s(L_1+L_2)}{\lcs^2} M(x_k-x^*) + \frac{L_sL_2}{4}.
    \end{align*}
\end{lemma}
We refer the reader to Appendix~\ref{app:proof_lem:boundonmoddk} for a proof. 

\begin{lemma}\label{lem:boundonadjustment1} We have
    \begin{align*} 
        \alpha_k & \lrp{d_{k+1} - \langle \nabla M(x_k - x^*), \E{ V_{x_k}(S_{k+1}) | S_k, x_k } \rangle} \\
        &\le M(x_k - x^*) \frac{\alpha^2_k  u L_s }{\lcs^2 \lrp{\lcs\wedge 1}^2}\lrp{3L_1 (1+A_{13})^2  + C'_1  } + \frac{ \alpha^2_k L_s }{\lcs^2\lrp{\lcs\wedge 1}^2}\lrp{3L_1 (A_{1}\|x^*\|_c + B_{123})^2 + C'_2},
    \end{align*}
    where
    \begin{align*}
        C'_1 &= A_{13}(L_1+L_2) + 2(L_1+L_2) + L_1\lrp{A_1\|x^*\|_c + B_{123}},\\
        C'_2 &= (L_1 + L_2){A_1 \|x^*\|_c + B_{123}} + (1+A_{13})L_2.
    \end{align*}

\end{lemma}
We refer the reader to Appendix~\ref{app:proof_lem:boundonadjustment1} for a proof.

\begin{lemma}\label{lem:boundondiffofdk}
$d_k$, defined in Equation~\eqref{eq:defdk}, satisfies the following recursion: 
    \begin{align*}
        \alpha_k(d_k - d_{k+1}) &\le \alpha_{k-1} d_k \lrp{1 - \tfrac{\eta}{2}\alpha_k} - \alpha_k d_{k+1}  +  \alpha^2_k \lrp{\tfrac{3}{\alpha} + \eta}\lrp{\tfrac{uL_s(L_1+L_2)}{\lcs^2} M(x_k-x^*) + \tfrac{L_sL_2}{4}} .
    \end{align*}
\end{lemma}
See Appendix~\ref{app:proof_lem:boundondiffofdk} for a proof. 

\begin{lemma} \label{lem:BoundOnT12}
The term $M_{12}$, introduced in Equation~\eqref{eq:T12}, is bounded as follows. 
\begin{align*}
    M_{12} \le & \,\,\alpha^2_k M(x_k - x^*) \tfrac{u L_s}{\lcs^2\lrp{\lcs\wedge 1}^2} \lrp{\lrp{\tfrac{3}{\alpha} + \eta}(L_1+L_2) + 3L_1 (1+A_{13})^2 + C'_1}\\
    &  + \alpha^2_k \left( \lrp{\tfrac{3}{\alpha} + \eta} \tfrac{L_sL_2}{4} + \tfrac{ L_s }{\lcs^2 \lrp{\lcs \wedge 1}^2}\lrp{3L_1 (A_{1} \|x^*\|_c + B_{123})^2 + C'_2} \right) \\
        & + \alpha_{k-1} d_k \lrp{1 - \tfrac{\eta}{2}\alpha_k} - \alpha_k d_{k+1}. 
\end{align*}
\end{lemma}
The bound on $M_{12}$ follows using the bounds of lemmas~\ref{lem:boundonadjustment1} and ~\ref{lem:boundondiffofdk} in Equation~\eqref{eq:T12intermediate1}. Next, let
\begin{align*}
    &C_1 := \tfrac{u L_s}{\lcs^2\lrp{\lcs\wedge 1}^2} \lrp{\lrp{\tfrac{3}{\alpha} + \eta}(L_1+L_2) + (3L_1+1) (1+A_{13})^2 + C'_1} + \tfrac{u L_s(1+A_{13})^2}{\lcs^2},\\
    &C_2 :=   \lrp{\tfrac{3}{\alpha} + \eta} \tfrac{L_sL_2}{4} + \tfrac{ L_s }{\lcs^2 \lrp{\lcs \wedge 1}^2}\lrp{(3L_1+1) (A_{1} \|x^*\|_c + B_{123})^2 + C'_2} + \tfrac{L_s}{\lcs^2}\lrp{A_1\|x^*\|_c + B_{123}}^2.
\end{align*}

\begin{lemma}\label{lem:rec1beforeexp} The following bound holds.
    \begin{align*}
        M(x_{k+1} - x^*) + \alpha_k d_{k+1} &\le M(\xk - x^*) \lrp{1-\alpha_k \eta + \alpha^2_k C_1 } \\
        &\qquad+ M_{11} + M_{13} + \alpha_{k-1}d_k \lrp{1-\tfrac{\eta}{2}\alpha_k} + \alpha^2_k C_2.
    \end{align*}
\end{lemma}
The above lemma follows by combining lemmas~\ref{lem:CrudeBoundOnM} and ~\ref{lem:BoundOnT12} with Equation~\eqref{eq:T1intermediate1}, substituting the bounds in Lemma~\ref{lem:CrudeBoundOnM}, and rearranging the terms in the resulting inequality. 

\begin{lemma}\label{lem:nonnegLyapunov} For $h \ge 8$ large enough so that for all $k\ge 1$
\[ 1 - \alpha_{k-1}\tfrac{uL_s(L_1+L_2)}{\lcs^2} \ge 0, \]
the following holds for all $k\ge 1$: 
    \begin{align*}
        M(x_k - x^*) + \alpha_{k-1} d_k + \tfrac{\alpha_{k-1}L_sL_2}{4} \ge 0.
    \end{align*}
\end{lemma}
\begin{proof}
    Consider the following inequalities.
    \begin{align*}
        &M(x_k - x^*) + \alpha_{k-1} d_k \\
        &\qquad \ge M(x_k - x^*) - \alpha_{k-1} \abs{d_k}\\
        &\qquad \ge M(x_k - x^*)  - \alpha_{k-1} \lrp{\frac{uL_s(L_1+L_2)}{\lcs^2} M(x_k-x^*) + \frac{L_sL_2}{4}} \tag{Lemma~\ref{lem:boundonmoddk}} \\
        &\qquad = M(x_k - x^*)\lrp{1 - \alpha_{k-1}\frac{uL_s(L_1+L_2)}{\lcs^2}} -\alpha_{k-1}\frac{L_sL_2}{4}\\
        &\qquad \ge -\alpha_{k-1}\frac{L_sL_2}{4},
    \end{align*}
    proving the desired inequality. 
\end{proof}

Below, we present additional technical results that are used in the proofs of the above lemmas. 

\begin{lemma}\label{lem:gradMsbound} For all $x\in\R^d$, 
    \[ \|\nabla M(x - x^*) \|^*_s \le L_s\|x - x^*\|_s \le \frac{L_s}{\lcs} \|x-x^*\|_c. \]
\end{lemma}
\begin{proof}
    Since $M(\cdot)$ is $L_s$-smooth with respect to $s$-norm (Equation~\eqref{eq:Smoothness}), for any $x\in \R^d$ and $y\in\R^d$ we equivalently have \citep[Chapter 5]{beck2017first}:
    \[ \|\nabla M(x - x^*) - \nabla M(y - x^*) \|^*_s  \le L_s\|x-y\|_s.\]
In particular, this implies (using $y=x^*$ and Equation~\eqref{eq:0gradatmin}) 
\[ \|\nabla M(x - x^*) \|^*_s \le L_s\|x - x^*\|_s \le \frac{L_s}{\lcs} \|x -x^*\|_c, \]
proving the result.
\end{proof}

\begin{lemma}\label{lem:noiseterms}
    For $k\in\N$, $\xk \in \R^d$, $A_{13} = A_1 + A_3$ and $B_{123} = B_1 + B_2 + B_3$, 
    \begin{align*} 
        \|F(\xk, S_k) + Z_k  - \xk\|_c &\le \lrp{1 + A_{13}}\|\xk - x^*\|_c + A_1 \|x^*\|_c  + B_{123},\\
        \|F(\xk, S_k) + Z_k  - \xk\|_s &\le \frac{1+ A_{13}}{\lcs} \|\xk - x^*\|_c + \frac{1}{\lcs} \lrp{A_{1} \|x^*\|_c  + B_{123}}, \\
        \|F(\xk, S_k) + Z_k - \xk\|^2_s 
        &\le \frac{2(1+A_{13})^2}{\lcs^2} \|\xk - x^*\|^2_c + \frac{2}{\lcs^2} \lrp{A_{1} \|x^*\|_c  + B_{123}}^2. 
    \end{align*}
    \begin{proof} For the first inequality, 
        \begin{align*}
            \|F(\xk, S_k) + Z_k - \xk\|_s
            &\le  \|F(\xk, S_k) - \bar{F}(\xk) \|_c  \\
            &\qquad + \|\bar{F}(\xk) - x^*\|_c + \|x^* - \xk\|_c + \|Z_k\|_c  \tag{Triangle ineq.}\\
            &\le A_1 \|\xk - x^*\|_c + A_1\|x^*\|_c + B_1 \\
            &\qquad + A_3 \|\xk - x^*\|_c + B_3 + \|\xk -x^*\|_c + B_2 \\
            &= \lrp{1 + A_1 + A_3} \|\xk - x^*\|_c + {A_1} \|x^*\|_c + B_1 + B_2 + B_3.
        \end{align*}
        The second inequality follows from the first, and using Equation~\eqref{eq:normequiv}. Finally, the last one follows from squaring the previous one and using $(a+b)^2 \le 2 a^2 + 2 b^2$.
    \end{proof}

\end{lemma}

\subsection{Proof Details for Supporting Results in Appendix~\ref{app:proofdetails_onestepbootstrap}}
In this appendix, we provide proof details for all the supporting results presented earlier in Appendix~\ref{app:proofdetails_onestepbootstrap}. 
\subsubsection{Proof of Lemma~\ref{lem:CrudeBoundOnM}}\label{app:proof_lem:CrudeBoundOnM}
\begin{proof}
Consider the following inequalities. 
\begin{align*}
    & M(\xkp - x^*) \\
    &\le M(\xk - x^*) + \langle \nabla M(\xk - x^*), \xkp - \xk \rangle + \frac{L_s}{2} \|\xkp - \xk\|^2_s \tag{Equation~\eqref{eq:Smoothness}}\\
    &= M(\xk - x^*) + \alpha_k \langle \nabla M(\xk - x^*), F(\xk, S_k) + Z_k - \xk \rangle + \frac{L_s \alpha^2_k}{2} \|F(\xk,S_k) + Z_k - \xk\|^2_s\\
    &\le M(\xk - x^*) + \alpha_k \langle \nabla M(\xk - x^*), F(\xk, S_k) + Z_k - \xk \rangle \\
        &\qquad +\alpha^2_k \lrp{\frac{L_s (1+A_{13})^2}{\lcs^2} \|\xk - x^*\|^2_c + \frac{L_s}{\lcs^2} \lrp{A_{1} \|x^*\|_c  + B_{123}}^2} \tag{Lemma~\ref{lem:noiseterms}}\\
    &\le M(\xk - x^*) (1-\alpha_k \eta) + \alpha_k \langle \nabla M(\xk - x^*), F(\xk, S_k) + Z_k - \bar{F}(x_k) \rangle \\
        &\qquad +\alpha^2_k \lrp{\frac{u L_s (1+A_{13})^2}{\lcs^2} M(\xk - x^*) + \frac{L_s}{\lcs^2} \lrp{A_{1} \|x^*\|_c  + B_{123}}^2} \tag{Equations~\eqref{eq:negdrift},~\eqref{eq:normMcM}} \\
    &= M(\xk - x^*) \lrp{1-\alpha_k \eta + \alpha^2_k  \frac{uL_s (1+A_{13})^2}{\lcs^2} } + \alpha_k \langle \nabla M(\xk - x^*), F(\xk, S_k) + Z_k - \bar{F}(x_k) \rangle \\
        &\qquad + \alpha^2_k \frac{L_s }{\lcs^2} \lrp{A_{1}\|x^*\|_c + B_{123}}^2,
\end{align*}
proving the desired bound.
\end{proof}

\subsubsection{Proof of Lemma~\ref{lem:T11Properties}}\label{app:proof_lem:T11Properties}
\begin{proof}
Since $x_k$ is a deterministic function of $x_{k-1}$, $S_{k-1}$, and $Z_{k-1}$, we have 
$$\E{\cdot | S_{k-1}, x_{k-1}, Z_{k-1}} = \E{\cdot | S_{k-1}, x_{k-1}, Z_{k-1}, x_k}.$$ Thus, 
\begin{align*}
    &\E{ M_{11} | S_{k-1}, x_{k-1}, Z_{k-1} } = \alpha_k \langle \nabla M(\xk - x^*), 0 \rangle = 0,\\
    &\E{ M_{13} | S_{k-1}, x_{k-1}, Z_{k-1} } = \alpha_k \langle \nabla M(\xk - x^*), \E{Z_k | S_{k-1}, x_{k-1}, Z_{k-1}}\rangle \overset{(a)}{=} 0,
\end{align*}
where $(a)$ follows from Assumption~\ref{assmp:boundedaddnoise} and the observation that  $Z_k$ is independent of $S_{k-1}, x_{k-1}, Z_{k-1}$.

Next, we prove the conditional independence of $M_{11}$ and $M_{13}$. To this end, observe that conditioned on $(x_{k-1}, S_{k-1}, Z_{k-1})$, $x_k$ is deterministic. Hence, conditioned on these, the only randomness in $M_{11}$ is due to $S_k$ and that in $M_{13}$ is due to $Z_k$, which are themselves independent, proving the conditional independence of $M_{11}$ and $M_{13}$.

To arrive at the bound for $|M_{11}|$, consider the following inequalities: 
\begin{align*}
    |M_{11}| 
    &= \alpha_k \abs{\langle \nabla M(\xk - x^*), V_{x_k}(S_k) - \E{V_{x_k}(S_k) | S_{k-1}, x_{k-1}, Z_{k-1}} \rangle}\\
    &\le \alpha_k  \|\nabla M(\xk - x^*)\|^*_s \|V_{x_k}(S_k) - \E{V_{x_k}(S_k) | S_{k-1}, x_{k-1}, Z_{k-1}}\|_s  \tag{Cauchy-Schwarz}\\
    &\le  \frac{\alpha_k L_s}{\lcs^2} \|\xk-x^*\|_c \|V_{\xk}(S_k) -  \E{ V_{\xk}(S_k) | S_{k-1}, x_{k-1}, Z_{k-1}}\|_c \tag{Lemma~\ref{lem:gradMsbound} \& Eq.~\eqref{eq:normequiv}}\\
    &\le \frac{2\alpha_k L_s}{\lcs^2} \|\xk - x^*\|_c\|V_{\xk}(S_k)\|_c\\
    &\le \frac{2\alpha_k L_s}{\lcs^2}\|\xk- x^*\|_c \lrp{L_1\|x_k - x^*\|_c + L_2}\tag{Eq.~\eqref{eq:Lipschitzness}}\\
    &= \frac{2\alpha_k L_s L_1}{\lcs^2}\|\xk- x^*\|^2_c + \frac{2\alpha_k L_s L_2}{\lcs^2}\|\xk-x^*\|_c,
 \end{align*}
proving the bound. 

Finally, we prove the bound for $|M_{13}|$ below: 
\begin{align*}
    |M_{13}| = \alpha_k | \langle \nabla M(x_k - x^*), Z_k \rangle &\le \alpha_k \| \nabla M(x_k - x^*) \|^*_s \|Z_k\|_s \tag{Cauchy-Schwarz}\\
    &\le  \alpha_k \frac{L_s}{\lcs^2} \|x_k - x^*\|_c \|Z_k\|_c  \tag{Using Lemma~\ref{lem:gradMsbound} and Eq.~\eqref{eq:normequiv}}\\
    &\le \alpha_k \frac{L_s B_2}{\lcs^2} \|x_k -x^*\|_c \tag{Assumption~\ref{assmp:boundedaddnoise}},
\end{align*}
proving the lemma. 
\end{proof}

\subsubsection{Proof of Lemma~\ref{lem:boundonmoddk}}\label{app:proof_lem:boundonmoddk}

\begin{proof} Consider the following inequalities:
    \begin{align*}
        |d_k| 
        &= \abs{\langle \nabla M(x_k - x^*), \E{V_{x_k}(S_k) | S_{k-1}, x_{k-1}, Z_{k-1}} \rangle}\\
        &\le \| \nabla M(x_k - x^*)\|^*_s \|\E{V_{x_k}(S_k)|S_{k-1}, x_{k-1}, Z_{k-1}}\|_s \tag{Cauchy-Schwarz}\\
        &\le \frac{L_s}{\lcs^2} \|\xk-x^*\|_c \E{\|V_{x_k}(S_k)\|_c | S_{k-1}, x_{k-1}, Z_{k-1} } \tag{Lemma~\ref{lem:gradMsbound} and Jensen's inequality}\\
        &\le \frac{L_s}{\lcs^2} \|x_k-x^*\|_c \E{\|V_{x_k}(S_k)-V_{x^*}(S_k)\|_c + \| V_{x^*}(S_k) \|_c| S_{k-1}, x_{k-1}, Z_{k-1} } \tag{Triangle inequality} \\
        &\le \frac{L_s}{\lcs^2}  \|x_k-x^*\|_c \E{(L_1\|x_k - x^*\|_c + L_2) | S_{k-1}, x_{k-1}, Z_{k-1} } \tag{Eq.~\eqref{eq:Lipschitzness}} \\
        &\overset{(a)}{=} \frac{L_s L_1}{\lcs^2}  \|x_k-x^*\|^2_c + \frac{L_s L_2}{\lcs} \|x_k - x^*\|_c \\
        &\le \frac{L_s(L_1+L_2)}{\lcs^2} \|x_k-x^*\|^2_c + \frac{L_sL_2}{4} \tag{$\because   2ab \le a^2 + b^2$}\\
        &\le \frac{uL_s(L_1+L_2)}{\lcs^2} M(x_k-x^*) + \frac{L_sL_2}{4},
    \end{align*}
where $(a)$ follows since $x_k$ is a deterministic function of $x_{k-1}$, $S_{k-1}$, and $Z_{k-1}$.
\end{proof}

\subsubsection{Proof of Lemma~\ref{lem:boundonadjustment1}}\label{app:proof_lem:boundonadjustment1}
\begin{proof} Recall, 
\[ d_k = \langle \nabla M(x_k - x^*), \E{ V_{x_k}(S_k) | S_{k-1}, x_{k-1}, Z_{k-1} }.\] 
Also, observe that 
\[ \E{V_{x_k}(S_{k+1}) | S_k, x_k }  = \E{ V_{x_k}(S_{k+1}) | S_k, x_k, Z_k }. \numberthis\label{eq:condZkForFree} \]
This follows since conditioned on $(S_k, x_k)$, $V_{x_k}(S_{k+1})$ is independent of $Z_k$ ($Z_k$ and the evolution of the underlying Markov chain are independent). 
    \begin{align*}
        d_{k+1} & - \langle \nabla M(x_k - x^*), \E{ V_{x_k}(S_{k+1}) | S_k, x_k } \rangle\\
        &= d_{k+1} - \langle \nabla M(x_k - x^*), \E{ V_{x_k}(S_{k+1}) | S_k, x_k, Z_k } \rangle \tag{Using Eq.~\eqref{eq:condZkForFree}}\\
        &= \langle \nabla M(x_{k+1} - x^*), \E{V_{x_{k+1}}(S_{k+1}) | S_k, x_k, Z_k} \rangle - \langle \nabla M(x_k - x^*), \E{ V_{x_k}(S_{k+1}) | S_k, x_k, Z_k } \rangle\\
        &= \langle \nabla M(x_{k+1} - x^*), \E{V_{x_{k+1}}(S_{k+1}) | S_k, x_k, Z_k} - \E{ V_{x_k}(S_{k+1}) | S_k, x_k, Z_k }\rangle \\
        &\qquad + \langle \nabla M(x_{k+1} - x^*) - \nabla M(x_k - x^*), \E{ V_{x_k}(S_{k+1}) | S_k, x_k, Z_k } \rangle.\numberthis\label{eq:temp1_T12}
    \end{align*}

    In the rest of this proof, we denote by $\mathcal F_{k+1}$ the sigma algebra with respect to which $S_k, x_k, Z_k$ are measurable, i.e., $\mathcal F_{k+1} = \sigma(S_k, x_k, Z_k)$. Note that we use the index $k+1$ in $\mathcal F_{k+1}$ since $x_{k+1}$ is a deterministic function of $S_k, x_k, Z_k$. So $x_{k+1} \in \mathcal F_{k+1}$.
    
    We now bound the two terms in Equation~\eqref{eq:temp1_T12} separately. The first term is bounded as below:
    \begin{align*}
        &\langle \nabla M(x_{k+1} - x^*), \E{V_{x_{k+1}}(S_{k+1}) | \mathcal F_{k+1}} - \E{ V_{x_k}(S_{k+1}) | \mathcal F_{k+1}}\rangle\\
        &\le  \|\nabla M(x_{k+1} - x^*)\|^*_s   \|\E{V_{x_{k+1}}(S_{k+1}) | \mathcal F_{k+1}} - \E{ V_{x_k}(S_{k+1}) | \mathcal F_{k+1} }\|_s \tag{Cauchy-Schwarz} \\
        &\le  \frac{L_s}{\lcs^2} \|x_{k+1}-x^*\|_c  \|\E{V_{x_{k+1}}(S_{k+1}) | \mathcal F_{k+1}} - \E{ V_{x_k}(S_{k+1}) | \mathcal F_{k+1}}\|_c \tag{Lemma~\ref{lem:gradMsbound}}\\
        &\le \frac{L_s}{\lcs^2} \ \|x_{k+1} - x^*\|_c  \E{\| V_{x_{k+1}}(S_{k+1}) - V_{x_k}(S_{k+1})\|_c | \mathcal F_{k+1}} \tag{Conditional Jensen's inequality}\\
        &\le  \frac{L_sL_1}{\lcs^2}\|x_{k+1} - x^*\|_c  \E{\| x_{k+1} - x_k \|_c  | \mathcal F_{k+1}} \tag{From Eq.~\eqref{eq:Lipschitzness}}\\
        &\overset{(a)}{\le} \frac{L_sL_1}{\lcs^2} \|x_{k+1} - x_k\|^2_c + \frac{L_sL_1}{\lcs^2}\|x_{k+1} - x_k\|_c \|x_k - x^*\|_c \tag{Triangle inequality}\\
        &= \frac{L_s L_1}{\lcs^2}  \alpha^2_k \|F(x_k, S_k) + Z_k - x_k\|^2_c + \frac{L_sL_1}{\lcs^2} \alpha_k \|F(x_k, S_k) + Z_k - x_k \|_c \|x_k - x^*\|_c\\
        &\le \frac{L_s L_1}{\lcs^2} \alpha^2_k \lrp{\frac{2(1+A_{13})^2}{\lcs^2} \|\xk - x^*\|^2_c + \frac{2}{\lcs^2} \lrp{A_{1} \|x^*\|_c  + B_{123}}^2} \tag{Lemma~\ref{lem:noiseterms}}\\
        &\qquad + \frac{L_s L_1}{\lcs^2} \alpha_k \lrp{\frac{1+A_{13}}{\lcs} \|\xk - x^*\|_c + \frac{1}{\lcs} \lrp{A_{3} \|x^*\|_c  + B_{123}}} \|x_k - x^*\|_c \\
        &= \|\xk - x^*\|^2_c \lrp{\frac{2 L_s L_1 (1+A_{13})^2}{\lcs^4} \alpha^2_k + \frac{L_s L_1 (1+A_{13})}{\lcs^3} \alpha_k} + \alpha^2_k \frac{2 L_s L_1}{\lcs^4}\lrp{A_{1}\|x^*\|_c + B_{123}}^2\\ 
            &\qquad + \|x_k - x^*\| \alpha_k\frac{L_s L_1}{\lcs^3} \lrp{A_{1}\|x^*\|_c + B_{123}} \\
        &\le \|\xk - x^*\|^2_c \lrp{\frac{2 L_s L_1 (1+A_{13})^2}{\lcs^4} \alpha^2_k + \frac{L_s L_1 (2+A_{13})}{\lcs^3} \alpha_k} + \alpha^2_k \frac{2 L_s L_1}{\lcs^4}\lrp{A_{1}\|x^*\|_c + B_{123}}^2\\ 
            &\qquad + \frac{\alpha_k L_s L_1}{4\lcs^3}  \lrp{A_{1}\|x^*\|_c + B_{123}}^2 \tag{$2ab \le a^2 + b^2$}\\
        &\overset{(b)}{\le} M(\xk - x^*) \alpha_k \frac{L_sL_1 u}{\lcs^3}\lrp{\frac{2 (1+A_{13})^2}{\lcs} \alpha_k + A_{13} + 2} +  \frac{3L_s L_1 \alpha_k}{\lcs^3\lrp{\lcs \wedge 1}}\lrp{A_{1}\|x^*\|_c + B_{123}}^2\!,\numberthis\label{eq:boundonfirstterm}
    \end{align*}
    where $(a)$ follows since $x_{k+1}$ is a deterministic function of $x_k, S_k, Z_k$, and in $(b)$ we used that $ \|\cdot\|^2_c \le u M(\cdot)$, along with $\alpha_k \le 1$.

    Let us now bound the second term in Equation~\eqref{eq:temp1_T12}. Since $M(\cdot)$ is $L_s$-smooth with respect to $s$-norm (Equation~\eqref{eq:Smoothness}), we equivalently have \citep[Chapter 5]{beck2017first}
    \[ \|\nabla M(x_{k+1} - x^*) - \nabla M(x_k - x^*) \|^*_s  \le L_s\|x_{k+1}-x_k\|_s.\]

    Using this, 
    \begin{align*}
        \langle \nabla &M(x_{k+1} - x^*) - \nabla M(x_k - x^*), \E{ V_{x_k}(S_{k+1}) | \mathcal F_{k+1}} \rangle \\
        &\le \| \nabla M(x_{k+1} - x^*) - \nabla M(x_k - x^*)\|^*_{s} \| \E{ V_{x_k}(S_{k+1}) | \mathcal F_{k+1}} \|_s \tag{Cauchy-schwarz} \\
        &\le  \alpha_k L_s\|F(x_k, S_k) + Z_k - x_k\|_s ~\E{\|V_{x_k}(S_{k+1})\|_s  | \mathcal F_{k+1}} \tag{Jensen's, Smoothness}\\
        &\le \frac{\alpha_k L_s}{ \lcs^2 } \|F(x_k, S_k) + Z_k - x_k\|_c ~\E{\|V_{x_k}(S_{k+1})\|_c | \mathcal F_{k+1}} \tag{Equation~\eqref{eq:normequiv}} \\
        &\le \frac{\alpha_k L_s}{\lcs^2} \lrp{\lrp{1+A_{13}}\|\xk - x^*\|_c + A_{1} \|x^*\|_c  + B_{123}} \lrp{L_1 \|\xk - x^*\|_c + L_2}\tag{Lemma~\ref{lem:noiseterms} \& Equation~\eqref{eq:Lipschitzness}}\\
        &= \alpha_k\frac{ L_s L_1 (1+A_{13})}{\lcs^2} \|x_k - x^*\|^2_c + \alpha_k\frac{ L_s L_2}{\lcs^2}(A_{1}\|x^*\|_c + B_{123})  \\ 
        &\qquad + \alpha_k\frac{ L_s}{\lcs^2}((1+A_{13})L_2 + L_1 (A_{1}\|x^*\|_c + B_{123})) \|x_k - x^*\|_c\\
        &\le \alpha_k \frac{L_s L_1 (1+A_{13})}{\lcs^2} \|x_k - x^*\|^2_c + \alpha_k \frac{L_s}{2\lcs^2}((1+A_{13})L_2 + L_1 (A_{1}\|x^*\|_c + B_{123})) \|x_k - x^*\|^2_c \\
        &\qquad + \alpha_k \frac{L_s}{\lcs^2}\lrp{(L_1 + L_2)A_1 \|x^*\|_c + (L_1+L_2) B_{123} + (1+A_{13})L_2 } \tag{$2ab \le a^2 + b^2$}\\
        &\le \alpha_k  M(x_k - x^*) \frac{u L_s}{\lcs^2}\lrp{(1+ A_{13})\lrp{L_1 + L_2 }  + L_1 (A_1\|x^*\|_c + B_{123})} \\
        &\qquad + \alpha_k \frac{L_s}{\lcs^2}\lrp{(L_1 + L_2)A_1 \|x^*\|_c + (L_1+L_2) B_{123} + (1+A_{13})L_2}\numberthis\label{eq:boundonsecterm}.
    \end{align*}
    One obtains the desired bound by  adding the bounds in equations~\eqref{eq:boundonfirstterm} and~\eqref{eq:boundonsecterm} and re-arranging the terms.
\end{proof}

\subsubsection{Proof of Lemma~\ref{lem:boundondiffofdk}}\label{app:proof_lem:boundondiffofdk}
\begin{proof} For $\tilde{\gamma}_c := \eta/2$,
    \begin{align*}
        \alpha_k(d_k - d_{k+1}) &= \alpha_{k-1} d_k \lrp{1 - \alpha_k\tilde{\gamma}_c} - \alpha_k d_{k+1} + 
        d_k \lrp{\alpha_k - \alpha_{k-1}\lrp{1 - \alpha_k \tilde{\gamma}_c}}\\
        &= \alpha_{k-1} d_k \lrp{1 - \alpha_k\tilde{\gamma}_c} - \alpha_k d_{k+1} + 
        d_k \lrp{\alpha_k - \alpha_{k-1} +  \alpha_{k-1}\alpha_k \tilde{\gamma}_c}\\
        &\le \alpha_{k-1} d_k \lrp{1 - \alpha_k\tilde{\gamma}_c} - \alpha_k d_{k+1} + 
        \abs{d_k}\abs{\alpha_k - \alpha_{k-1}} + \abs{d_k}  \alpha_{k-1}\alpha_k \tilde{\gamma}_c\\
        &\le \alpha_{k-1} d_k \lrp{1 - \alpha_k\tilde{\gamma}_c} - \alpha_k d_{k+1} + 
        3\abs{d_k}\frac{\alpha^2_k}{\alpha} + 2\abs{d_k}  \alpha^2_k \tilde{\gamma}_c \tag{lemmas~\ref{lem:diffofalpha},~\ref{lem:prodofalpha}}\\
        &= \alpha_{k-1} d_k \lrp{1 - \alpha_k\tilde{\gamma}_c} - \alpha_k d_{k+1} + \alpha^2_k \abs{d_k}\lrp{\frac{3}{\alpha} + 2 \tilde{\gamma}_c}\\
        &\le \alpha_{k-1} d_k \lrp{1 - \alpha_k\tilde{\gamma}_c} - \alpha_k d_{k+1} \\
        &\qquad +  \alpha^2_k \lrp{\frac{3}{\alpha} + 2 \tilde{\gamma}_c}\lrp{\frac{uL_s(L_1+L_2)}{\lcs^2} M(x_k-x^*) + \frac{L_sL_2}{4}} \tag{Lemma~\ref{lem:boundonmoddk}}\\
        &= \alpha_{k-1} d_k \lrp{1 - \frac{\eta}{2}\alpha_k} - \alpha_k d_{k+1} \\
        &\qquad +  \alpha^2_k \lrp{\frac{3}{\alpha} + \eta}\lrp{\frac{uL_s(L_1+L_2)}{\lcs^2} M(x_k-x^*) + \frac{L_sL_2}{4}}. 
    \end{align*}
\end{proof}

\subsubsection{Proof of Lemma~\ref{lem:rec1}}\label{app:proof_lem:rec1}
\begin{proof} 
    Notice that $x_k$ is $\mathcal F_k = \sigma(x_{k-1},S_{k-1}, Z_{k-1})$ measurable since $x_k$ is a deterministic function of $x_{k-1}, S_{k-1}, $ and $Z_{k-1}$.
    
    In this proof, we define $\tilde{\gamma}_c := \eta/2$. Next, $\lambda_{k+1} \ge 0$, multiplying both sides of inequality in Lemma~\ref{lem:rec1beforeexp} by $\lambda_{k+1}$, exponentiating, and multiplying by, we get
    \begin{align*}
        e^{\lambda_{k+1}\lrp{M(x_{k+1} - x^*) + \alpha_k d_{k+1}}} 
        &\le e^{\lambda_{k+1} M(\xk - x^*) \lrp{1 - 2\alpha_k\tilde{\gamma}_c + \alpha^2_k C_1 }} \\
            &\qquad e^{\lambda_{k+1} (M_{11}+M_{13})} e^{\lambda_{k+1} \alpha_{k-1}d_k (1-\alpha_k \tilde{\gamma}_c) + \lambda_{k+1}\alpha^2_k C_2}.
    \end{align*}
    Now, observe that conditioned on $\mathcal F_k = \sigma(x_{k-1}, S_{k-1}, Z_{k-1})$, all the terms in the right hand side, except $e^{\lambda_{k+1}(M_{11}+M_{13})}$, are measurable. Thus, taking the conditional expectation in the above inequality, we get
    \begin{align*}
        &\E{e^{\lambda_{k+1}\lrp{M(x_{k+1} - x^*) + \alpha_k d_{k+1}}} \big| \mathcal F_k } \\
        &\,\,\,\le e^{\lambda_{k+1}M(\xk - x^*) \lrp{1 - 2\alpha_k\tilde{\gamma}_c + \alpha^2_k C_1 }} \E{e^{\lambda_{k+1} (M_{11}+M_{13})} \big| \mathcal F_k} e^{\lambda_{k+1}\alpha_{k-1}d_k (1-\alpha_k \tilde{\gamma}_c) + \lambda_{k+1}\alpha^2_k C_2} \\
        &\,\,\, \le e^{\lambda_{k+1}M(\xk - x^*) \lrp{1 - 2\alpha_k\tilde{\gamma}_c + \alpha^2_k C_1 }} \E{e^{\lambda_{k+1} M_{11}} \big| \mathcal F_k} \E{e^{\lambda_{k+1} M_{13}} \big| \mathcal F_k} e^{\lambda_{k+1}\alpha_{k-1}d_k (1-\alpha_k \tilde{\gamma}_c) + \lambda_{k+1}\alpha^2_k C_2},\numberthis\label{eq:boundcondexp}
    \end{align*}
    where the last inequality follows from the conditional independence of $M_{11}$ and $M_{13}$ (Lemma~\ref{lem:T11Properties}). Next, using Lemma~\ref{lem:T11Properties} and Hoeffding's inequality, we get
    \[ \E{ e^{\lambda_{k+1} M_{13}} \big| \mathcal F_k} \le e^{\lambda^2_{k+1} \frac{\alpha^2_k L^2_s B^2_2}{2 \lcs^4} \|x_k - x^*\|^2_c  } \le  e^{\lambda^2_{k+1} \frac{\alpha^2_k L^2_s B^2_2 u}{2 \lcs^4} M(x_k - x^*)  }. \]
    Again, from Lemma~\ref{lem:T11Properties}, we see that the conditional mean of $M_{11}$, conditioned on $\mathcal F_k$, is $0$. Moreover, since $\|x_k-x^*\|_c^2 \leq T_k$ almost surely, $M_{11}$ is bounded as below:
    \begin{align*} |M_{11}| &\le \frac{2\alpha_k L_s L_1}{\lcs^2}\|\xk- x^*\|^2_c + \frac{2\alpha_k L_s L_2}{\lcs^2}\|\xk-x^*\|_c \\
    &\le \frac{2\alpha_k L_s L_1}{\lcs^2}\|\xk- x^*\|_c T^\frac{1}{2}_k + \frac{2\alpha_k L_s L_2}{\lcs^2}\|\xk-x^*\|_c \\
    &\le \frac{2\alpha_k L_s}{\lcs^2}\|\xk- x^*\|_c  \lrp{L_1 T^\frac{1}{2}_k + L_2}. 
    \end{align*}
    Hence, using Hoeffding's lemma, we have
    \begin{align*}
    \E{e^{\lambda_{k+1} M_{11}} \big| \mathcal F_k} 
        &\le e^{\lambda^2_{k+1}\alpha^2_k \frac{2 L^2_s}{\lcs^4}\|x_k - x^*\|^2_c\lrp{L_1 T^\frac{1}{2}_k + L_2}^2}\\
        &\le e^{ \lambda^2_{k+1} \alpha^2_k 
        \frac{2L^2_s u}{\lcs^4} M(x_k - x^*) \lrp{L_1 T^\frac{1}{2}_k + L_2}^2 }\\
        &\le e^{ \lambda^2_{k+1} \alpha^2_k \frac{4 L^2_s u }{\lcs^4}   M(x_k - x^*) \lrp{L^2_1 T_k + L^2_2} },
    \end{align*}
    where the last inequality follows from $(a+b)^2 \le 2a^2 + 2b^2$, for all $a,b\in \R$.  Using the above bounds on conditional expectations of $M_{11}$ and $M_{13}$  in Equation~\eqref{eq:boundcondexp}, we have    
    \begin{align*}
        &\E{ e^{\lambda_{k+1}\lrp{M(x_{k+1} - x^*) + \alpha_k d_{k+1}}} \big| \mathcal F_k } \\
        &\qquad\le e^{\lambda_{k+1} M(\xk - x^*) \lrp{1 - 2\alpha_k\tilde{\gamma}_c + \alpha^2_k C_1 }} e^{\lambda^2_{k+1} \frac{\alpha^2_k L^2_s B^2_2 u}{2\lcs^4} M(\xk - x^*) } \\
        &\qquad\qquad e^{\lambda^2_{k+1} \alpha^2_k \frac{4 L^2_s u }{\lcs^4}   M(x_k - x^*) \lrp{L^2_1 T_k+ L^2_2}  + \lambda_{k+1}\alpha_{k-1}d_k (1-\alpha_k \tilde{\gamma}_c) + \lambda_{k+1}\alpha^2_k C_2}\\
        &\qquad \le  e^{\lambda_{k+1} M(\xk - x^*) \lrp{1 - 2\alpha_k\tilde{\gamma}_c + \alpha^2_k \lrp{  C_1 + \lambda_{k+1}  \frac{4 L^2_s u}{\lcs^4}\lrp{L^2_1 T_k + L^2_2 + B^2_2 } }}} \\
        &\qquad\qquad e^{\lambda_{k+1}\alpha_{k-1}d_k (1-\alpha_k \tilde{\gamma}_c) + \lambda_{k+1}\alpha^2_k C_2}\\
        &\qquad =  e^{ \lambda_{k+1}\lrp{M(\xk - x^*) M_3 + \alpha_{k-1}d_k (1-\alpha_k \tilde{\gamma}_c) }} e^{\lambda_{k+1} \alpha^2_k C_2}, \numberthis \label{eq:condbound1}
        \end{align*}
        where, 
        \[ M_3 := 1 - 2\alpha_k\tilde{\gamma}_c + \alpha^2_k \lrp{  C_1 + \lambda_{k+1}  \frac{ 4 L^2_s u}{\lcs^4}\lrp{L^2_1 T_k + L^2_2 + B^2_2 }  }. \]
    Now, from the definition of $\lambda_k$, we have
    \begin{align*}
        \alpha^2_k \lambda_{k+1}\frac{4 L^2_s u}{\lcs^4}\lrp{L^2_1 T_k + L^2_2 + B^2_2}
        &=    \frac{\theta 4 L^2_s u \alpha^2_k}{\alpha_{k+1}\lcs^4 T_{k+1}}\lrp{L^2_1 T_k + L^2_2 + B^2_2} \\
        &\le \frac{\theta 4 L^2_s u \alpha^2_k}{\alpha_{k+1}\lcs^4 T_{k+1}}\lrp{L^2_1 T_{k+1} + L^2_2 + B^2_2} \tag{$\because T_k\le T_{k+1}$}\\
        &\le \frac{\theta 4 L^2_s u \alpha^2_k}{\alpha_{k+1}\lcs^4 T_0}\lrp{L^2_1 T_0 + L^2_2 + B^2_2}\\
        &\le \frac{\theta 4 L^2_s u \alpha^2_k}{\alpha_{k+1}\lcs^4 \|x_0 - x^*\|^2_c}\lrp{L^2_1 \|x_0 - x^*\|^2_c + L^2_2 + B^2_2} \tag{$\because T_0 \ge \|x_0 - x^*\|^2_c$} \\
        &= \frac{\eta \alpha^2_k}{8 \alpha_{k+1}},
    \end{align*}
    where the last equality follows from choosing 
    \[ \theta =  \frac{\eta \lcs^4 \|x_0 - x^*\|^2_c}{32 L^2_s u  \lrp{L^2_1\|x_0 - x^*\|^2_c + L^2_2 + B^2_2} }. \]
    Next, since for $k \ge 0$ and $h \ge 8$, from Lemma~\ref{lem:ratioofalpha}, we have 
    \[ \frac{\alpha_{k}}{\alpha_{k+1}} \le 2, \]
    and hence, we further have
    \[ \alpha^2_k \lambda_{k+1}  \frac{4 L^2_s u}{\lcs^4}\lrp{L^2_1 T_k + L^2_2 + B^2_2} \le \frac{\eta \alpha_k}{4}, \]
    which implies 
    \[ M_3 \le  1 - \alpha_k \eta + \alpha^2_k C_1 + \frac{\eta\alpha_k}{4}. \]
    Now, choosing $h \ge 8$ large enough so that \[ \alpha_0  \le \frac{\eta}{4 C_1}, \]
    ensures that 
    \[ M_3 \le 1 - \alpha_k \eta + \frac{\alpha_k \eta}{4} + \frac{\alpha_k \eta}{4} = 1 - \frac{\alpha_k \eta}{2} = 1 - \alpha_k \tilde{\gamma}_c. \]

    Substituting this back in Equation~\eqref{eq:condbound1}, since $M(\cdot) \ge 0$, we have
    \begin{align*}
        &\E{ e^{\lambda_{k+1}\lrp{M(x_{k+1} - x^*) + \alpha_k d_{k+1}}} \big| \mathcal F_k }\\
        &\qquad \le e^{ \lambda_{k+1}(1-\alpha_k \tilde{\gamma}_c)\lrp{M(\xk - x^*) + \alpha_{k-1}d_k  }} e^{\lambda_{k+1} \alpha^2_k C_2}\\
        &\qquad \le e^{ \lambda_{k+1}(1-\alpha_k \tilde{\gamma}_c)\lrp{M(\xk - x^*) + \alpha_{k-1}d_k  }}  e^{2 \lambda_k \alpha^2_k C_2}, \numberthis\label{eq:expbound2}
    \end{align*}
    where the last inequality follows since $\alpha^2_k C_2 \ge 0$, and 
    \begin{equation*}
        \lambda_{k+1} = \lambda_k\frac{\lambda_{k+1}}{\lambda_k} = \lambda_k\frac{\alpha_k T_k}{\alpha_{k+1}T_{k+1}} \le \lambda_k\frac{\alpha_k}{\alpha_{k+1}} \le 2\lambda_k \tag{$\because   T_k \le T_{k+1}$ and Lemma~\ref{lem:ratioofalpha}}. 
    \end{equation*}
    Now, multiplying both sides of Equation~\eqref{eq:expbound2} by $e^{\lambda_{k+1}\alpha_k\frac{ L_s L_2}{4} }$, we get
\begin{align*}
        &\E{ e^{\lambda_{k+1}\lrp{M(x_{k+1} - x^*) + \alpha_k d_{k+1} + \alpha_k\frac{ L_s L_2}{4} }} \big| \mathcal F_k }\\
        &\qquad \le e^{ \lambda_{k+1}(1-\alpha_k \tilde{\gamma}_c)\lrp{M(\xk - x^*) + \alpha_{k-1}d_k + \alpha_{k-1}\frac{ L_s L_2}{4} } }  e^{2 \lambda_k \alpha^2_k C_2 + \lambda_{k+1} M_4}, \numberthis\label{eq:expbound3}
\end{align*}    
where,
\[ M_4 :=  \alpha_k\frac{ L_s L_2}{4}- (1-\alpha_k \tilde{\gamma}_c) \alpha_{k-1}\frac{ L_s L_2}{4}.\]
We now bound $M_4$ below. 
\begin{align*}
    M_4 &=  \frac{ L_s L_2}{4}\lrp{\alpha_k - \alpha_{k-1} + \alpha_k \alpha_{k-1}\tilde{\gamma}_c}\\
    &\le  \frac{ \tilde{\gamma}_cL_s L_2}{4}\alpha_k \alpha_{k-1} \tag{ $\because \alpha_k \le \alpha_{k-1}$}\\
    &= \alpha^2_k\frac{\alpha_{k-1}}{\alpha_k} \frac{ \tilde{\gamma}_cL_s L_2}{4}\\
    &\le 2 \alpha^2_k\frac{ \tilde{\gamma}_cL_s L_2}{4} \tag{Lemma~\ref{lem:ratioofalpha}}.
\end{align*}
Substituting the above in Equation~\eqref{eq:expbound3}, and as earlier, using $\lambda_{k+1} \le 2 \lambda_k$, we get
\begin{align*}
        &\E{ e^{\lambda_{k+1}\lrp{M(x_{k+1} - x^*) + \alpha_k d_{k+1} + \alpha_k \frac{L_s L_2}{4} }} \big| \mathcal F_k }\\
        &\qquad \le e^{ \lambda_{k+1}(1-\alpha_k \tilde{\gamma}_c)\lrp{M(\xk - x^*) + \alpha_{k-1}d_k + \frac{\alpha_{k-1} L_sL_2}{4} } }  e^{2 \lambda_k \alpha^2_k \lrp{C_2 + \frac{\tilde{\gamma}_c L_sL_2}{2} }}\\
        &\qquad \le \lrp{e^{ \lambda_{k}\lrp{ M(\xk - x^*) + \alpha_{k-1}d_k + \frac{\alpha_{k-1} L_sL_2}{4} }}}^{\frac{\lambda_{k+1}}{\lambda_k}(1-\alpha_k \tilde{\gamma}_c)}  e^{2 \lambda_k \alpha^2_k \lrp{C_2 + \frac{\tilde{\gamma}_c L_sL_2}{2} }}.
\end{align*}    

Next, from Lemma~\ref{lem:nonnegLyapunov}, we find that the exponent of $e$ is nonnegative in the above expression. Using $\lambda_{k+1}/\lambda_k \le \alpha_k/\alpha_{k+1}$ in the above expression, we get 
\begin{align*}
        &\E{ e^{\lambda_{k+1}\lrp{M(x_{k+1} - x^*) + \alpha_k d_{k+1} + \frac{\alpha_k L_sL_2 }{4 } }} \big| \mathcal F_k }\\
        &\qquad \le \lrp{e^{ \lambda_{k}\lrp{ M(\xk - x^*) + \alpha_{k-1}d_k + \frac{\alpha_{k-1} L_sL_2}{4}}}}^{\frac{\alpha_{k}}{\alpha_{k+1}}(1-\alpha_k \tilde{\gamma}_c)}  e^{2 \lambda_k \alpha^2_k \lrp{C_2 + \frac{\tilde{\gamma}_c L_sL_2}{2} }},%\numberthis\label{eq:expbound4}
\end{align*}
proving the first inequality.

Next, we prove that the exponent of the first term lies in $[0,1]$. Clearly, $(1-\alpha_k \tilde{\gamma}_c) \ge 0$ follows from choosing $h \ge 8$ large enough so that $\alpha_0 \le 1/\tilde{\gamma}_c$. Next, for proving the upper bound, consider the following for $z \in [0,1]$. 
\begin{align*}
    \frac{\alpha_k }{\alpha_{k+1}}(1-\alpha_k\tilde{\gamma}_c) 
    &= \lrp{\frac{k+h +1 }{k+h}}^z\lrp{1-\frac{\alpha \tilde{\gamma}_c}{(k+h)^z}}\\
    &\le \lrp{\frac{k+h+1}{k+h}}^ze^{-\frac{\alpha \tilde{\gamma}_c}{(k+h)^z}}\\
    &= \lrp{\lrp{1+\frac{1}{k+h}}^{k+h}}^\frac{z}{k+h} e^{-\frac{\alpha \tilde{\gamma}_c}{(k+h)^z}}\\
    &\le e^{ \frac{z}{k+h} - \frac{ \alpha \tilde{\gamma}_c}{(k+h)^z}} \tag{$\lrp{1+1/x}^x \le e$, for all $x \in \R$}.
\end{align*}

When $z=1$, on choosing $\alpha > 1/\tilde{\gamma}_c$, we have a negative exponent in the above expression, ensuring $(\alpha_{k}/\alpha_{k+1}) (1-\alpha_k \tilde{\gamma}_c) < 1$. 

When $z \in (0,1)$, choosing $h \ge \lrs{2z/(\alpha\tilde{\gamma}_c)}^{1/(1-z)}$, we have
\begin{align*}
    \frac{\alpha_k}{\alpha_{k+1}} \lrp{1 - \alpha_k \tilde{\gamma}_c} &\le e^{ \frac{z}{k+h} - \frac{ \alpha \tilde{\gamma}_c}{(k+h)^z}} \\
    &= e^{ \frac{z - (k+h)^{1-z}\alpha \tilde{\gamma}_c}{k+h}} \\
    &\le e^{ \frac{z - h^{1-z}\alpha \tilde{\gamma}_c}{k+h}}\\
    &\le 1.
\end{align*}
\end{proof}

%\subsubsection{Proof of Lemma~\ref{lem:martingale}}\label{app:proof_lem:martingale}

\subsubsection{Proof of Corollary~\ref{lem:boundonZk}}\label{app:proof_lem:boundonZk}
Letting $\tilde{\gamma}_c = \eta/2$ and recursively opening the bound in Proposition~\ref{lem:rec2}, we get for all $k\ge 0$, 
\begin{align*}
    Z_k &\le Z_0 ~ \underbrace{ \prod\limits_{i=0}^{k-1} \frac{\alpha_i}{\alpha_{i+1}} \lrp{1 - \tilde{\gamma}_c \alpha_i}}_{A_1}  + 2 D_2 ~ \underbrace{ \sum\limits_{i=0}^{k-1} \lambda_i \alpha^2_i \prod\limits_{j=i+1}^{k-1} \frac{\alpha_j}{\alpha_{j+1}}\lrp{1 - \alpha_j \tilde{\gamma}_c}}_{A_2}. \numberthis\label{eq:recZkopen}
\end{align*}

We now individually bound the two terms above. The first term $A_1$ is bounded as below. 
\begin{align*}
    A_1 
    &= \prod\limits_{i=0}^{k-1} \frac{\alpha_i}{\alpha_{i+1}} \lrp{1 - \tilde{\gamma}_c \alpha_i}\\
    &=  \frac{\alpha_0}{\alpha_{k}} \prod\limits_{i=0}^{k-1} \lrp{1-\tilde{\gamma}_c\alpha_i}\\
    &= \lrp{ \frac{k+h}{h} }^z ~ \prod\limits_{i=0}^{k-1} \lrp{1-\frac{\tilde{\gamma}_c\alpha}{(i+h)^z} }\\
    &\le  \lrp{ \frac{k+h}{h} }^z ~ e^{-\tilde{\gamma}_c\alpha~\sum\limits_{i=0}^{k-1}\frac{1}{(i+h)^z} } \tag{$\because 1+x \le e^x, \forall x\in \R$}\\
    &\le \lrp{ \frac{k+h}{h} }^z ~ e^{-\tilde{\gamma}_c\alpha~\int\limits_{0}^{k}\frac{1}{(x+h)^z} ~dx }.
\end{align*}

Separately evaluating the above integral for $z=1$ and $z \in (0,1)$, we get
\begin{align*}
    A_1 &\le 
    \begin{cases}
        \lrp{ \frac{k+h}{h} } ~ e^{-\tilde{\gamma}_c\alpha~ \ln\lrp{\frac{k+h}{h}} }, \qquad z = 1\\
        \lrp{ \frac{k+h}{h} }^z ~ e^{-\frac{\tilde{\gamma}_c\alpha}{1-z}~\lrp{(k+h)^{1-z} - h^{1-z}} }, \qquad z \in (0,1)
    \end{cases}\\
    &= \begin{cases}
        \lrp{ \frac{h}{k+h} }^{\alpha\tilde{\gamma}_c - 1}, \qquad z = 1\\
        \lrp{ \frac{k+h}{h} }^z ~ e^{-\frac{\tilde{\gamma}_c\alpha}{1-z}~\lrp{(k+h)^{1-z} - h^{1-z}} }, \qquad z \in (0,1)
    \end{cases}.
\end{align*}

Consider the following to bound $A_2$.
\begin{align*}
    A_2 &= \sum\limits_{i=0}^{k-1} \lambda_i \alpha^2_i \prod\limits_{j=i+1}^{k-1} \frac{\alpha_j}{\alpha_{j+1}}\lrp{1 - \alpha_j \tilde{\gamma}_c}\\
    &= \sum\limits_{i=0}^{k-1} \frac{\theta \alpha^2_i}{T_i \alpha_i}  \prod\limits_{j=i+1}^{k-1} \frac{\alpha_j}{\alpha_{j+1}}\lrp{1 - \alpha_j \tilde{\gamma}_c} \tag{By choice of $\lambda_i$}\\
    &\le \sum\limits_{i=0}^{k-1} \frac{\theta \alpha_i}{T_0 }  \prod\limits_{j=i+1}^{k-1} \frac{\alpha_j}{\alpha_{j+1}}\lrp{1 - \alpha_j \tilde{\gamma}_c} \\
    & \le \frac{\theta}{\|x_0-x^*\|^2_c }\sum\limits_{i=0}^{k-1} \alpha_i \prod\limits_{j=i+1}^{k-1} \frac{\alpha_j}{\alpha_{j+1}}\lrp{1 - \alpha_j \tilde{\gamma}_c}\\
    &= \frac{\theta}{\|x_0-x^*\|^2_c } \lrs{\sum\limits_{i=0}^{k-1} \alpha_i \frac{\alpha_{i+1}}{\alpha_k}\prod\limits_{j=i+1}^{k-1} \lrp{1 - \alpha_j \tilde{\gamma}_c}}\\
    &= \frac{\theta}{\|x_0-x^*\|^2_c  \alpha_k} \lrs{\sum\limits_{i=0}^{k-1} \alpha_i \alpha_{i+1} \prod\limits_{j=i+1}^{k-1} \lrp{1 - \alpha_j \tilde{\gamma}_c}}\\
    &\le \frac{\theta}{\|x_0-x^*\|^2_c  \alpha_k} \lrs{\sum\limits_{i=0}^{k-1} \alpha^2_i\prod\limits_{j=i+1}^{k-1} \lrp{1 - \alpha_j \tilde{\gamma}_c}} \tag{$\because \alpha_{i+1}\le \alpha_i $ }.
\end{align*}

The term in the brackets in the bound above is well studied in  the literature for $\alpha_k = \alpha/(k+h)^z$, with $z\in(0,1)$. Using the bound from \citet[Appendix A.3.7]{chen2024lyapunov}, we get
\begin{align*}
    A_2 &\le 
    \begin{cases}
            \frac{\theta}{\|x_0-x^*\|^2_c  \alpha_k} \lrp{ \frac{4e\alpha \alpha_k}{\tilde{\gamma}_c\alpha - 1} }, \qquad z = 1, \alpha > \frac{1}{\tilde{\gamma}_c}\\
            \frac{\theta}{\|x_0-x^*\|^2_c  \alpha_k} \lrp{ \frac{2\alpha_k}{\tilde{\gamma}_c} }, \qquad z\in (0,1), \alpha > 0, h \ge \lrp{\frac{2z}{\alpha\tilde{\gamma}_c}}^\frac{1}{1-z}
    \end{cases}\\
    &= \begin{cases}
            \frac{\theta}{\|x_0-x^*\|^2_c } \lrp{ \frac{4e\alpha }{\tilde{\gamma}_c\alpha - 1} }, \qquad z = 1, \alpha > \frac{1}{\tilde{\gamma}_c}\\
            \frac{\theta}{\|x_0-x^*\|^2_c  } \lrp{ \frac{2}{\tilde{\gamma}_c} }, \qquad z\in (0,1), \alpha > 0, h \ge \lrp{\frac{2z}{\alpha\tilde{\gamma}_c}}^\frac{1}{1-z}
    \end{cases}
\end{align*}

Substituting the bounds obtained above on $A_1$ and $A_2$ in Equation~\eqref{eq:recZkopen}, we get
\begin{align*}
    Z_k \le 
    \begin{cases}
        Z_0\lrp{ \frac{h}{k+h} }^{\alpha\tilde{\gamma}_c - 1} + \frac{8 e \alpha\theta D_2 }{\|x_0-x^*\|^2_c \lrp{\tilde{\gamma}_c\alpha - 1} }, & \text{ if } z = 1, \alpha > \frac{1}{\tilde{\gamma}_c}\\
        Z_0\lrp{ \frac{k+h}{h} }^z ~ e^{-\frac{\tilde{\gamma}_c\alpha}{1-z}~\lrp{(k+h)^{1-z} - h^{1-z}} } +  \frac{4 D_2 \theta }{\tilde{\gamma}_c\|x_0-x^*\|^2_c  }, & \text{ if } z\in (0,1), \alpha > 0, h \ge \lrp{\frac{2z}{\alpha\tilde{\gamma}_c}}^\frac{1}{1-z}.
    \end{cases}
\end{align*}

\section{Technical details from Section~\ref{sec:generalcontractive}}\label{app:generalcontractive}
In this section, we present the details missing from Section~\ref{sec:generalcontractive}.
\subsection{Constants in Theorem~\ref{th:generalnoiseconc}}\label{app:const_general_contractive}
In the setup of Section~\ref{sec:generalcontractive}, the bound $B_2$ from Assumption~\ref{assmp:boundedaddnoise} for the auxiliary algorithm depends on ${\delta'}, {T}$ (recall, $B_2(\delta',T) = 2B(\delta'/T)$), for certain constants $\delta'\in (0,1)$ and $T\in\N$. This dependence also percolates to the other terms that were constants in Theorem~\ref{th:ballargumentbound}, and also appear in Theorem~\ref{th:generalnoiseconc}. We make this dependence explicit in the notation below. For brevity, we let $B_{123}(\delta',T) := B_1 + B_2(\delta',T) + B_3$. Let 
\begin{align*}
C'_1(\delta',T) &= A_{13}(L_1+L_2) + 2(L_1+L_2) + L_1\lrp{A_1\|x^*\|_c + B_{123}(\delta',T)},\\ C'_2(\delta',T) &= (L_1 + L_2){A_1 \|x^*\|_c + B_{123}(\delta', T)} + (1+A_{13})L_2,
\end{align*}
and define the following: 
\begin{align*}
    C_1(\delta',T) &= \tfrac{u L_s}{\lcs^2\lrp{\lcs\wedge 1}^2} \lrp{\lrp{\tfrac{3}{\alpha} + \eta}(L_1+L_2) + (3L_1+1) (1+A_{13})^2 + C'_1(\delta',T)}  + \tfrac{u L_s(1+A_{13})^2}{\lcs^2},\\
    C_2(\delta',T) &=   \lrp{\tfrac{3}{\alpha} + \eta} \tfrac{L_sL_2}{4}  + \tfrac{L_s}{\lcs^2}\lrp{A_1\|x^*\|_c + B_{123}(\delta', T)}^2 \\
    &\qquad \qquad \qquad \qquad + \tfrac{ L_s }{\lcs^2 \lrp{\lcs \wedge 1}^2}\lrp{(3L_1+1) (A_{1} \|x^*\|_c + B_{123}(\delta',T))^2 + C'_2(\delta',T)},
\end{align*}
and
\begin{align*}
    D_2(\delta',T) &= C_2(\delta',T) + \tfrac{\eta L_sL_2}{4}\\
    D_3(\delta',T) &= \frac{ D_2(\delta',T) \eta \lcs^4 }{16 u L^2_s \lrp{L^2_1\|x_0 - x^*\|^2_c + L^2_2 + B^2_2(\delta', T)}}.     
\end{align*}
    Further, let
    \begin{align*}
    \theta(\delta',T) &=  \frac{\eta \lcs^4 \|x_0 - x^*\|^2_c}{32 u L^2_s \lrp{L^2_1\|x_0 - x^*\|^2_c + L^2_2 +  B^2_2(\delta', T)} }, \qquad \quad~\bar{b}_1(\delta',T) = \frac{u\alpha }{\theta(\delta'/2, T)},\\
    \bar{b}_2(\delta',T) &= u\alpha\lrp{ \frac{8\alpha e  D_2(\delta'/2, T)}{(\alpha \tilde{\gamma} - 1)\|x_0 - x^*\|^2_c} + \frac{ 2u L_s (L_1 + L_2)  }{ l\lcs^2}}, \quad \bar{b}_4(\delta',T) = \frac{\alpha^2 D_3(\delta'/2,T) u}{\theta(\delta'/2, T)}, \\
    \bar{b}_3(\delta', T) &= \frac{\eta\lcs^2 \alpha}{64 l L^2_s \alpha_0 L^2_1 L_2 \theta(\delta'/2, T) }\lrp{2\lcs^2 L_2 +  2 u \alpha_{-1} (L_1 + L_2)  L_2 L_s + \lcs^2 \alpha_{-1} l L_s L^2_1 },
    \end{align*}
and
\[\bar{a}_1(\delta', T) = \begin{cases} \lrp{\frac{1}{h-1}}^{2\alpha (A_{13}-1)} \lrp{ \|x_0 - x^*\|_c + \frac{A_1 + B_{123}(\delta'/2, T)}{A_{13}-1} }^2, &\text{if } A_{13} - 1 > 0, \\
        2\|x_0 - x^*\|^2_c + 2(A_1 \|x^*\|_c + B_{123}(\delta'/2, T))^2 \alpha^2,&\text{if } A_{13} - 1= 0,\\
        2\|x_0 - x^*\|^2_c + \frac{2(A_1 \|x^*\|_c + B_{123}(\delta'/2, T))^2}{(A_{13}- 1)^2},&\text{if } A_{13} -1 < 0.
        \end{cases}
\]
Finally, let $ B_{23}(\delta',T) = \bar{b}_2(\delta',T) + \bar{b}_3(\delta',T)$.

\subsection{Missing proofs from Section~\ref{sec:proof_th:generalnoiseconc}}
\subsubsection{Proof of Lemma~\ref{lem:biasboundaux}}\label{app:proof_lem:biasboundaux}
\begin{proof} Consider the following inequalities: 
    \begin{align*}
    \|x^* - \tilde{x}\|_c 
    &= \left\| \bar{F}(x^*) + \E{Z} - \bar{F}(\tilde{x}) - \E{Z {\bf 1}_{\{ \| Z \|_c \le B(\delta'/T)\}}} \right\|_c \\ %\tag{$\E{Z} = 0$} \\
    &\le \| \bar{F}(x^*)  - \bar{F}(\tilde{x})\|_c  +  \left\| \E{Z}- \E{Z {\bf 1}_{\{\|Z \|_c \le B(\delta'/T)\}}} \right\|_c \tag{Triangle ineq.}\\
    %&= \| \bar{F}(x^*)  - \bar{F}(\tilde{x})\|_c  +  \|\E{Z {\bf 1}_{\|Z \|_c > B(\delta'/T)}} \|_c \\
    &\le \gamma_c \|x^* - \tilde{x}\|_c + \left\| \E{Z {\bf 1}_{\{\|Z\|_c > B(\delta'/T)\}}} \right\|_c.  \tag{Assumption~\ref{assmp:contraction}}
\end{align*}
The bound in the lemma now follows from rearranging the above inequality.
\end{proof}

\subsubsection{Proof of Lemma~\ref{lem:probboundtrue_aux_bad}}\label{app:proof_lem:probboundtrue_aux_bad}
\begin{proof}
Consider the following inequalities: 
\begin{align*}
    \mathbb{P}&\lrp{ \exists k \in [T]: \|x_k - x^*\|^2_c \ge 2 f(k, \delta, \delta') + 2g^2(\tfrac{\delta'}{T}) }\\
    &\le \mathbb{P}\lrp{ \exists k\in [T]: \|x_k - x^*\|^2_c \ge 2 f(k, \delta, \delta') + 2g^2(\tfrac{\delta'}{T}); \mathcal B_T(\delta') } + \delta' \tag{$A \subset (A\cap B) \cup B^c$; Eq.~\eqref{eq:probbcomp}} \\
    &\le \mathbb{P}\lrp{ \exists k\in [T]: 2\|x_k - \tilde{x} \|^2_c + 2\| \tilde{x} - x^*\|^2_c \ge 2 f(k, \delta, \delta') + 2g^2(\tfrac{\delta'}{T}); \mathcal B_T(\delta') } + \delta' \tag{Triangle ineq.} \\ 
    &\le \mathbb{P}\lrp{ \exists k\in [T]: \|x_k - \tilde{x} \|^2_c \ge f(k, \delta, \delta'); \mathcal B_T(\delta') } + \delta' \tag{Eq.~\eqref{eq:bias};~\eqref{eq:biasofaux}} \\ 
    %&= \mathbb{P}\lrp{ \exists k\in [T]: \|\tilde{x}_k - \tilde{x} \|^2_c \ge f(k, \delta, \delta'); \mathcal B_T(\delta') } + \delta'\tag{$\xk = \tilde{x}_k$ on $\mathcal B_T(\delta')$} \\
    &= \mathbb{P}\lrp{ \exists k\in [T]: \|\tilde{x}_k - \tilde{x} \|^2_c \ge f(k, \delta, \delta') } + \delta'\tag{$\xk = \tilde{x}_k$ on $\mathcal B_T(\delta')$}.
\end{align*}
This completes the proof.
\end{proof}

\subsubsection{Proof of  Lemma~\ref{lem:assmpcheck}}\label{app:proof_lem_assmpcheck}

\begin{proof}
Since the operator $F$ and $\bar{F}$ satisfy Assumption~\ref{assmp:boundedmultnoise}, we have that $H$, $\bar{H}$ also satisfy Assumption~\ref{assmp:boundedmultnoise}. Next, the additive noise $Z {\bf 1}_{\{\|Z\|_c \le B(\delta'/T)\}} - \mathbb{E}\big[{Z{\bf 1}_{\{\|Z\|_c \le B(\delta'/T)\}}}\big]$ has $0$ mean. Moreover, 
\begin{align*}
    \left\|Z {\bf 1}_{\left\{\|Z\|_c \le B\lrp{\frac{\delta'}{T}}\right\}} - \E{Z {\bf 1}_{\left\{\|Z\|_c \le B\lrp{\frac{\delta'}{T}}\right\}}}\right\|_c
    %&\le \|Z\|_c {\bf 1}_{\|Z\|_c \le B\lrp{\frac{\delta'}{T}}} + \left\|\E{Z{\bf 1}_{\|Z\|_c \le B\lrp{\frac{\delta'}{T}}}}\right\|_c 
    \le 2B\lrp{\tfrac{\delta'}{T}}. 
\end{align*}
Thus, the additive noise for the auxiliary algorithm satisfies the Assumption~\ref{assmp:boundedaddnoise} with $B_2 = 2 B(\delta'/T)$. In the following, we denote $B_2$ by $B_2(\delta',T)$ to explicitly highlight its dependence on these parameters. 

Since $\bar{H}(\cdot) = \bar{F}(\cdot) + \text{ a constant}$, $\bar{H}$ is also a $\gamma_c$-contraction. Finally, it only remains to show that the operator $H(\cdot,\cdot)$ is uniformly Lipschitz in the first argument,  with parameter  $L_F$. This follows immediately, from the following. Consider $x_1, x_2 \in \R^d$ and $s\in \mathcal S$. Then 
\begin{align*}
    \|H(x_1, s) - H(x_2, s)\|_c &= \|F(x_1, s) - F(x_2, s)\|_c \le L_F\|x_1 - x_2\|_c, 
\end{align*}
completing the proof.
\end{proof}

\subsection{Bounding the Truncation Bias}\label{app:boundbias}
In this section, we upper bound the bias term $g(\cdot)$, defined in Equation~\eqref{eq:bias}. Our bound depends on the tail of the random variable $W$ introduced in Assumption~\ref{assmp:addnoise}.  Recall $B(\cdot)$ and $g(\cdot)$ from equations~\eqref{eq:quantile} and~\eqref{eq:bias}, respectively. Also recall  $\mathcal B_T(\cdot)$ from Equation~\eqref{eq:BT}. In the following sections, we instantiate an upper bound on $g(\cdot)$ for different choices of the random variable $W$. 

\begin{comment}
\subsubsection{Sub-Gaussian}\label{app:subgaussian}
Suppose $W$ is a non-negative, $\sigma^2$ sub-Gaussian random variable, i.e.,
\[ \mathbb{P}\lrs{W \ge x} \le c e^{-\frac{x^2}{2\sigma^2}}, \quad \text{ for } x \ge 0. \]
Then, $B(\frac{\delta'}{T}) = \mathcal{O}(\sigma\sqrt{2 \log\frac{T}{\delta'}})$, and on $\mathcal B_T(\delta')$, for $k\in [T]$, $\|Z_k\|_c$ are bounded between $[0, B(\frac{\delta'}{T})]$. Then, \begin{align*}
    \E{ W{\bf 1}_{W \ge B(\delta'/T)} } 
    &= \int\limits_{B(\delta'/T)}^\infty \Prob{ W \ge x } dx \le  \sum\limits_{i=0}^\infty \Prob{W \ge K_i} (K_{i+1}  - K_i),
\end{align*}    
where $K_i := B(\frac{\delta'}{2^i T})$. Observe that $K_i$ are increasing with  $K_0 = B(\frac{\delta'}{T})$. By choice of $B(\frac{\delta'}{T})$, we have $\Prob{ W \ge K_i} \le \frac{\delta'}{2^i T}$, and 
\[\abs{K_{i+1} - K_i} = \mathcal{O}\lrp{\sigma\sqrt{2\log\frac{2^{i+1}T}{\delta'}}-\sigma\sqrt{ 2 \log\frac{ 2^{i}T}{\delta'}}}\overset{(a)}{\le} \mathcal{O}\lrp{\frac{\sigma \log(2)}{\sqrt{\log\frac{2^i T}{\delta'}}}} \le \mathcal{O}\lrp{\sigma \log(2)} ,\]
%\[\abs{K_{i+1} - K_i} = \sigma\sqrt{2\log\frac{2^{i+1}T}{\delta'}} - \sigma \sqrt{ 2 \log\frac{ 2^{i}T}{\delta'}} = \sigma \sqrt{2\log\frac{T}{\delta'}}\lrp{\sqrt{i+2} - \sqrt{i+1}} \le \sigma \sqrt{2\log\frac{2}{\delta'}} ,\]
where the last inequality follows from concavity of $\sqrt{\cdot}$. This gives
\[ \E{ W{\bf 1}_{W \ge t(\delta')} } \le \mathcal{O}\lrp{\sigma \log(2) \sum_i \frac{\delta' }{2^iT}}  \le {\frac{ c'  \sigma \delta' \log(2) }{T}}, \]
for some constant $c'$.
\end{comment}

\begin{lemma}\label{lem:subweibull}
    Suppose $W$ is a Sub-Weibull random variable such that 
    \[ \Prob{W \ge x} \le p e^{-qx^\frac{1}{\theta}}, \quad \forall x\ge 0, \text{ and  some }\theta > \frac{1}{2}, p, q> 0. \]
    Then for $\gamma \in (0,1)$,  $B(\gamma) = \lrp{ \tfrac{1}{q} \log\left(\tfrac{p}{\gamma}\right)}^\theta$, and for a constant $c > 0$, 
    \[  \E{ W{\bf 1}_{\{W \ge B(\gamma)\}} } \le \frac{\gamma \theta c}{q} \lrp{\frac{1}{q}\log\left(\frac{p}{\gamma}\right)}^{\theta - 1}. \]
\end{lemma}
\begin{proof}
Clearly, $B(\gamma) = \lrp{ \tfrac{1}{q}\log\left(\tfrac{p}{\gamma}\right) }^\theta$. This follows from the definition of $B(\cdot)$ and the probability density function of $W$. Next, 
\begin{align*}
    \E{ W{\bf 1}_{\{W \ge B(\gamma)\}} } &\le \int\limits_{B(\gamma)}^\infty \Prob{W \ge x} dx \le  \sum\limits_{i=0}^\infty \Prob{W \ge K_i} (K_{i+1}  - K_i),
\end{align*}    
where $K_i := B(\gamma/2^i)$. Hence, $K_0 = B(\gamma)$. Further, by choice of $B(\gamma)$, we have $\Prob{W \ge K_i} \le \gamma/2^i$, and 
\[\abs{K_{i+1} - K_i} = \lrp{\frac{1}{q}\log\left(\frac{p 2^{i+1}}{\gamma}\right)}^\theta - \lrp{\frac{1}{q}\log\left(\frac{p 2^{i}}{\gamma}\right)}^\theta.\]
Below, we analyze the two cases $\theta \le 1$ and $\theta > 1$ separately.\\

\noindent{Case 1: $\theta \in (1/2,1]$}. Let $h(a) := a^\theta$, for $a > 0$. For $1/2 \le \theta \le 1$, $h$ is a concave function. Thus, for any two points $a > 0$ and $b > 0$, $ h(a) \le h(b) + h'(b)(a-b)$, where $h'(b) = \theta b^{\theta - 1}$. Using this $\abs{K_{i+1} - K_i}$ is at most 
    \[  \theta\lrp{\frac{1}{q}\log\left(\frac{p 2^{i}}{\gamma}\right)}^{\theta - 1}\lrp{ \frac{1}{q}\log\left(\frac{p 2^{i+1}}{\gamma}\right) - \frac{1}{q}\log\left(\frac{p 2^{i}}{\gamma}\right) } = \frac{\theta}{q}\lrp{\frac{1}{q}\log\left(\frac{p 2^{i}}{\gamma}\right)}^{\theta - 1}\log(2),\]
    giving 
    %\[ \E{ W{\bf 1}_{W \ge B(\delta'/T)} } \le \frac{\theta}{q} \log(2) \sum\limits_{i=0}^\infty \frac{\delta'}{2^i T}\lrp{\frac{1}{q}\log \frac{p 2^{i}}{\delta'}}^{\theta - 1} \le \frac{\theta}{q} \log(2) \sum\limits_{i=0}^\infty \frac{\delta'}{2^i T} \le \frac{\theta \delta' c \log(2)}{qT}, \]
    \begin{align*} 
        \E{ W{\bf 1}_{\{W \ge B(\gamma)\}} } &\le \frac{\theta}{q} \log(2) \sum\limits_{i=0}^\infty \frac{\gamma}{2^i}\lrp{\frac{1}{q}\log\left(\frac{p 2^{i}}{\gamma}\right)}^{\theta - 1} \\
        &\overset{(a)}{\le} \frac{\gamma \theta \log(2)}{q} \lrp{\frac{1}{q}\log\left(\frac{p}{\gamma}\right)}^{\theta - 1}\sum\limits_{i=0}^\infty \frac{1}{2^i}\\
        &\le \frac{\gamma \theta c}{q} \lrp{\frac{1}{q}\log\left(\frac{p}{\gamma}\right)}^{\theta - 1}, 
    \end{align*}
    for a constant $c > 0$, where $(a)$ follows since $\theta - 1 \le 0$.\\

\noindent{Case 2: $\theta > 1$.} In this case,
    \[\abs{K_{i+1} - K_i} = \lrp{\frac{1}{q}\log\left(\frac{p 2^{i+1}}{\gamma}\right)}^\theta - \lrp{\frac{1}{q}\log\left(\frac{p 2^{i}}{\gamma}\right)}^\theta \le \frac{\theta \log(2)}{q}\lrp{\frac{1}{q}\log\left(\frac{p 2^{i+1}}{\gamma}\right)}^{\theta-1},\]
    where we used the convexity of the map $x \rightarrow x^\theta$.
    Thus, 
    \begin{align*} 
    \E{ W{\bf 1}_{\{W \ge B(\gamma)\}} } &\le \sum\limits_{i=0}^\infty \frac{\gamma \theta \log(2)}{2^i q}\lrp{\frac{1}{q}\log\left(\frac{p 2^{i+1}}{\gamma}\right)}^{\theta-1} \\
    &\le \frac{\gamma\theta\log(2)}{q}\lrp{\frac{1}{q}\log\left(\frac{p}{\gamma}\right)}^{\theta-1}  \sum\limits_{i=0}^\infty \frac{(i+1)^\theta}{2^i}\\
    &\le  \frac{c \gamma \theta}{q}\lrp{\frac{1}{q}\log\left(\frac{p}{\gamma}\right)}^{\theta-1},
    \end{align*}
    for a constant $c > 0$. 
\end{proof}

\begin{lemma}\label{lem:subpareto}
Suppose $W$ is a heavy-tailed random variable such that 
\[ \Prob{W \ge x} \le \frac{p}{x^\theta}, \quad \forall x\ge 0, \text{ and  some }\theta \ge 2, p> 0. \]
Then for $\gamma \in (0,1)$,  $B(\gamma) = \lrp{ p/\gamma }^{1/\theta}$, and a constant $c > 0$, 
\[ \E{ W{\bf 1}_{\{W \ge B(\gamma)\}} } \le c\gamma^\frac{\theta - 1}{\theta}\frac{p^\frac{1}{\theta}}{\theta}. \]
\end{lemma}
\begin{proof}
The expression for $B(\gamma)$ follows from the definition of $B(\cdot)$ and the probability density function of $W$. Next,
\begin{align*}
    \E{ W{\bf 1}_{\{W \ge B(\gamma)\}} } 
    &= \int\limits_{B(\gamma)}^\infty \Prob{ W \ge x } dx \le  \sum\limits_{i=0}^\infty \Prob{W \ge K_i} (K_{i+1}  - K_i),
\end{align*}    
where $K_i := B(\gamma/2^i)$. Hence, $K_0 = B(\gamma)$. By choice of $B(\cdot)$, we have $\Prob{W \ge K_i} \le \gamma/2^i$, and 
\[\abs{K_{i+1} - K_i} = \lrp{ \frac{2^{i+1}p}{\gamma} }^\frac{1}{\theta} - \lrp{\frac{2^{i}p}{\gamma}}^\frac{1}{\theta}.\]
Let $h(a) := a^{1/\theta}$, for $a > 0$. Since $\theta \ge 2$, $h$ is a concave function. Thus, for any two points $a > 0$ and $b > 0$, $h(a) \le h(b) + h'(b)(a-b)$, where $h'(b) = b^{1/\theta - 1}/\theta$. Using this, $\abs{K_{i+1} - K_i}$ is at most 
\[  \frac{1}{\theta} \lrp{ \frac{2^{i}p}{\gamma} }^{\frac{1}{\theta} - 1}\lrp{  \frac{2^{i+1}p}{\gamma} -  \frac{2^{i}p}{\gamma} } = \frac{1}{\theta} \lrp{ \frac{2^{i}p}{\gamma} }^{\frac{1}{\theta} }, \]
giving 
\[ \E{ W{\bf 1}_{\{W \ge B(\gamma)\}} } \le \sum\limits_{i=0}^\infty \frac{\gamma}{2^i} \frac{1}{\theta} \lrp{ \frac{2^{i}p}{\gamma} }^{\frac{1}{\theta} } = \gamma^\frac{\theta - 1}{\theta}\frac{p^\frac{1}{\theta}}{\theta} \sum\limits_{i=0}^\infty \frac{1}{2^{i \frac{\theta - 1}{\theta}}} \le c\gamma^\frac{\theta - 1}{\theta}\frac{p^\frac{1}{\theta}}{\theta}, \]
for a constant $c > 0$.
\end{proof}

\section{Certain Technical Lemmas}
\begin{lemma}\label{lem:ratioofalpha}
    For $\alpha > 0$, $h \ge 8$, $z\in (0,1]$, and for all $k\ge -1$,  $\alpha_k = \alpha/(k+h)^z$ or $\alpha_k = \alpha$, we have
    \[ \frac{\alpha_k}{\alpha_{k+1}}\le 2. \]
\end{lemma}
\begin{proof}
    This trivially follows when $\alpha_k = \alpha$ for all $k\ge 1$. Now consider the other case, in which $\alpha_k = \alpha/(k+h)^z$, for some $z\in (0,1]$, and consider the following inequalities: 
    \begin{align*}
        \frac{\alpha_k}{\alpha_{k+1}} &= \lrp{\frac{k+h+1}{k+h}}^z = \lrp{1+ \frac{1}{k+h}}^z \le 2^z \le 2,
    \end{align*}
    proving the desired bound.
\end{proof}

\begin{lemma}\label{lem:diffofalpha}
    For any $\alpha > 0$ and $h \ge 8$, for all $k\ge 0$ and $\alpha_k = \alpha/(k+h)^z$ for $z\in (0,1]$, or $\alpha_k = \alpha
    $, we have
    \[ \alpha_{k-1} - \alpha_k \le \frac{2\alpha^2_k}{\alpha}. \]
\end{lemma}
\begin{proof}
    The inequality trivially holds for $\alpha_k = \alpha$, for all $k\ge 0$. To see the inequality for the other case, that is, when $\alpha_k = \alpha/(k+h)^z$, consider the following inequalities:
    \begin{align*}
        \alpha_{k-1} - \alpha_k &= \alpha_k \lrp{ \frac{\alpha_{k-1}}{\alpha_k} - 1 }\\
        &= \alpha_k \lrs{ \lrp{\frac{k+h}{k+h-1}}^z - 1 }\\
        &= \frac{\alpha^2_k}{\alpha_k} \lrs{ \lrp{\frac{k+h}{k+h-1}}^z - 1 }\\
        &= \frac{\alpha^2_k}{\alpha} (k+h)^z \lrs{ \lrp{1 + \frac{1}{k+h-1}}^z - 1 }\\
        &\le \frac{\alpha^2_k}{\alpha} \frac{k+h}{k+h-1} \\
        &= \frac{\alpha^2_k}{\alpha} \lrp{1 + \frac{1}{k+h-1}} \\
        &\le \frac{2\alpha^2_k}{\alpha}.
    \end{align*}
\end{proof}
\begin{lemma}\label{lem:prodofalpha}
    For any $\alpha > 0$, and either $\alpha_k = \alpha$ for all $k\ge 0$, or for $h \ge 8$, and $z\in(0,1]$,  $\alpha_k = \alpha/(k+h)^z$ for $k \ge 0$, we have \[ \alpha_k \alpha_{k-1} \le 2\alpha^2_k .\]
\end{lemma}
\begin{proof}
    This follows from the following: \( \alpha_k \alpha_{k-1}  = \alpha^2_k \alpha_{k-1}/\alpha_k \le 2 \alpha^2_k.\)
\end{proof}
\end{document}